%% file: self_simulation_main.tex
\documentclass[11pt,letterpaper]{amsart}
\makeatletter
\renewcommand\normalsize{%
	\@setfontsize\normalsize{11.7}{14pt plus .3pt minus .3pt}%
	\abovedisplayskip 10\p@ \@plus4\p@ \@minus4\p@
	\abovedisplayshortskip 6\p@ \@plus2\p@
	\belowdisplayshortskip 6\p@ \@plus2\p@
	\belowdisplayskip \abovedisplayskip}
\renewcommand\small{%
	\@setfontsize\small{9.5}{12\p@ plus .2\p@ minus .2\p@}%
	\abovedisplayskip 8.5\p@ \@plus4\p@ \@minus1\p@
	\belowdisplayskip \abovedisplayskip
	\abovedisplayshortskip \abovedisplayskip
	\belowdisplayshortskip \abovedisplayskip}
\renewcommand\footnotesize{%
	\@setfontsize\footnotesize{8.5}{9.25\p@ plus .1pt minus .1pt}%
	\abovedisplayskip 6\p@ \@plus4\p@ \@minus1\p@
	\belowdisplayskip \abovedisplayskip
	\abovedisplayshortskip \abovedisplayskip
	\belowdisplayshortskip \abovedisplayskip}
\ifdefined\pdfpagewidth
\else
\fi
\calclayout
\makeatother

\usepackage{amssymb}
\usepackage[]{epsf,epsfig,amsmath,amssymb,amsfonts,latexsym}
\usepackage{enumerate}

\usepackage{tikz}
\usetikzlibrary{arrows,decorations.markings, calc, matrix, shapes,automata,fit,patterns}
\usepackage{color}
\usepackage{cancel}
\usepackage{amsmath}
\usepackage[normalem]{ulem}
\usepackage{amsfonts}
\usepackage{bbm}
\usepackage{amsthm}
\usepackage{wasysym}

\usepackage[mathscr]{eucal}
\usepackage[active]{srcltx}
\usepackage{fancyhdr}
\usepackage{verbatim}
\usepackage{fullpage}
\usepackage{hyperref}
\usepackage{mathtools}
\hypersetup{
	colorlinks = true, %
	urlcolor = blue, %
	linkcolor = red, %
	citecolor = blue %
}
\allowdisplaybreaks

\def \NN{\mathbb N}
\def \ZZ{\mathbb Z}
\def \ZZd{{\mathbb Z}^d}

\def \RR{\mathbb R}

\def \ag{A}
\def \FF{\mathcal F}

\newcommand{\Aut}{\operatorname{Aut}}
\newcommand{\Out}{\operatorname{Out}}
\newcommand{\htop}{h_{\mathrm{top}}}

\newcommand{\agb}{\{\symb{0},\symb{1}\}}
\newcommand{\symb}[1]{\mathtt{#1}}

\newcommand{\define}[1]{\textbf{#1}}
	
\newcommand{\isdef}{=}

\newcommand{\paradox}{\mathbf{P}}
\newcommand{\dont}{\textvisiblespace}

\newtheorem{theorem}{Theorem}[section]
\newtheorem{lemma}[theorem]{Lemma}
\newtheorem{claim}[theorem]{Claim}
\newtheorem{proposition}[theorem]{Proposition}
\newtheorem{corollary}[theorem]{Corollary}
\newtheorem{definition}[theorem]{Definition}
\theoremstyle{remark}
\newtheorem{remark}[theorem]{Remark}
\newtheorem{example}[theorem]{Example}
\newtheorem{question}[theorem]{Question}

\newcommand\xqed[1]{%
	\leavevmode\unskip\penalty9999 \hbox{}\nobreak\hfill
	\quad\hbox{#1}}
\newcommand\qee{\xqed{$\fullmoon$}}

\title{Self-simulable groups}

\author{
	Sebasti\'an Barbieri, Mathieu Sablik and Ville Salo
}

\newcommand{\Addresses}{{
		\bigskip

		\hskip-\parindent   S.~Barbieri, \textsc{DMCC, Universidad de Santiago de Chile.}\par\nopagebreak
		\textit{E-mail address}: \texttt{sebastian.barbieri@usach.cl}
		
		\medskip
		
		\hskip-\parindent   M.~Sablik, \textsc{Universit\'e Paul Sabatier.}\par\nopagebreak
		\textit{E-mail address}: \texttt{mathieu.sablik@math.univ-toulouse.fr}
		
		\medskip
		
		\hskip-\parindent   V.~Salo, \textsc{University of Turku.}\par\nopagebreak
		\textit{E-mail address}: \texttt{vosalo@utu.fi}
}}

\date{}

\begin{document}
	\large

\begin{abstract}
	 We say that a finitely generated group $\Gamma$ is self-simulable if every effectively closed action of $\Gamma$ on a closed subset of $\{\symb{0},\symb{1}\}^{\NN}$ is the topological factor of a $\Gamma$-subshift of finite type. We show that self-simulable groups exist, that any direct product of non-amenable finitely generated groups is self-simulable, that under technical conditions self-simulability is inherited from subgroups, and that the subclass of self-simulable groups is stable under commensurability and quasi-isometries of finitely presented groups.
	 
	 Some notable examples of self-simulable groups obtained are the direct product $F_k \times F_k$ of two free groups of rank $k \geq 2$, non-amenable finitely generated branch groups, the simple groups of Burger and Mozes, Thompson's $V$, the groups $\operatorname{GL}_n(\ZZ)$, $\operatorname{SL}_n(\ZZ)$, $\Aut(F_n)$ and $\Out(F_n)$ for $n \geq 5$; The braid groups $B_m$ for $m \geq 7$, and certain classes of RAAGs. We also show that Thompson's $F$ is self-simulable if and only if $F$ is non-amenable, thus giving a computability characterization of this well-known open problem. We also exhibit a few applications of self-simulability on the dynamics of these groups, notably, that every self-simulable group with decidable word problem admits a nonempty strongly aperiodic subshift of finite type. 
	 
	\medskip

	\noindent
	\emph{Keywords:} group actions, symbolic dynamics, effectively closed actions, non-amenable groups, subshifts of finite type, Thompson's groups.

	\smallskip
	
	\noindent
	\emph{MSC2010:} \textit{Primary:}
	37B10, %
	\textit{Secondary:}
	37B05,  %
	20F10.  %
	
\end{abstract}

\maketitle

\section{Introduction}
	A result which ensures that every ``complicated'' computable object from a class can be embedded into a ``simple'' computable object from another class is called a simulation theorem. A beautiful illustration is the result by Higman~\cite{Highman1961}, which states that every recursively presented group can be obtained as a subgroup of a finitely presented group. 

In the category of dynamical systems, an important simulation theorem due to Hochman~\cite[Theorem 1.6]{Hochman2009b} states that any effectively closed action of $\mathbb{Z}$ on a subset of $\{\symb{0},\symb{1}\}^{\mathbb{N}}$ (roughly speaking, an action that can be approximated up to an arbitrary precision by a Turing machine) can be realized as a topological factor of the $(\ZZ \times \{(0,0)\})$-subaction of a $\mathbb{Z}^3$-subshift of finite type.

Simulation theorems can shed light on complicated aspects of the ``simple'' objects. In the setting of symbolic dynamics, the term ``swamp of undecidability'' was coined by Lind~\cite{Lind2004} to describe the lack of simple algebraic procedures to gain information about the properties of multidimensional subshifts of finite type. It was only after the result by Hochman and further developments~\cite{AubrunSablik2010,DurandRomashchenkoShen2010,Miller2012,simpson_medvedev_2012} that the origin of this swamp of undecidability was truly understood. Namely, the existence of simple algebraic descriptions can be proscribed due to the fact that every effectively closed action of $\ZZ$ can be obtained as a topological factor of a subaction of a multidimensional subshift of finite type; and (expansive) effectively closed actions of $\ZZ$ can have arbitrary effectively closed Medvedev degrees~\cite{Miller2012}. This intuitively means that the ``simplest'' configurations in a multidimensional subshift of finite type may already be quite complex in terms of computability.

Simulation results can be used as black boxes to obtain interesting results. A wonderful illustration of this method is that the famous result proven independently by Novikov and Boone~\cite{novikov1955,Boone1958}, which states that there exist finitely presented groups with undecidable word problem, can be obtained as a relatively simple corollary from Higman's theorem. In fact, Higman's theorem can be used to prove something much stronger. Namely, the existence of a universal finitely presented group, that is, one whose finitely generated subgroups are up to isomorphism precisely all finitely generated recursively presented groups (for a proof, see~\cite[Theorem 7.6]{LyndonSchupp1977}). 

In recent work by two of the authors, two generalizations of Hochman's theorem have been developed for actions of arbitrary finitely generated groups. The first one~\cite[Theorem 4.8]{BS2018} states that for any finitely generated group $\Gamma$ and every semidirect product $ \ZZ^d\rtimes_{\varphi} \Gamma$ with $d \geq 2$ and $\varphi$ a homomorphism from $\Gamma$ to $\operatorname{GL}_d(\ZZ)$, every effectively closed action of $\Gamma$ on a subset of $\{\symb{0},\symb{1}\}^{\NN}$ can be realized as a topological factor of the $(\{\vec{0}\} \times \Gamma )$-subaction of a $(\ZZ^d\rtimes_{\varphi} \Gamma)$ subshift of finite type. The second one~\cite[Theorem 3.1]{Barbieri_2019_DA} shows that for any triple of infinite and finitely generated groups $\Gamma,H,V$, every effectively closed action of $\Gamma$ on a subset of $\{\symb{0},\symb{1}\}^{\NN}$ can be realized as a topological factor of the $(\Gamma \times \{1_{H},1_{V}\})$-subaction of a $(\Gamma \times H \times V)$-subshift of finite type.

In this work we shall study groups where there is no need to take proper subactions to recover the effectively closed dynamics. That is, groups whose computable zero-dimensional dynamics are completely described as topological factors of subshifts of finite type.

\begin{definition}\label{def:selfsimulable}
	We say that a group $\Gamma$ is \define{self-simulable} if every effectively closed action $\Gamma \curvearrowright X \subseteq \{\symb{0},\symb{1}\}^{\NN}$ is a topological factor of a $\Gamma$-subshift of finite type.
\end{definition}

A more precise, albeit longer, name for this class of groups would be ``group with self-simulable effective zero-dimensional dynamics''. The reader should bear in mind that the name ``self-simulable'' is always referring to the zero-dimensional dynamics of the group. The name should be then considered as an analogy to Hochman's simulation result, with the caveat that both groups involved coincide.

Self-simulable groups have striking dynamical consequences. For instance, in Corollary~\ref{cor:stronglyaperiodic} we show that every self-simulable group with decidable word problem admits a nonempty strongly aperiodic subshift of finite type, that is, one on which the group acts freely. Also, we show in Corollary~\ref{cor:sunnysideupsofic} that the natural action of a self-simulable group over its one point compactification is topologically conjugate to a sofic subshift for every self-simulable group with decidable word problem.

Both in Hochman's original result and in the two subsequent generalizations, the group which simulates the subaction is ``large''. In the case of Hochman's result, $\ZZ^3$ cannot replaced by $\ZZ^2$: an unpublished example by Jeandel, the horizontal shift action on the mirror shift, (see Remark~\ref{rem:mirrorshift} or the paragraph after~\cite[Theorem 2.14]{ABS2017} for the definition) shows that there exist effectively closed $\ZZ$-actions which are not a topological factor of a $(\ZZ \times \{0\})$-subaction of any $\ZZ^2$-subshift of finite type. This suggests that in order to obtain a dynamical simulation theorem one should require a strictly larger group. Indeed, it is not hard to see that under mild computability assumptions, there are effectively closed actions with infinite topological entropy, and thus no amenable group satisfying the assumptions can be self-simulable because topological entropy decreases under factor maps (see Proposition~\ref{prop:SelfSimuImpNonAmenable}).

In the case where the effectively closed action of $\ZZ$ is also expansive, i.e.\ topologically conjugate to a subshift, the role of $\ZZ^3$ in Hochman's result can be replaced by $\ZZ^2$. Specifically, a result proven independently by Aubrun and Sablik~\cite{AubrunSablik2010}, and Durand, Romashchenko and Shen~\cite{DurandRomashchenkoShen2010} shows that every effectively closed $\mathbb{Z}$-subshift is topologically conjugate to the $(\ZZ \times \{0\})$-subaction of a sofic $\ZZ^2$-subshift. However, even if we restrict the scope of a simulation theorem to cover only expansive actions, it can be shown, under mild computability assumptions, that no amenable group nor any group with infinitely many ends is self-simulable~\cite{ABS2017} (See also Propositions~\ref{prop:amenable} and~\ref{prop:infends}). Moreover, there exist one-ended non-amenable groups, such as $F_k \times \ZZ$ for $k \geq 2$ (see Proposition~\ref{prop:villexample}) which are not self-simulable. We shall provide a series of obstructions to self-simulability in Section~\ref{sec:counterexamples}.

The previous obstructions make it plausible to believe that self-simulable groups may not exist at all. The main result of this article is that there is a large family of non-amenable groups which are self-simulable. 

\begin{theorem}\label{thm:selfsimulation}
	The direct product $\Gamma = \Gamma_1 \times \Gamma_2$ of any pair of finitely generated non-amenable groups $\Gamma_1$ and $\Gamma_2$ is self-simulable.
\end{theorem}

In the case where the direct product $\Gamma$ is a recursively presented group, a corollary of Theorem~\ref{thm:selfsimulation} is that the class of sofic $\Gamma$-subshifts coincides with the class of effectively closed $\Gamma$-subshifts. This answers a question of two of the authors~\cite[Section 3.4]{ABS2017}.

Seward proved in~\cite{Seward2014} that every action of a countable non-amenable group is the topological factor of a subshift. Our result shows that there is a subclass of  non-amenable groups for which that subshift can always be chosen of finite type as long as the action is effectively closed.

The proof of Theorem~\ref{thm:selfsimulation} is given in Section~\ref{sec:selfsimulable} and relies on two main ingredients which can be summarized as follows

\begin{itemize}
	\item The first ingredient comes from a well-known characterization of non-amenability: a group $\Gamma$ is
	non-amenable if and only if it admits a $2$-to-$1$ surjective map $\varphi \colon \Gamma \to \Gamma$ for which $g^{-1}\varphi(g)$ is in a fixed finite set $K$ for every $g \in \Gamma$. For a fixed set $K$, we encode the space of all such maps in a $\Gamma$-subshift of finite type that we call the \define{paradoxical subshift}, and use its configurations to encode families of disjoint one-sided infinite paths indexed by the elements of the group in a local manner. By taking the direct product of two non-amenable groups, we are able to encode families of disjoint $\NN^2$-grids indexed by the elements of the group.
	\item The second ingredient is the well-known fact that the space-time diagrams of Turing machines can be encoded in $\NN^2$ using \define{Wang tiles}. In particular, given an effectively closed set $X \subseteq \{0,1\}^{\NN}$, we can define a finite set of Wang tiles in such a way that the associated valid tilings of $\NN^2$ correspond to elements of $x$.
\end{itemize}

The rest of the proof essentially consists in showing that all of these $\NN^2$-grids can communicate their information locally. Contrary to the other dynamical simulation results mentioned above, this construction does not rely on any sort of hierarchical or self-similar structure. The construction of the paradoxical subshift is given in Section~\ref{sec:paradox}, where we also provide an effective version of Seward's result in the case where the action is effectively closed (Theorem~\ref{thm:seward_new}).

Theorem~\ref{thm:selfsimulation} provides the first examples of self-simulable groups, however, the result might strike as fairly artificial as every example is a direct product. A natural question is whether we can obtain examples which are not direct products through stability properties of the class of self-simulable groups. The next two sections are devoted to this question.

In Section~\ref{sec:stability_properties} we provide a general criterion for a group $\Gamma$ which contains a self-simulable subgroup $\Delta$ to be self-simulable itself. More precisely, we say that $\Delta$ is a \define{mediated} subgroup, if there exists a finite generating set $T$ of $\Gamma$ such that for every $t \in T$ the subgroup $\Delta \cap t\Delta t^{-1}$ is nonamenable. This condition can also be described in terms of paradoxical path-covers which is what we use to prove our results. With this notion we can show the following result.

\begin{theorem}\label{thm:mediaintroduction}
	Let $\Gamma$ be a finitely generated and recursively presented group which contains a mediated self-simulable subgroup $\Delta$. Then $\Gamma$ is self-simulable.
\end{theorem}

A consequence of this result is that any finitely generated and recursively presented group which admits a normal self-simulable subgroup is self-simulable (Corollary~\ref{cor:normalSS}).

Next we study rigid classes of self-simulable groups. Namely, we show that self-simulability is a commensurability invariant, and that it is preserved under quasi-isometries of finitely presented groups.

\begin{theorem}\label{thm:commensurable}
	Let $\Gamma_1,\Gamma_2$ be two finitely generated groups which are commensurable. Then $\Gamma_1$ is self-simulable if and only if $\Gamma_2$ is self-simulable.
\end{theorem}

\begin{theorem}\label{thm:quasiisometry}
	Let $\Gamma_1,\Gamma_2$ be two finitely presented groups which are quasi-isometric. Then $\Gamma_1$ is self-simulable if and only if $\Gamma_2$ is self-simulable.
\end{theorem}

The proof of Theorems~\ref{thm:commensurable} and~\ref{thm:quasiisometry} are given in Section~\ref{sec:QIrigidity}. The latter is essentially based on a construction of Cohen~\cite{Cohen2014} which shows that quasi-isometric finitely presented groups can be embedded in a cocycle-like manner in subshifts of finite type on either group.

In Section~\ref{sec:applications} we employ the results obtained in the last sections to exhibit several explicit examples of groups which are self-simulable. More precisely, we show

\begin{theorem}
	The following groups are self-simulable.
	\begin{itemize}
		\item Finitely generated non-amenable branch groups (Corollary~\ref{cor:branches}).
		\item The finitely presented simple groups of Burger and Mozes~\cite{BurgerMozes2020} (Corollary~\ref{cor:Burgermozes}).
		\item Thompson's group $V$ and higher-dimensional Brin-Thompson's groups $nV$ (Corollary~\ref{cor:ThompsonV}).
		\item The general linear groups $\operatorname{GL}_n(\ZZ)$ and special linear groups $\operatorname{SL}_n(\ZZ)$ for $n \geq 5$ (Corollary~\ref{cor:GL}).
		\item The automorphism group $\Aut(F_n)$ and outer automorphism group $\Out(F_n)$ of the free group on at least $n \geq 5$ generators (Corollary~\ref{cor:AutFn}).
		\item Braid groups $B_n$ on at least $n \geq 7$ strands (Corollary~\ref{cor:braidgroups}).
		\item Right-angled Artin groups associated to the complement of a finite connected graph for which there are two edges at distance at least $3$ (Corollary~\ref{cor:raag1}).
\end{itemize}\end{theorem}

Furthermore, we show a conditional result on Thompson's groups $F$ and $T$ (Corollary~\ref{cor:FT}). Namely, that Thompson's group $F$ is self-simulable if and only if it is non-amenable, and that if $F$ is non-amenable then $T$ is self-simulable. This has the interesting consequence of providing a computability criterion for the well-known open problem of the amenability of $F$. 

\begin{theorem}
	Thompson's group $F$ is amenable if and only if there exists an effectively closed action of $F$ which is not the topological factor of an $F$-subshift of finite type.
\end{theorem}

In order to obtain the previous results, we produce a toolbox which provides simple criteria to check that a group is self-simulable. We provide an abstract set of orthogonality conditions (Lemmas~\ref{lem:acts_on_product} and~\ref{lem:suborthogonal}) ensuring self-simulability, which works for most of the groups we list. These roughly generalize the idea of a group acting on a direct product, in such a way that the subactions that act on a bounded number of coordinates generate the group, and are non-amenable (for example, the general linear group is generated by elementary matrices, and $\mathrm{GL}_2(\ZZ)$ is non-amenable).

Another tool is Lemma~\ref{lem:acts_on_0d} which states that if a finitely generated recursively presented group acts faithfully on a zero-dimensional space such that for every non-empty open set the subgroup of elements which act trivially on the complement is non-amenable, then the group is self-simulable. We note that it is a common theme in the theory of infinite groups that if a group acts in a sufficiently complex way on small open sets, then we can say something about the abstract group. In particular this is common in proofs of simplicity of infinite groups, see~\cite{Higman1954,Epstein1970,Br04,Matui2015}, and a similar assumption appears in an important theorem of Rubin~\cite{Rubin1989}.

In section~\ref{sec:questions} we present a few perspectives and questions we were not able to solve ourselves. Particularly, we discuss potential dynamical properties of the subshift of finite type extension and the class of self-simulable groups. 

\textbf{Acknowledgments}: The authors wish to thank Nathalie Aubrun for many fruitful discussions on a potential candidate for self-simulable group, which were a guiding influence for many of the ideas in this paper. We also thank three anonymous reviewers whose comments greatly improved the presentation of this article, particularly shortening Section~\ref{sec:QIrigidity} considerably. S. Barbieri was supported by the FONDECYT grants 11200037 and 1240085, and the ANR projects CoCoGro (ANR-16-CE40-0005) and CODYS (ANR-18-CE40-0007). M. Sablik was supported by ANR project Difference (ANR-20-CE40-0002) and the project Computability of asymptotic properties of dynamical systems from CIMI Labex (ANR-11-LABX-0040). V. Salo was supported by the Academy of Finland project 2608073211.

\section{Preliminaries}

We denote by $\NN$ the set of non-negative integers and use the notation $A \Subset B$ to denote that $A$ is a finite subset of $B$.

We will begin with a short introduction to computability, mainly to set the notation for the rest of the article. We refer readers interested in a more in-depth introduction to~\cite{arora2009computational,sipser2006introduction,Rogers1987theory}.

A \define{Turing machine} $\mathcal{M}$ is given by a tuple \[ \mathcal{M} = (Q,\Sigma, q_0, q_F,\delta), \]
where $Q$ is a finite set of \define{states}, $\Sigma$ is a finite set called the \define{tape alphabet}, $q_0\in Q$ and $q_F \in Q$ are the \define{initial} and \define{final} state respectively and $\delta \colon Q \times \Sigma \to Q \times \Sigma \times \{-1,0,1\}$ is the \define{transition function}.

Let us briefly describe the dynamics of Turing machines on one-sided tapes. Let $\mathcal{M}$ be a Turing machine as above. Given $(q,x,n)\in Q \times \Sigma^{\NN} \times  \NN$, we interpret $x \in \Sigma^{\NN}$ as an infinite tape filled with symbols from $\Sigma$, and $(q,n) \in Q\times \NN$ as a working ``head'' which is on position $n$ of the tape at state $q$. If $\delta(q,x(n)) = (r,a,d)$ and $n+d \geq 0$, then $\mathcal{M}(q,x,n) = (r,x',n+d)$ where $x'$ is the configuration which coincides with $x$ in $\NN \setminus \{n\}$ and $x'(n)=a$. If $n+d <0$ then $\mathcal{M}(q,x,n) = (q,x,n)$. In other words, the head reads the symbol at position $n$, and according to its transition function, changes said symbol to $a \in \Sigma$, changes its state to $r \in Q$, and moves itself by $d \in \{-1,0,1\}$ if it can do it without exiting the tape. A Turing machine \define{accepts} an \define{input} $x \in \Sigma^{\NN}$ if starting from $(q_0,x,0)$ it eventually reaches the final state $q_F$, that is, there is $t \in \NN$ so that $\mathcal{M}^t(q_0,x,0) \in \{q_F\} \times \Sigma^{\NN} \times \NN$. 

Let $A$ be a finite set and denote by $\sqcup \notin A$ a \define{blank symbol}. Let $\mathcal{M}$ be a Turing machine whose tape alphabet contains $A \cup \{\sqcup\}$. We say that $\mathcal{M}$ accepts a word $w = w_0\dots w_{n-1} \in A^*$ if it accepts the infinite word $x \in \Sigma^\NN$ given by \[ x(k) = \begin{cases}
	w_k & \mbox{ if } 0 \leq k < n\\
	\sqcup & \mbox{ if } n \leq k.
\end{cases}  \]

Let $L \subset A^*$ be a language. We say $L$ is \define{recursively enumerable} if there exists a Turing machine $\mathcal{M}$ which accepts $w \in A^*$ if and only if $w \in L$. We say that $L$ is \define{co-recursively enumerable} if $A^*\setminus L$ is recursively enumerable. Finally, we say that $L$ is decidable if it is both recursively and co-recursively enumerable.

\begin{remark}
	For all our purposes, we will only consider Turing machines which never move the head outside of the tape. It is a straightforward exercise to show that the notions of decidability do not change with this assumption.
\end{remark}

On what follows we will informally refer to Turing machines which have several inputs on different ``tapes''. This situation can be easily modeled by choosing as tape alphabet the direct product of the alphabets on each tape. 

\subsection{Group actions}

For a group $\Gamma$, we denote its identity by $1_{\Gamma}$, and for a word $w = w_1w_2\dots w_k \in \Gamma^*$, we denote by $\underline{w}$ the element of $\Gamma$ which corresponds to multiplying all the $w_i$ from left to right. If $\Gamma$ is finitely generated, we write $S_{\Gamma}$ for a finite set of $\Gamma$ which generates $\Gamma$ as a monoid, that is, for every $g \in \Gamma$ there is $w \in S_{\Gamma}^*$ such that $g = \underline{w}$. We shall drop the subscript $\Gamma$ if the context is clear. Throughout the article, we assume all groups are infinite unless stated otherwise.

The \define{word problem} of a group $\Gamma$ with respect to $S \Subset \Gamma$ is the language \[ \texttt{WP}_{S}(\Gamma) = \{ w \in S^* :  \underline{w} = 1_{\Gamma}   \}.  \]

A finitely generated group $\Gamma$ is \define{recursively presented} (with respect to $S \Subset \Gamma$) if $\texttt{WP}_{S}(\Gamma)$ is recursively enumerable. A group is said to have \define{decidable word problem} (with respect to $S \Subset \Gamma$) if $\texttt{WP}_{S}(\Gamma)$ is decidable. These notions do not depend upon the choice of generating set $S$.

Let $\Gamma \curvearrowright X$ and $\Gamma \curvearrowright Y$ be (left) actions by homeomorphisms of a group $\Gamma$ on two compact metrizable spaces $X$ and $Y$. A map $\phi\colon X \to Y$ is called a \define{topological morphism} if it is continuous and $\Gamma$-equivariant, that is, if $\phi(gx) = g\phi(x)$ for every $g \in \Gamma$ and $x \in X$. A topological morphism which is onto is called a \define{topological factor map}. If there exists a topological factor map $\phi\colon X \to Y$ we say that $\Gamma \curvearrowright X$ is an \define{extension} of $\Gamma \curvearrowright Y$ and that $\Gamma \curvearrowright Y$ is a \define{factor} of $\Gamma \curvearrowright X$. If the topological morphism is a bijection, we say that $\Gamma \curvearrowright X$ is \define{topologically conjugate} to $\Gamma\curvearrowright Y$.

Given an action $\Gamma \curvearrowright X$ and a subgroup $H \leqslant \Gamma$, the \define{$H$-subaction} of $\Gamma \curvearrowright X$ is the action $H \curvearrowright X$ defined by the restriction of $\Gamma \curvearrowright X$ to $H$. 

\subsection{Effectively closed actions}
Let $A$ be a finite set, for instance $A = \{\symb{0},\symb{1}\}$. For a word $w = w_0w_1\dots w_{n-1} \in A^*$ we denote by $[w]$ the cylinder set of all $x \in A^{\NN}$ such that $x_i = w_i$ for $0 \leq i \leq n-1$. We endow $A^{\NN}$ with the prodiscrete topology whose basis is given by the cylinders $[w]$ for $w
\in A^*$.  

\begin{definition}
	A closed subset $X \subseteq A^{\NN}$ is \define{effectively closed} if there exists a Turing machine which accepts $w \in A^*$ on input if and only if $[w] \cap X = \varnothing$.
\end{definition}

Note that if $X$ is effectively closed and $[w]\cap X \neq \varnothing$, then the Turing machine run on $w$ will just loop indefinitely. Intuitively, a set $X$ is effectively closed if there is an algorithm which provides a sequence of approximations of the complement of $X$ as a union of cylinders. The next definition endows effectively closed sets with a computable action of a finitely generated group on $X$.

Let $X \subseteq A^\NN$, $\Gamma \curvearrowright X$ and $S$ be a finite generating set of $\Gamma$ which contains the identity. Consider the alphabet $B = A^S$ and for $y \in B^{\NN}$ and $s \in S$ let $\pi_s y = \{ y(n)(s)\}_{n \in \NN} \in A^\NN$. The \define{set representation} of $\Gamma \curvearrowright X$ determined by $S$ is the set $Y_{\Gamma \curvearrowright X, S} \subseteq B^{\NN}$ given by \[ Y_{\Gamma \curvearrowright X, S} = \{ y \in B^\NN : \pi_{1_{\Gamma}}y \in X, \mbox{ and for every } s \in S, \pi_s y = s (\pi_{1_{\Gamma}}y) \}.  \]

\begin{definition}\label{def:ecaction}
	Let $\Gamma$ be a group generated by $S \Subset \Gamma$ and $X \subseteq A^{\NN}$. An action $\Gamma \curvearrowright X$ is \define{effectively closed} if the set representation $Y_{\Gamma \curvearrowright X, S}$ is an effectively closed subset of $(A^S)^{\NN}$.
\end{definition}

This definition is the natural generalization to finitely generated groups of the definition of Hochman~\cite[Definition 1.2]{Hochman2009b} for actions of $\ZZd$. Intuitively, an action $\Gamma \curvearrowright X$ is effectively closed if there is a Turing machine which on input $y \in (A^S)^{\NN}$ can detect whether $\pi_{1_{\Gamma}}y \notin X$, or if there is $s \in S$ such that $\pi_s y \neq s (\pi_{1_{\Gamma}}y)$. 

\begin{remark}
	Some authors use the term ``effective action'' to mean an action with is faithful (i.e.\ that every $g \in \Gamma\setminus \{1_{\Gamma}\}$ acts non-trivially). We do not use this term with this meaning. For instance, the trivial action of any non-trivial group on an effectively closed set is an effectively closed action which is not faithful.
\end{remark}

While Definition~\ref{def:ecaction} is exactly what we will use to prove our results, sometimes we will make use of an equivalent statement which is useful for showing that a particular action is effectively closed. Namely, that effectively closed actions are precisely those for which the set is effectively closed and every generator induces a computable map.

\begin{definition}\label{def:ecaction2}
	Let $\Gamma$ be a group generated by $S \Subset \Gamma$ and $X \subseteq A^{\NN}$ be effectively closed. An action $\Gamma \curvearrowright X$ is \define{effectively closed} if there exists a Turing machine which on inputs $a \in A$, $s \in S$, $n \in \NN$ and $x\in X$, accepts if and only if $a =(sx)_n$.
\end{definition}

A way to think about an effectively closed action is as an action such that for every $s \in S$, there is a Turing machine with two tapes. One tape contains all the symbols of some $x \in X$ in order and the other tape is empty. The machine can use the other tape to perform computation and sequentially produce the symbols of $sx$ while reading coordinates of $x$. Notice that in order to output a symbol $(sx)_n$ the Turing machine only visits finitely many positions of $x$. 

\begin{remark}
	Every effectively closed action $\Gamma \curvearrowright X$ is continuous. Indeed, if it were not continuous, there would exist $m \in \NN$, $s \in S$ so that for every $n \in \NN$ there are $x^n,y^n \in X$ such that $x^n_k = y^n_k$ for every $k \leq n$ and $(sx^n)_m \neq (s{y^n})_m$. It follows that the algorithm computing $(sx^n)_m$ must take at least $n+1$ steps. Taking a limit point $\bar{x}\in X$ of the sequence of $x^n$ yields a point for which the algorithm cannot compute $\bar{x}_m$ in any number of steps, therefore contradicting the assumption that $\Gamma \curvearrowright X$ is effectively closed.
\end{remark}

\begin{remark}
	In the definition of effectively closed action the Turing machine computes the image of $x \in X$ under a generator $s \in S$ of $\Gamma$. The algorithm can of course be adapted so that it computes the image of $x$ under any $w \in S^*$. That is, there is a Turing machine which on input $w \in S^*$, $n \in \NN$ and $x \in X$ computes $(\underline{w}x)_n$. In particular the notion of effectively closed action does not depend upon the choice of generators, and it can be naturally be extended to uniformly computable actions of countable groups which are not finitely generated but whose elements can be recursively enumerated.
\end{remark}

\begin{remark} In the definition of effectively closed action we may consider just $A = \{\symb{0},\symb{1}\}$ and obtain the same class of actions up to topological conjugacy. Indeed, every finite set $A$ can be coded as binary words in $\{\symb{0},\symb{1}\}^{\lceil \log(|A|) \rceil}$ and thus we may recode any action on a subset of $A^{\NN}$ as an action on a subset of $\agb^\NN$ by replacing every symbol of $A$ by its coding.
\end{remark}

\begin{remark}\label{remark_ec_BS}
	There is a third equivalent notion of effectively closed action which is used in~\cite{BS2018,Barbieri_2019_DA} and is an analogue of Definition~\ref{def:ecaction2} which only performs computation on words. Namely, $\Gamma \curvearrowright X \subseteq A^{\NN}$ is effectively closed if the action is continuous and there exists a Turing machine which on input $s \in S$ and $u,v \in A^*$ accepts if and only if $[v] \cap s([u] \cap X)  = \varnothing$. Notice that in this definition we assume that the action is continuous.
\end{remark}

The definition of effectively closed action does not require anything on the acting group other than being finitely generated. However, in order to be able to naturally construct actions which are effectively closed, we shall often ask that the acting group is recursively presented or that it has decidable word problem. To justify this, we shall show that even admitting a single faithful effectively closed action imposes strong computability constraints on the word problem of the acting group.

To this end, we remark that effectively closed sets are called $\Pi_1^0$ sets in the arithmetical hierarchy and their complements are called $\Sigma_1^0$. The $\Sigma_2^0$ sets are those for which there is a Turing machine with an oracle to $\Pi_1^0$ which accepts if the element belongs to that set, see~\cite[Chapter 14]{Rogers1987theory} for further background. $\Pi_2^0$ sets are the complements of those in $\Sigma_2^0$. In terms of languages, the word problem of a group would be in $\Pi_2^0$ if there is a Turing machine with an oracle to $\Pi_1^0$ which accepts a word on a given set of generators if and only if it is not the identity of the group.

\begin{proposition}\label{prop:comp_restriction_ec}
	Suppose $\Gamma$ is a finitely generated group which admits an effectively closed faithful action $\Gamma \curvearrowright X$. Then the word problem of $\Gamma$ is in $\Pi_2^0$.
\end{proposition}

\begin{proof}
	Let $S$ be a set of generators for $\Gamma$ and $\Gamma \curvearrowright X\subseteq \{\symb{0},\symb{1}\}^{\NN}$ a faithful effectively closed action. Consider the algorithm which on input $w\in S^*$, initializes a variable $n = 1$, enumerates all words $u \in \{\symb{0},\symb{1}\}^{n}$. For each word, it calls the oracle to $\Pi_1^0$ to determine if $[u]\cap X \neq \varnothing$. If this is not the case, it tries to compute the image under $\underline{w}$ of any configuration $x\in \{\symb{0},\symb{1}\}^{\NN}$ such that $x_0\dots x_{n-1} =u$. If at some point the algorithm requires coordinates of $x$ outside of $\{0,\dots,n-1\}$, it increases the value of $n$ and restarts the process. If at any point the algorithm computes a coordinate $(\underline{w}x)_m \neq x_m$ then we know $\underline{w}x \neq x$ and thus that $\underline{w}\neq 1_{\Gamma}$. Such an instance always exists due to the action being faithful. 
\end{proof}

In particular, the only effectively closed actions of a simple group whose word problem is not on $\Pi_0^2$ would be the trivial actions on effectively closed sets. Notice that there are countably many $\Pi_2^0$ sets, and therefore (as there are uncountably many non-isomorphic simple finitely generated groups~\cite{Olshanskii1981}) there are uncountably many finitely generated groups which do not admit any nontrivial effectively closed action.

\begin{remark}
	If $\Gamma$ admits a faithful effectively closed action on $\{\symb{0},\symb{1}\}^{\NN}$, then we can ignore the call to $\Pi_1^0$ in the previous proof and conclude that the word problem of $\Gamma$ is co-recursively enumerable.
\end{remark}

\subsection{Shift spaces}

Let $\ag$ be a finite set and $\Gamma$ be a group. The set $\ag^\Gamma = \{ x\colon \Gamma \to \ag\}$ equipped with the left \define{shift} action $\Gamma \curvearrowright \ag^{\Gamma}$ by left multiplication given by 
\[ gx(h) \isdef x(g^{-1}h) \qquad \mbox{  for every } g,h \in \Gamma 
\mbox{ and } x \in \ag^\Gamma, \]
is the \define{full $\Gamma$-shift}. The elements $a \in \ag$ and $x \in \ag^\Gamma$ are called \define{symbols} and \define{configurations} respectively. We endow $\ag^\Gamma$ with the prodiscrete topology generated by the clopen subbase given by the \define{cylinders} $[a]_g \isdef \{x \in \ag^\Gamma : x(g) = a\}$ where $g \in \Gamma$ (we just denote by $[a]$ the cylinder $[a]_{1_{\Gamma}}$). A \define{support} is a finite subset $F \Subset \Gamma$. Given a support $F$, a \define{pattern} with support $F$ is an element $p \in \ag^F$. We denote the cylinder generated by $p$ by $[p] = \bigcap_{h \in F}[p(h)]_{h}$. Given two patterns $p,q$ we say that $q$ \define{occurs} in $p$ if $[p]\subseteq g[q]$ for some $g \in \Gamma$.

\begin{definition}
	A subset $X \subseteq \ag^\Gamma$ is a \define{$\Gamma$-subshift} if and only if it is $\Gamma$-invariant and closed in the prodiscrete topology. 
\end{definition}

If the context is clear, we drop the $\Gamma$ from the notation and speak plainly of a subshift.  Equivalently, $X$ is a subshift if and only if there exists a set of forbidden patterns $\FF$ such that \[X=X_\FF \isdef  {\ag^\Gamma \setminus \bigcup_{p \in \FF, g \in \Gamma} g[p]}.\]

Let us briefly recall that an action $\Gamma \curvearrowright X$ on a compact metrizable space is expansive if there is a constant $C>0$ so that whenever $d(gx,gy)< C$ for every $g \in \Gamma$ then $x=y$. It is a well known fact that an action $\Gamma \curvearrowright X$ on a closed subset $X$ of $\{\symb{0},\symb{1}\}^{\NN}$ is topologically conjugate to a $\Gamma$-subshift if and only if the action is expansive~\cite{Hedlund1969}.

\begin{definition}
	A subshift $X \subseteq \ag^\Gamma$ is a \define{subshift of finite type (SFT)} if there exists a finite set of forbidden patterns $\FF$ such that $X = X_{\FF}$.
\end{definition}

It is clear that if we fix $F \Subset \Gamma$ and $\mathcal{L} \subseteq A^{F}$, we may define a subshift of finite type $X \subseteq \ag^{\Gamma}$ by declaring that $x \in X$ if and only if for every $g \in \Gamma$, $(x(gh))_{h \in F} \in \mathcal{L}$. In other words, we may also define subshifts of finite type by saying which patterns are allowed within a finite support.

\begin{definition}
	A subshift $Y \subseteq \ag^\Gamma$ is \define{sofic} if it is the topological factor of a subshift of finite type.
\end{definition}

It will be useful to describe topological morphisms between subshifts in an explicit way. The following classical theorem provides said description. A modern proof for actions of countable groups may be found in~\cite[Theorem 1.8.1]{ceccherini-SilbersteinC09}.

\begin{theorem}[Curtis-Lyndon-Hedlund~\cite{Hedlund1969}]\label{thm:curtis_lyndon_hedlund}
	Let $\Gamma$ be a countable group, $\ag_X$ and $\ag_Y$ be two finite sets  and $X \subseteq \ag_X^\Gamma, Y \subseteq \ag_Y^\Gamma$ be two $\Gamma$-subshifts. A map $\phi\colon X \to Y$ is a topological morphism if and only if there exists a finite set $F \Subset \Gamma$ and $\Phi\colon \ag_X^F \to \ag_Y$ such that for every $x \in X$ and $g \in \Gamma$ then $(\phi(x))(g) = \Phi((g^{-1}x)|_{F})$.
\end{theorem}

In the above definition, a map $\phi\colon X \to Y$ for which there is $\Phi\colon \ag_X \to \ag_Y$  so that $(\phi(x))(g) = \Phi(x(g))$ for every $g \in \Gamma$ is called a \define{$1$-block map}. If $Y$ is a sofic subshift, it is always possible to construct an SFT $\widetilde{X}$ and a $1$-block topological factor map $\widetilde{\phi}\colon \widetilde{X} \to Y$. Furthermore, if $\Gamma$ is generated by a finite set $S$, one can ask that the SFT $\widetilde{X}$ is \define{nearest neighbor} with respect to $S$, that is, it is described by a set of forbidden patterns whose supports are of the form $\{1_{\Gamma},s\}$ where $s$ is a generator of $\Gamma$. In particular, if $A_{\widetilde{X}}$ is the alphabet of $\widetilde{X}$ and $1_{\Gamma}\in S$, there is $\mathcal{L} \subseteq (A_{\widetilde{X}})^S$ such that $\widetilde{x} \in \widetilde{X}$ if and only if $(g\widetilde{x})|_{S} \in \mathcal{L}$ for every $g \in \Gamma$. A proof of this elementary fact can be found in~\cite[Proposition 1.7]{tesis}.

\subsection{Effectively closed subshifts}

When working with subshifts it is often useful to give an explicit notion of effectively closed in order to be able to work with the space $A^{\Gamma}$ instead of $A^{\NN}$. In the case of $\ZZ$-subshifts, patterns (with convex support) can be naturally identified with words, and thus it is natural to think of effectively closed $\ZZ$-subshifts as those that can be described by a recursively enumerable set of forbidden words. For a general finitely generated group $\Gamma$, we need to codify patterns in such a way that a Turing machine might perform computation on them. The following formalism was introduced in~\cite{ABS2017}.

\begin{definition}
	Let $\Gamma$ be a finitely generated group, $S$ a finite set of generators for $\Gamma$ and $\ag$ an alphabet. A function $c\colon W \to \ag$ from a finite subset $W \Subset S^*$ is called a \define{pattern coding}. The \define{cylinder} defined by a pattern coding $c$ is given by
	\[ [c] = \bigcap_{w \in W} [c(w)]_{\underline{w}}.  \]
\end{definition}

A pattern coding can be thought of as pattern on free monoid $S^*$. It can be represented on the tape of a Turing machine as a finite sequence of tuples $\{(w_1,a_1),(w_2,a_2),\dots,(w_k,a_k)\}$ where $w_i \in S^*, a_i \in \ag$ and $c(w_i)=a_i$. A set $\mathcal{C}$ of pattern codings defines a $\Gamma$-subshift $X_{\mathcal{C}}$ by setting \[X_{\mathcal{C}} = \ag^\Gamma \setminus \bigcup_{g \in \Gamma, c \in \mathcal{C}} g[c].\]

\begin{definition}
	A $\Gamma$-subshift $X$ is \define{effectively closed by patterns} if there exists a recursively enumerable set of pattern codings $\mathcal{C}$ such that $X = X_{\mathcal{C}}$. 
\end{definition}

A Turing machine which accepts a pattern coding $c$ if and only if $c \in \mathcal{C}$ is said to \define{define} $X = X_{\mathcal{C}}$. Note that neither $\mathcal{C}$ nor the Turing machine which defines $X$ are unique.

It is important to stress that, unintuitively, the property of being an effectively closed subshift by patterns is weaker than the property of being an effectively closed expansive action. For instance, the full $\Gamma$-shift is a faithful action which is effectively closed by patterns for any group, but by Proposition~\ref{prop:comp_restriction_ec} there are groups which do not admit any faithful effectively closed action. However, these two notions coincide whenever the acting group is recursively presented.

\begin{proposition}\label{prop:conjugacy_rec_presented_effsubshift}
	Let $\Gamma$ be a finitely generated and recursively presented group. Then, up to topological conjugacy, expansive effectively closed actions coincide with subshifts which are effectively closed by patterns.
\end{proposition}

\begin{proof}
	In~\cite[Proposition 4.1]{BS2018} it is shown that every subshift which is the factor of an effectively closed action is effectively closed by patterns, so that direction holds regardless of the acting group. Conversely, if $\Gamma$ is recursively presented and generated by $S \Subset \Gamma$, given a computable bijection $\rho \colon \NN \to S^*$ there is an algorithm which accepts $(n,m) \in \NN^2$ if and only if $\underline{\rho(n)}=\underline{\rho(m)}$. This computable bijection and algorithm can be used to encode configurations $x \in A^{\Gamma}$ as sequences in $A^{\NN}$ as $x_{\underline{\rho(0)}}x_{\underline{\rho(1)}}x_{\underline{\rho(2)}}\dots$ and it can be shown that an algorithm to construct forbidden pattern codings can be translated into an algorithm to enumerate the complement of the set representation of the encoded action. The details of this construction are given in~\cite[Theorem 4.3]{BS2018}.
\end{proof}

In consequence with Proposition~\ref{prop:conjugacy_rec_presented_effsubshift}, whenever we have a recursively presented group we shall just say ``effectively closed subshift''  and omit the suffix ``by patterns''.

\section{Self-simulable groups}\label{sec:counterexamples}

Let us recall from the introduction the definition of self-simulable group.

\begin{definition}[Definition~\ref{def:selfsimulable}]
	We say that a group $\Gamma$ is \define{self-simulable} if every effectively closed action $\Gamma \curvearrowright X \subseteq \{\symb{0},\symb{1}\}^{\NN}$ is a topological factor of a $\Gamma$-subshift of finite type.
\end{definition}

In this brief section we will present two dynamical consequences of self-simulation and a list of obstructions for a group to have this property. We remark that the results of this section are not needed elsewhere in the article, thus a reader interested only in our main result can directly skip to Section~\ref{sec:paradox}.

\subsection{Two dynamical consequences of self-simulation}

\begin{definition}
	A $\Gamma$-subshift is \define{strongly aperiodic} if the shift action of $\Gamma$ is free. 
\end{definition}

Namely, A $\Gamma$-subshift is strongly aperiodic if the shift action satisfies that whenever $gx = x$ for some $x \in X$ then $g = 1_{\Gamma}$.

We shall use the following result from~\cite{ABT2018} which ensures that every finitely generated group with decidable word problem admits a nonempty effectively closed subshift which is strongly aperiodic. The proof relies on a probabilistic argument of~\cite{alonetal_nonrepetitivecoloringofgraphs} which shows that every graph of bounded degree can be colored with finitely many colors in such a way that the sequence of colors of every non self-intersecting path does not contain two subsequent occurrences of the same sequence of colors.

\begin{theorem}[Theorem 2.6 of~\cite{ABT2018}]\label{thm:ABT}
	For every finitely generated group with decidable word problem there exists a nonempty strongly aperiodic and effectively closed subshift.
\end{theorem}

A clear consequence of Theorem~\ref{thm:ABT} is the following.

\begin{corollary}\label{cor:stronglyaperiodic}
	Every finitely generated self-simulable group with decidable word problem admits a nonempty strongly aperiodic subshift of finite type.
\end{corollary}

\begin{proof}
	Let $\Gamma$ be a finitely generated self-simulable group with decidable word problem. By Theorem~\ref{thm:ABT} there exists a strongly aperiodic and effectively closed $\Gamma$-subshift $Y$. As $\Gamma$ is self-simulable (and recursively presented), there exists a nonempty $\Gamma$-subshift of finite type $X$ and a factor map $\phi\colon X \to Y$. Let $g \in \Gamma$ and $x \in X$ so that $gx = x$. Then we have that $g\phi(x) = \phi(gx) = \phi(x)$. As $Y$ is strongly aperiodic, we conclude that $g = 1_{\Gamma}$ and thus $X$ is strongly aperiodic.
\end{proof}

In particular, there exist nonempty strongly aperiodic subshifts of finite type on every group with decidable word problem which we present in Section~\ref{sec:applications}.

\begin{definition}
	Let $\Gamma$ be a group. The sunny-side up subshift is given by \[X_{\leq 1} = \{ x \in \{\symb{0},\symb{1}\}^{\Gamma} : \mbox{ if } x(g)=x(h)=1 \mbox{ then } g = h\}. \]
\end{definition}

In words, the sunny-side up subshift consists of all configurations on the alphabet $\{\symb{0},\symb{1}\}$ which have at most one position marked by the symbol $\symb{1}$. This object is the symbolic representation of the translation action of $\Gamma$ on its one-point compactification $\Gamma \cup \{\infty\}$. It is well known that if $\Gamma$ is an infinite and finitely generated group, then $X_{\leq 1}$ is not a subshift of finite type (see~\cite[Proposition 2.10]{ABS2017}). An interesting problem is figuring out for which groups $\Gamma$ the sunny-side up subshift is sofic. It has been shown by Dahmani and Yaman in~\cite{DahmaniYaman2008} that (1) if $\Gamma$ splits in a short exact sequence $1 \to N \to \Gamma \to H \to 1$ and $X_{\leq 1}$ is sofic for both $N$ and $H$ then it is also sofic for $\Gamma$, (2) soficity of the sunny-side up subshift is an invariant of commensurability and that (3) the sunny-side up subshift is sofic for finitely generated free groups, finitely generated abelian groups, hyperbolic groups and poly-hyperbolic groups.

The first examples of finitely generated groups where the sunny-side up subshift is not sofic were given in~\cite{ABS2017}. Namely~\cite[Proposition 2.11]{ABS2017} states that if $\Gamma$ is a recursively presented group then $X_{\leq 1}$ is effectively closed if and only if the word problem of $\Gamma$ is decidable. It follows that if $\Gamma$ is a recursively presented group with undecidable word problem then $X_{\leq 1}$ is not effectively closed and therefore is also not sofic. To the knowledge of the authors, it is still unknown whether there exists a finitely generated group $\Gamma$ with decidable word problem for which the sunny-side up subshift is not sofic. 

An obvious consequence of a group being self-simulable, is that if the group has decidable word problem, then the sunny-side up subshift is sofic, as all effectively closed subshifts are automatically sofic.

\begin{corollary}\label{cor:sunnysideupsofic}
	Let $\Gamma$ be a finitely generated, self-simulable group with decidable word problem. The sunny-side up subshift on $\Gamma$ is sofic.
\end{corollary}

\subsection{Obstructions to self-simulation}

In this brief section, we present classes of groups which are not self-simulable. As we mentioned before, our focus is on groups which are both infinite and finitely generated. Let us briefly justify that self-simulability is not interesting on finite groups, or groups which are not finitely generated.

If $\Gamma$ is a finite group, it follows that every subshift $X\subset A^{\Gamma}$ is finite. In particular, the trivial action on an infinite effectively closed set can not be the factor of a subshift on $\Gamma$, and thus $\Gamma$ is not self-simulable. On the other hand, every $\Gamma$-subshift is of finite type, and thus expansive effectively closed actions of $\Gamma$ are trivially factors of SFTs.

On the other hand, while the definition of effectively closed action (Definition~\ref{def:ecaction2}) can be naturally extended to countable groups with a recursive presentation, non-finitely generated groups can never be self-simulable for trivial reasons. Notice that the trivial action of a group on a finite set is always expansive and effectively closed, and thus it is topologically conjugate to an effectively closed subshift on that group. In particular, if said trivial action were the factor of a subshift of finite type, it would be topologically conjugate to a sofic subshift with finitely many configurations. However, it is an immediate consequence of the definition that in any countable non-finitely generated group, sofic subshifts are either trivial or uncountable.

Let us turn our attention back to infinite and finitely generated groups. We will first show that amenability is an obstruction to self-simulability. Recall that a group $\Gamma$ is amenable if for every $\delta>0$ and $K\Subset \Gamma$ there is $F\Subset \Gamma$ such that \[ |KF \triangle F |\leq \delta|F|, \textrm{ where }KF \triangle F=(KF\cup F)\smallsetminus (KF\cap F).   \]

To every action $\Gamma \curvearrowright X$ of a countable amenable group on a compact metrizable space $X$ we can associate a non-negative extended real number $\htop(\Gamma \curvearrowright X)$ called the topological entropy of $\Gamma \curvearrowright X$. (See section 9 of ~\cite{KerrLiBook2016}). It is well known that if $\Gamma \curvearrowright Y$ is a factor of $\Gamma \curvearrowright X$ then \[ \htop(\Gamma \curvearrowright X) \geq \htop(\Gamma \curvearrowright Y).  \]

For every recursively presented and finitely generated amenable group $\Gamma$ there are effectively closed actions with infinite topological entropy. Indeed, let $S$ be a finite symmetric set of generators of $\Gamma$. Identify $\{\symb{0},\symb{1}\}^{\NN}$ with $\{\symb{0},\symb{1}\}^{S^* \times \NN}$ through a computable bijection and let $X \subseteq \{\symb{0},\symb{1}\}^{S^* \times \NN}$ be such that $x \in X$ if and only if for every $n \in \NN$ and $w,w'\in S^*$ such that $\underline{w} = \underline{w'}$ we have $x(w,n)=x(w',n)$. Consider the action $\Gamma \curvearrowright X$ given by \[ (g(x))(w,n) = x(g^{-1}w,n)  \mbox{ for every } g \in \Gamma, (w,n)\in S^* \times \NN.  \] 
Where $g^{-1}w$ stands for any word in $S^*$ representing the element $g^{-1}\underline{w}$.

We may think of this action as a full $\Gamma$-shift on an infinite alphabet. A direct computation shows that this action has infinite topological entropy. On the other hand, as $\Gamma$ is recursively presented it follows that the set presentation of $\Gamma \curvearrowright X$ determined by $S$ is effectively closed, and thus $\Gamma \curvearrowright X$ is an effectively closed action.

\begin{proposition}
	\label{prop:SelfSimuImpNonAmenable}
	Let $\Gamma$ be a finitely generated and recursively presented group. If $\Gamma$ is self-simulable then it is non-amenable.
\end{proposition}

\begin{proof}
	Assume that $\Gamma$ is amenable. If $\Gamma \curvearrowright X$ is a subshift, then $\htop(\Gamma \curvearrowright X) < \infty$ (see Example 9.41 of~\cite{KerrLiBook2016}). On the other hand, as $\Gamma$ is recursively presented there exist non-expansive effectively closed actions with infinite {topological entropy} which can thus never be factors of a subshift.
\end{proof}

The previous argument strongly relies on non-expansive actions. One might wonder whether if one limits the notion of self-simulation to effectively closed subshifts then it might be possible for amenable groups to be self-simulable. A partial answer is that as long as the group has decidable word problem then there are always effectively closed subshifts which are not sofic. The proof is based on a counting argument and may be found in~\cite[Theorem 2.16]{ABS2017}.

\begin{proposition}[Theorem 2.16 of~\cite{ABS2017}]\label{prop:amenable}
	Let $\Gamma$ be a finitely generated amenable group with decidable word problem. There exists an effectively closed $\Gamma$-subshift which is not sofic.
\end{proposition}

Recall that the number of ends of a finitely generated group is the limit as $n$ tends to infinity of the number of connected components obtained by considering its Cayley graph with respect to a generating set and removing all elements of length $n$ with respect to the word metric induced by that generating set. The free group $F_k$ of rank $k \geq 2$ is an example of a group with infinitely many ends. The next proposition shows that these non-amenable groups are not self-simulable.

\begin{proposition}[Theorem 2.17 of~\cite{ABS2017}]\label{prop:infends}
	Let $\Gamma$ be a finitely generated group with decidable word problem and infinitely many ends. There exists an effectively closed $\Gamma$-subshift which is not sofic.
\end{proposition}

In views of Theorem~\ref{thm:selfsimulation} and Propositions~\ref{prop:amenable} and~\ref{prop:infends} one might wonder if maybe all it takes to be self-simulable is to be a $1$-ended non-amenable group. The following result shows that this is not the case. 

\begin{proposition}\label{prop:villexample}
	Let $k \geq 1$, $F_k$ be the free group of rank $k$ and $\Gamma = F_k \times \ZZ$. There exists an effectively closed $\Gamma$-subshift which is not sofic. In particular $F_k \times \ZZ$ is not self-simulable.
\end{proposition}

\begin{proof}
	Denote the generators of $F_k$ by $a_1,\dots,a_k$.  Denote the elements of $\Gamma$ as ordered pairs $(g,n)$ where $g \in F_k$ and $n \in \ZZ$. Consider the alphabet $A = \{\symb{1},\symb{2},\symb{\star}\}$ and the subshifts $X_1,X_2,X_3$ of $A^{\Gamma}$ given by \[ X_1 = \{ x \in A^{\Gamma} : \mbox{ if } x(g,n)=x(h,n)=\symb{\star} \mbox{ then } g=h \},   \]
	\[ X_2 = \{ x \in A^{\Gamma} : \mbox{ if }x(g,n)=\symb{\star} \mbox{ then } x(g,m) =\symb{\star} \mbox{ for every } m \in \ZZ    \},   \]
	\[ X_3 = \{ x \in A^{\Gamma} : \mbox{ if }x(g,n)=\symb{\star} \mbox{ then } x(gh,n) = x(gh^{-1},n) \mbox{ for every } h \in F_k    \}.   \]
	
	Let $X = X_1 \cap X_2 \cap X_3 \subseteq A^{\Gamma}$. The subshift $X$ consists of all configurations on the alphabet $A$ such that if the symbol $\symb{\star}$ occurs at some position $(g,n)$, then it occurs at position $(g,m)$ for every $m \in \ZZ$, and does not occur at any other $(g',k)$ for $g' \neq g$. Therefore if $\symb{\star}$ occurs, it occurs solely on an infinite ``pillar''. Moreover, if $\symb{\star}$ occurs, it acts as a mirror in the sense that the configuration is symmetric with respect to taking inverses relative to the pillar, that is, for every $h \in F_k$, we have that $x(gh,m) = x(gh^{-1},m)$ for every $m \in \ZZ$. In other words, following $h$ or $h^{-1}$ from the pillar leads to the same symbol. An illustration of part of a configuration in $X$ can be seen in Figure~\ref{fig:villexample}.
	
	\begin{figure}[ht!]
		\centering
		\input{img/villexample}
		\caption{Part of a configuration in $X$ for $\Gamma = F_2 \times \ZZ$ (left) and $\Gamma =F_1 \times \ZZ= \ZZ \times \ZZ$ (right).}
		\label{fig:villexample}
	\end{figure}
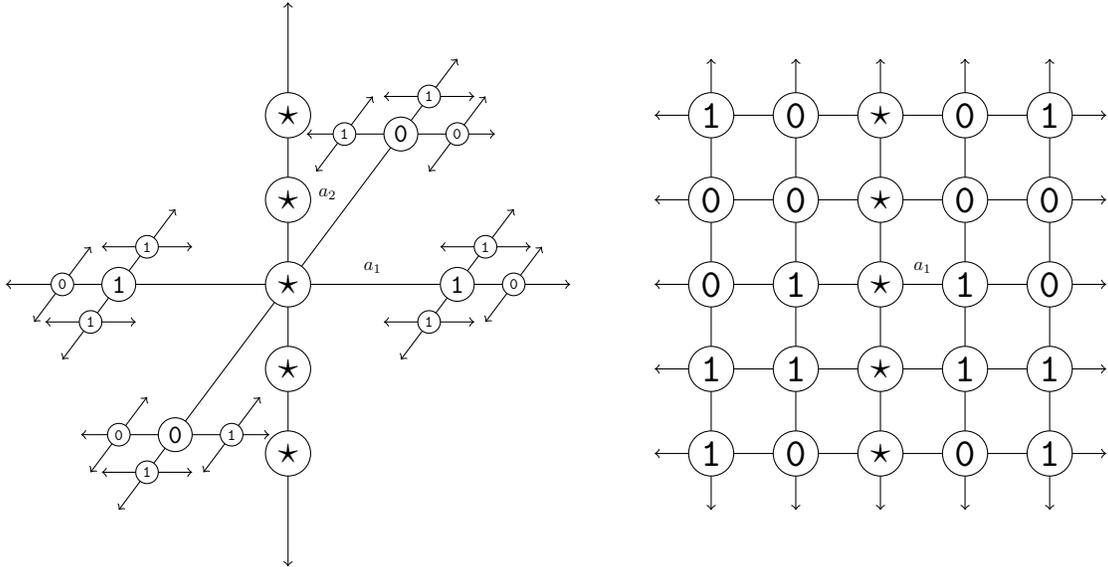

	It is straightforward to check that $X$ is effectively closed. Suppose that $X$ is a sofic subshift. Then there exists a finite alphabet $A'$, a nearest neighbor subshift of finite type $Y\subseteq (A')^{\Gamma}$, and a $1$-block factor map $\phi\colon Y \to X$. Let $N$ be an integer such that \[ N > 3\log_2(|A'|).  \]
	
	We claim that there exist two distinct configurations $y,y' \in Y$ such that \begin{enumerate}
		\item $\phi(y)(1_{F_k},0) = \phi(y')(1_{F_k},0) = \symb{\star}$.
		\item $y(1_{F_k},n) = y'(1_{F_k},n)$ for every $n \in \ZZ$.
		\item There exists $n \in \{1,\dots,N\}$ such that $\phi(y)(a_1^n,0) \neq \phi(y')(a_1^n,0)$.
	\end{enumerate}
	
	We argue by compactness. Let $m \in \NN$ and let \[  D_m = \{  (a_1^n,k) : n \in \{1,\dots,N\}, k \in \{ 0,\dots,m-1 \}  \}, \] \[ H_m = \{ (1_{F_k},k) : k \in \{-m,\dots,2m-1\}.  \]
	
	Let $p \colon D_m \to \{\symb{0},\symb{1}\}$ be an arbitrary function. By construction of $X$ there exists a configuration $x_p \in X$ such that $x_p|_{D_m} = p$ and $x_p|_{H_m}$ is identically $\symb{\star}$. In particular, for every such $p$ there must exist $y_p \in Y$ such that $\phi(y_p)=x_p$.
	Notice that as $N \geq 3\log_2(|A'|)$, then \[ |\{\symb{0},\symb{1}\}^{D_m}| = 2^{Nm} > |A'|^{3m} = |(A')^{H_m}|.  \]
	We deduce that there must exist two distinct $p,p' \colon D_m \to \{\symb{0},\symb{1}\}$ such that $y_p|_{H_m} = y_{p'}|_{H_m}$. Let $(a_1^n,r)\in D_m$ such that $\phi(y_p)(a_1^n,r)\neq \phi(y_{p'})(a_1^n,r)$, then the configurations \[ y_m = (1_{F_k},r)y_p \mbox{ and } y'_m = (1_{F_k},r)y_{p'},   \]
	given by
	\[ y_m(g,n) = y_p(g,n-r) \mbox{ and } y'_m(g,n) = y'_p(g,n-r) \mbox{ for every } (g,n) \in F_k \times \ZZ,  \]
	satisfy that: \begin{enumerate}
		\item $\phi(y_m)(1_{F_k},0) = \phi(y'_m)(1_{F_k},0) = \symb{\star}$.
		\item $y_m(1_{F_k},k) = y'_m(1_{F_k},k)$ for every $k \in \{-m,\dots,m\}$.
		\item There exists $n \in \{1,\dots,N\}$ such that $\phi(y_m)(a_1^n,0) \neq \phi(y'_m)(a_1^n,0)$.
	\end{enumerate}
	Any accumulation point of the sequence $(y_m,y'_m)_{m \in \NN}$ yields a pair $y,y' \in Y$ which satisfies the three properties above, thus proving our claim. 
	
	Let us now reach a contradiction. Let us define the configuration $y^* \in (A')^{\Gamma}$ given by \[ y^*(g,n) = \begin{cases}
		y(g,n) & \mbox{ if } g \mbox{ as a reduced word begins with }a_1\\
		y'(g,n) & \mbox{ otherwise.}
	\end{cases}    \] 
	As $Y$ is a nearest neighbor subshift of finite type and $y(1_{F_k},n) = y'(1_{F_k},n)$ for every $n \in \ZZ$, it follows that $y^*$ contains no forbidden patterns and thus $y^* \in Y$. On the other hand, we have that $\phi(y^*)(1_{F_k},0) = \symb{\star}$ and, as there is $n \in \{1,\dots,N\}$ such that $\phi(y)(a_1^n,0) \neq \phi(y')(a_1^n,0)$, we have that $\phi(y^*)(a_1^n,0) \neq \phi(y^*)(a_1^{-n},0)$. Therefore $\phi(y^*) \notin X$, contradicting the fact that $\phi \colon Y \to X$ is a factor map.\end{proof}

\begin{remark}\label{rem:mirrorshift} 
	The subshift $X$ in the proof of Proposition~\ref{prop:villexample} is usually called the \define{Mirror shift} when $k = 1$.
\end{remark}

\section{The paradoxical subshift}\label{sec:paradox}

The purpose of this section is to exploit the structure of finitely generated non-amenable groups to construct a family of subshifts of finite type with a remarkable structural property. More precisely, let $\Gamma$ be a finitely generated non-amenable group. We shall construct a nonempty $\Gamma$-subshift of finite type $\paradox$, which we call the \define{paradoxical subshift}, that has the property that any configuration $\rho\in \paradox$ induces an injective map from $\Gamma \times \NN$ to $\Gamma$ which is ``Lipschitz'' on the second component. In other words, the configuration assigns a one-sided path with bounded jumps to every element of $\Gamma$, and the paths do not intersect. 

We remark that constructions of this nature have already appeared in the literature (although not explicitly encoded in the form of a subshift of finite type). A good instance is the proof of Whyte that every non-amenable group admits a translation-like action by the free group~\cite{whyte_amenability_1999}. See also~\cite{Seward2014_translationlike}.

The existence of this subshift of finite type relies on the well-known characterization of non-amenable groups by the existence of paradoxical decompositions in terms of a convenient statement on the existence of bounded surjective $2$-to-$1$ maps.

\begin{proposition}[Theorem 4.9.2 of~\cite{ceccherini-SilbersteinC09}]\label{prop_ceccSilb}
	A group $\Gamma$ is non-amenable if and only if there exists a $2$-to-$1$ surjective map $\varphi \colon \Gamma \to \Gamma$ and a finite set $K \Subset \Gamma$ so that $g^{-1}\varphi(g)\in K$ for every $g \in \Gamma$.
\end{proposition}

The paradoxical subshift shall be the main tool in the proof of Theorem~\ref{thm:selfsimulation}. As an immediate application of this object, we shall provide at the end of this section an alternative proof of a result of Seward~\cite{Seward2014} which states that every action of a countable non-amenable group admits an extension which is a subshift. Furthermore, we shall show that under certain computability restrictions the subshift can always be picked effectively closed, see Theorem~\ref{thm:seward_new}.

\subsection{The paradoxical subshift}

Given a countable group $\Gamma$ and a symmetric set $K \Subset \Gamma$ (that is, so that $K = K^{-1}$), we define the alphabet $A_{K} = K^3 \times \{ \symb{G},\symb{B}\}$.  Given $a = ((a_1,a_2,a_3),\symb{t}) \in A_{K}$, we write $L_{\symb{G}}(a) = a_1$, $L_{\symb{B}}(a) = a_2$, $R_{\symb{t}}(a) = a_3$ and $\symb{t}(a) = \symb{t}$. The value $\symb{t}(a) \in \{ \symb{G},\symb{B}\}$ is called the \define{color} of the symbol $a \in A_K$, we interpret the color as being either \define{green} ($\symb{G}$) or \define{blue} ($\symb{B}$). The letters $L$ and $R$ stand for \define{left} and \define{right} respectively.

\begin{definition}\label{def:paradoxical_subshift}
	Let $\Gamma$ be a countable group and $K \Subset \Gamma$. The \define{paradoxical subshift} is the $\Gamma$-subshift $\paradox \subseteq (A_K)^{\Gamma}$ defined by the condition that $\rho \in \paradox$ if and only if for every $g \in \Gamma$ if we have $\rho(g) = a$ then:
	\begin{enumerate}
		\item $b = \rho(gL_{\symb{G}}(a))$ is of color $\texttt{G}$ and $R_{\symb{G}}(b) = (L_{\symb{G}}(a))^{-1}$,
		\item $c = \rho(gL_{\symb{B}}(a))$ is of color $\texttt{B}$ and $R_{\symb{B}}(c) = (L_{\symb{B}}(a))^{-1}$,
		\item If $d = \rho(gR_{\symb{t}}(a))$, then $L_{\symb{t}}(d) = (R_{\symb{t}}(a))^{-1}$.
	\end{enumerate}
\end{definition}

\begin{figure}[ht!]
	\begin{tikzpicture}
		\node (AA) at (3,0) {$\textcircled{d}$};
		\node (A) at (0,0) {$\textcircled{a}$};
		\node (B) at (-3, 1) {$\textcircled{b}$};
		\node (C) at (-3,-1) {$\textcircled{c}$};
		\node (BA) at (-6, 1.5) {$\bigcirc$};
		\node (BB) at (-6, 0.5) {$\bigcirc$};
		\node (CA) at (-6, -0.5) {$\bigcirc$};
		\node (CB) at (-6,-1.5) {$\bigcirc$};
		\draw[->, very thick, green!50!black] (A) -- (AA) node [midway, below] {$R_{\texttt{G}}(a)$};
		\draw[->, thick, green!50!black, dashed, bend right=20] (A) to node[auto, above] {$L_{\texttt{G}}(a)$} (B);
		\draw[->, very thick, green!50!black] (B) -- (A) node [midway, below] {$R_{\texttt{G}}(b)$};
		\draw[->, thick, blue, dashed, bend left=20] (A) to node[auto, below] {$L_{\texttt{B}}(a)$} (C);
		\draw[->, very thick, blue] (C) -- (A) node [midway, above] {$R_{\texttt{B}}(c)$};
		
		\draw[->, thick, green!50!black, dashed, bend right=20] (AA) to node[auto, above] {$L_{\texttt{G}}(d)$} (A);

		\draw[->, thick, green!50!black, dashed, bend right=20] (B) to node[auto, below] {$L_{\texttt{G}}(b)$} (BA);
		\draw[->, thick, blue, dashed, bend left=20] (B) to node[auto, below] {$L_{\texttt{B}}(b)$} (BB);
		\draw[->, thick, green!50!black, dashed, bend right=20] (C) to node[auto, below] {$L_{\texttt{G}}(c)$} (CA);
		\draw[->, thick, blue, dashed, bend left=20] (C) to node[auto, below] {$L_{\texttt{B}}(c)$} (CB);

	\end{tikzpicture}
	\caption{The local structure of the paradoxical subshift. The two first components of the alphabet are drawn with dashed arrows, while the third component is drawn with a thick arrow. The three conditions in the definition simply correspond to the property that following an arrow of a certain color and then following the inverse arrow of the same color amounts to no movement at all.}
	\label{fig_paradoxical}
\end{figure}
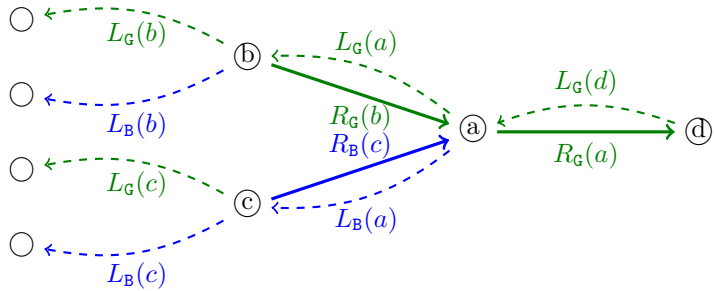

A representation of the paradoxical subshift is given on Figure~\ref{fig_paradoxical}, note that it induces a directed graph where each vertex has exactly two different ancestors, one labeled by $\symb{G}$ and another by $\symb{B}$. A finite set of forbidden patterns which defines the paradoxical subshift is the set of all patterns $p \in (A_K)^{(K \cup \{1_{\Gamma}\})}$ so that at least one of the three conditions in the definition is not satisfied for $g = 1_{\Gamma}$. It follows that the paradoxical subshift is a $\Gamma$-subshift of finite type.

\begin{proposition}\label{prop:paradox_nonempty}
	Let $\Gamma$ be a non-amenable group. Then there exists a symmetric set $K \Subset \Gamma$ containing the identity so that the paradoxical subshift $\paradox$ is nonempty.
\end{proposition}

\begin{proof}
	As $\Gamma$ is non-amenable, Proposition~\ref{prop_ceccSilb} ensures that there exists a $2$-to-$1$ surjective map $\varphi \colon \Gamma \to \Gamma$ and a finite set $K \Subset \Gamma$ so that $g^{-1}\varphi(g)\in K$ for every $g \in \Gamma$. Potentially taking a larger set $K$, we can assume that $K$ is symmetric and contains the identity, and so $\varphi(g)^{-1}g\in K$ as well.
	
	As $\varphi$ is $2$-to-$1$ and surjective, there is a partition $\Gamma = \Gamma_{\symb{G}} \cup  \Gamma_{\symb{B}}$ so that the restrictions $\varphi_{\symb{G}}$ and $\varphi_{\symb{B}}$ of $\varphi$ to $\Gamma_{\symb{G}}$ and $\Gamma_{\symb{B}}$ respectively are bijections onto $\Gamma$. Let $\rho = \{\rho(g)\}_{g \in \Gamma} \in (A_K)^{\Gamma}$ be defined by
	
	\[ \rho(g) = \begin{cases}
		\left(\left(g^{-1}\varphi_{\symb{G}}^{-1}(g), g^{-1}\varphi_{\symb{B}}^{-1}(g), g^{-1}\varphi(g)   \right), \symb{G} \right)   & \mbox{ if } g \in  \Gamma_{\symb{G}}   \\
		\left(\left(g^{-1}\varphi_{\symb{G}}^{-1}(g), g^{-1}\varphi_{\symb{B}}^{-1}(g), g^{-1}\varphi(g)   \right), \symb{B} \right)  \\
	\end{cases},   \]

	Notice that letting $h = \varphi_{\symb{G}}^{-1}(g)$ we have $\varphi(h)^{-1}h = (\varphi(\varphi_{\symb{G}}^{-1}(g)   )   )^{-1}\varphi_{\symb{G}}^{-1}(g) = g^{-1}\varphi_{\symb{G}}^{-1}(g) \in K$. Similarly, $g^{-1}\varphi_{\symb{B}}^{-1}(g) \in K$. This shows that $\rho(g) \in A_K$ for every $g \in \Gamma$. It is a straightforward computation to show that $\rho$ satisfies the three conditions in Definition~\ref{def:paradoxical_subshift}.\end{proof}

Given a color $\symb{t} \in \{ \symb{G}, \symb{B}  \}$, denote by $\overline{\symb{t}}$ the opposite color. We can associate to each $g \in \Gamma$ a continuous map $\gamma_{g} \colon \NN \times \paradox \to \Gamma$ by following the left arrows of the opposite color as that of $\rho(g)$ as follows:

\[   \gamma_g(n,\rho) = \begin{cases}
	gL_{\overline{\symb{t}}}(\rho(g)) & \mbox{ if } n = 0 \mbox { and } \rho(g) \mbox{ has color } \symb{t} \\
	\gamma_{g}(n-1,\rho)L_{\overline{\symb{t}}}(\rho(\gamma_{g}(n-1,\rho))) & \mbox{ if } n >0 \mbox { and } \rho(g) \mbox{ has color } \symb{t}\end{cases}.
\]

In simple words, the map $\gamma_g$ assigns a path to $g$ by ``following the left arrows of the opposite color of $g$''. It is clear that $\gamma_g$ is continuous, as $\gamma_g(n,\rho)$ only depends upon the values of $\rho$ in $gK^n$. Notice also that for every $g \in \Gamma$ we have \begin{align}\label{formula:bonita}
	\gamma_g(n,\rho) = g\gamma_{1_{\Gamma}}(n,g^{-1}\rho).\tag{$\heartsuit$}
\end{align}

\begin{remark}\label{rem:inverse_equation}
	Notice that if $\rho(g)$ has color $\symb{G}$, then for every $n \in \NN$ the symbol $\rho(\gamma_{g}(n,\rho))$ has color $\symb{B}$. Similarly, if $\rho(g)$ has color $\symb{B}$, then for every $n \in \NN$ the symbol $\rho(\gamma_{g}(n,\rho))$ has color $\symb{G}$. In particular, for every $n >0$,
	\[ \gamma_g(n-1,\rho) = \gamma_g(n,\rho)R_{\symb{t}}(\rho(\gamma_g(n,\rho))).\]
\end{remark}

The function $\gamma_g$ assigns a one sided path to $g$. The next lemma says that all these paths are disjoint.

\begin{lemma}\label{lem:injectivepaths}
	For every $\rho \in \paradox$, the map $(n,g) \mapsto \gamma_g(n,\rho)$ for $(n,g)\in \NN \times \Gamma$ is injective.
\end{lemma}

\begin{proof}
	Suppose there are $g,h \in \Gamma$ and $n,m$ so that $\gamma_g(n,\rho) = \gamma_h(m,\rho)$. By Remark~\ref{rem:inverse_equation} we have that $\rho(g)$ and $\rho(h)$ have the same color. Furthermore, repeatedly using the relation in the remark, we may assume that either $n = 0$ or $m=0$. Without loss of generality let $n = 0$.
	
	Let $v = \gamma_g(0,\rho)$. We have that $g = vR_{\symb{t}}(\rho(v))$ and thus $\rho(v)$ and $\rho(g)$ have different color. If $m >0$, then $\rho(v)$ is of the same color as $\rho(\gamma_h(m-1,\rho))= \rho(vR_{\symb{t}}(\rho(v)) = \rho(g)$, therefore we conclude that $m = 0$ and thus $\gamma_h(0,\rho)R_{\symb{t}}(\gamma_h(0,\rho)) = h$. Hence $g =h$.\end{proof}

\subsection{Existence of effective subshift extensions}

Before finishing this section we shall show a simple application of our construction. In~\cite[Theorem 1.2]{Seward2014} the author shows that every action of a countable group on a compact metrizable space admits a subshift extension. The proof is essentially based on theorem of Whyte~\cite{whyte_amenability_1999} on the existence of translation-like actions of $F_2$ on non-amenable groups and a generalization of the classical factor map between the full $2$-shift and the full $4$-shift in $F_2$ (See Section C of the appendix in~\cite{OrnsteinWeiss1987}).

We shall provide an alternative proof of that theorem. Furthermore, we shall show that whenever $\Gamma$ is finitely generated and $\Gamma \curvearrowright X$ is effectively closed, then the subshift can be chosen to be effectively closed (by patterns).

\begin{theorem}\label{thm:seward_new}
	Every action $\Gamma \curvearrowright X$ of a countable non-amenable group over a compact metrizable space $X$ is the factor of some $\Gamma$-subshift $Z$.
	
	Furthermore, if $\Gamma$ is finitely generated, $X \subseteq A^{\NN}$ for some finite $A$ and $\Gamma \curvearrowright X$ is effectively closed, then $Z$ can be chosen effectively closed by patterns.
\end{theorem}

\begin{proof}
	It is well known that every action of a countable group $\Gamma$ over a compact metrizable space $X$ admits a zero-dimensional extension, for instance, as an inverse limit of symbolic covers of radius $\tfrac{1}{n}$ of $X$ (see for instance~\cite[Lemma 4.1]{Seward2014}). We shall thus assume, without loss of generality, that $X$ is a closed subset of $A^{\NN}$.
	
	As $\Gamma$ is a countable non-amenable group, there exists $K \Subset \Gamma$ such that the paradoxical subshift $\paradox \subseteq (A_K)^{\Gamma}$ is nonempty. Recall that $\paradox$ comes equipped with a family of continuous maps $\gamma_{g}\colon \NN \times \paradox\to \Gamma$.
	
	Consider the subshift $Z \subseteq \paradox \times A^{\Gamma}$ which consists on all configurations $z = (\rho,y)\in \paradox \times A^{\Gamma}$ which satisfy the following two constraints.
	\begin{enumerate}
		\item For every $n \in \NN$ and $g \in \Gamma$, if we consider \[ u = y(\gamma_g(0,\rho))y(\gamma_g(1,\rho))\dots y(\gamma_g(n-1,\rho)) \in A^n,  \]
		then $[u]\cap X \neq \varnothing$.
		\item For every $n,m \in \NN$ and $g,h \in \Gamma$ if we consider \[ u = y(\gamma_g(0,\rho))y(\gamma_g(1,\rho))\dots y(\gamma_g(n-1,\rho)) \in A^n,  \]\[ v = y(\gamma_{gh}(0,\rho))y(\gamma_{gh}(1,\rho))\dots y(\gamma_{gh}(m-1,\rho)) \in A^m,  \]
		then $[v]\cap h^{-1}([u]\cap X) \neq \varnothing$.
	\end{enumerate}

	Clearly $Z$ is a closed $\Gamma$-invariant set, and therefore a $\Gamma$-subshift. Consider the map $\phi \colon Z \to A^{\NN}$ given by \[ \phi(\rho,y)_n = y(\gamma_{1_{\Gamma}}(n,\rho)) \mbox{ for every } n \in \NN.  \]
	The map $\phi$ is obviously continuous. By the first constraint, we have that $[\phi(\rho,y)|_{\{0,\dots,n-1\}}] \cap X \neq \varnothing$ for every $n \in \NN$. As $X$ is closed this means that $\phi(\rho,y)\in X$. 
	
	Let $g \in \Gamma$ and $k \in \NN$. From here we obtain, \[\phi(g\rho,gy)_k = (gy)(\gamma_{1_{\Gamma}}(k,g\rho)) = y(g^{-1}\gamma_{1_{\Gamma}}(k,g\rho)) = y( \gamma_{g^{-1}}(k,\rho)  ). \]
	
	The second constraint implies that for every $n,m \in \NN$ we have that \[ [\phi(g\rho,gy)|_{\{0,\dots,m-1\}}]\cap g([\phi(\rho,y)|_{\{0,\dots,n-1\}}]\cap X) \neq \varnothing.  \]
	Thus $\phi(g\rho,gy) = g\phi(\rho,y)$, hence the map $\phi$ is $\Gamma$-equivariant.
	
	Finally, let us show that $\phi$ is surjective. Let $x \in X$ and choose some $\rho \in \paradox$. Let $y \in A^{\Gamma}$ be any configuration such that \[ y(\gamma_g(n,\rho)) = (g^{-1}x)_n.  \]
	Such a configuration $y$ exists due to the map $(n,g)\mapsto \gamma_g(n,\rho)$ being injective. It is direct from the definition that $(\rho,y)\in Z$ as it satisfies the two constraints, and that $\phi(\rho,y)=x$. This shows that $\Gamma \curvearrowright X$ is a factor of the $\Gamma$-subshift $Z$.
	
	Now suppose that $\Gamma$ is finitely generated and that $\Gamma \curvearrowright X$ is effectively closed. We will show that $Z$ can be defined through a recursively enumerable set of forbidden pattern codings. As $\paradox \times A^{\Gamma}$ is a $\Gamma$-subshift of finite type, it is effectively closed, so we only need to show that we can enforce conditions (1) and (2) in this way.
	
	Let $S\Subset \Gamma$ be a finite symmetric set of generators of $\Gamma$ such that $K \subseteq S$. Denote by $\varepsilon$ the empty word on $S^*$. We define $\mathcal{C}$ as the set of pattern codings $c \colon \bigcup_{k \leq n}S^k \to A_K \times A$ for some $n > 1$ which fail to satisfy the following conditions:
	\begin{enumerate}
		\item[(1')] For every $w \in \bigcup_{k \leq n-1}S^k$ the conditions from Definition~\ref{def:paradoxical_subshift} (replacing $g$ by $w$) are satisfied.
		Notice that if condition (1') holds, then for every $w \in \bigcup_{k \leq n-1}S^k$ we may properly define on $c$ a local analogue $\widetilde{\gamma}_c^{w} \colon \NN \to \bigcup_{k \leq n}S^k$ of the map $\gamma_{\underline{w}} \colon \NN \times \paradox \to \Gamma$ on $c$, that is, such that \[\widetilde{\gamma}^{w}_c(k) = \gamma_{\underline{w}}(k,\rho) \mbox{ for every } k \leq n-|w|-1 \mbox{ and } \rho \in [c]\cap \paradox.   \]
		\item[(2')] Let $k \leq n$. If $u = \widetilde{\gamma}^{\varepsilon}_c(0)\widetilde{\gamma}^{\varepsilon}_c(1)\dots \widetilde{\gamma}^{\varepsilon}_c(k)$ then $[u]\cap X \neq \varnothing$.
		\item[(3')] Let $s \in S$ and $\ell,\ell' \leq n-2$. Let  \[u = \widetilde{\gamma}^{\varepsilon}_c(0)\widetilde{\gamma}^{\varepsilon}_c(1)\dots \widetilde{\gamma}^{\varepsilon}_c(\ell)\] \[v =  \widetilde{\gamma}^{s}_c(0)\widetilde{\gamma}^{s}_c(1)\dots \widetilde{\gamma}^{s}_c(\ell')\]
		Then $[v]\cap s^{-1}([u]\cap X)\neq \varnothing$.
	\end{enumerate}
	
	Failure of condition (1') can easily be checked with an algorithm. Algorithms to check conditions (2') and (3') exist due to the action being effectively closed and Remark~\ref{remark_ec_BS}. It is a straightforward exercise that the set of pattern codings above defines a subshift $\widetilde{Z}$ which satisfies condition (1) and condition (2) with $h\in \Gamma$ replaced by some $s \in S$. As $\Gamma$ is generated by $S$ this suffices in this case.
\end{proof}

Note that in the second part of Theorem~\ref{thm:seward_new} we do not ask that the group is recursively presented. Let us stress again that effectively closed subshifts (by patterns) coincide up to topological conjugacy with expansive effectively closed actions only if the group is recursively presented, therefore if we wish to interpret our effectively closed subshift extension as an effectively closed expansive action on a closed subset of $\{\symb{0},\symb{1} \}^{\NN}$, we need to ask that the group is recursively presented.

\section{Self-simulable groups}\label{sec:selfsimulable}

The purpose of this section is to prove Theorem~\ref{thm:selfsimulation}. Our main tool shall be the paradoxical subshift as defined in Section~\ref{sec:paradox}. The main idea is quite similar to the application shown on the previous section, namely, we saw that the possibility to associate paths to group elements allows us to encode arbitrary actions in subhifts by putting sequences on the path. By taking a direct product of two finitely generated groups, we will be able to associate not only a path but an actual $\NN^2$-grid to each group element. This will enable us to verify with local rules that a configuration is on some effectively closed set by implementing Turing machines via Wang tilesets.

More precisely, given an effectively closed action $\Gamma \curvearrowright X \subseteq \{\symb{0},\symb{1}\}^{\NN}$, we use a classical embedding of Turing machines through seeded Wang tilings to construct a tiling of $\NN^2$, so that every configuration in which a special symbol occurs at $(0,0)$ encodes in a straightforward manner a configuration on the set representation of $\Gamma \curvearrowright X$. This idea is quite old and can be traced back to~\cite{Wang:1960:PTP:367177.367224,Wang1961} and~\cite{KahrMooreWang62}. These ideas were famously used by Berger to settle the undecidability of the domino problem~\cite{Berger1966}. The grids shall encode a Turing machine which verifies that the configuration indeed belongs to the set representation.

Finally, we shall encode a seeded instance of our tiling of $\NN^2$ on top of the $\NN^2$-grids of a product of two paradoxical subshifts on each group, in such a way that the seed occurs at the origin of every grid. We then endow this object with local rules ensuring that all the $\NN^2$-grids communicate information coherently with the action of $\Gamma$. From here the factor map will be defined naturally using the grids and the set representation.

\subsection{The computation tilespace}

In this section we shall construct a set of tilings of $\NN^2$ which is defined through local rules, and which has the property that the subset of tilings such that a particular type of symbol, called the \define{seed}, occurs at $(0,0)$, maps surjectively into a fixed effectively closed set. 

Let us describe how to code the dynamics of a Turing machine $\mathcal{M} = \{Q,\Sigma, q_0,q_F, \delta \}$ in a seeded tiling of $\NN^2$. Consider the alphabet $\mathcal{W}_{\mathcal{M}}$ given by the square tiles drawn on Figure~\ref{fig:Wang_turingo_no_machina}. Notice that we give special names to three of the tiles on the top row.

\begin{figure}[ht!]
	\centering
	\input{img/Wang_turingo_no_machina}
	\caption{The alphabet $\mathcal{W}_{\mathcal{M}}$ associated to a Turing machine $\mathcal{M} = \{Q,\Sigma, q_0,q_F, \delta \}$. $a \in \Sigma$ is an arbitrary symbol and $q \in Q$ is an arbitrary state. $(s,b),(\ell,c),(r,d) \in Q \times \Sigma$ are any pairs such that $\delta(s,b)=(s',b',0)$, $\delta(\ell,c)=(\ell',c',-1)$ and $\delta(r,d)=(r',d',+1)$.}
	\label{fig:Wang_turingo_no_machina}
\end{figure}
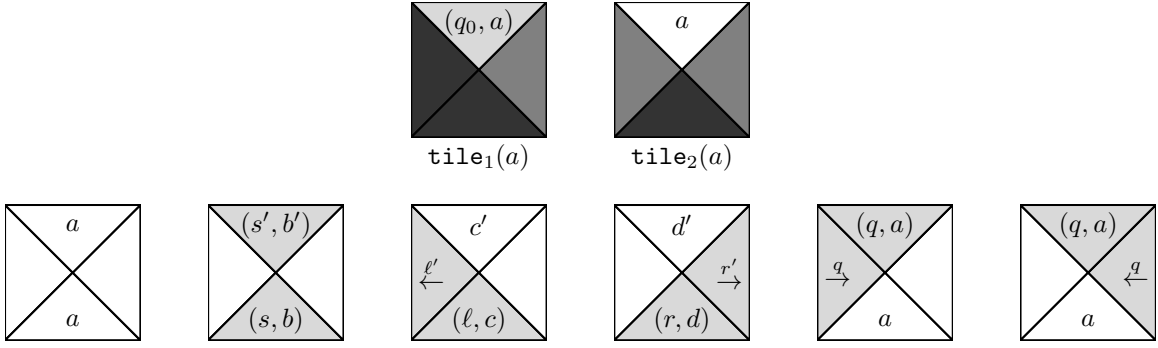

Given a Turing machine $\mathcal{M}$, we define its associated set of $\NN^2$-tilings as the set $\mathcal{T}_{\mathcal{M}}$ of all maps $\tau \colon \NN^2 \to \mathcal{W}_{\mathcal{M}}$ so that whenever two tiles are horizontally or vertically adjacent to each other, both the color and the symbol on their adjacent sides coincide. See Figure~\ref{fig:Turingo_no_exampluru} for an example of pattern satisfying this constraint.

\begin{remark}
	The $\NN^2$-tilings which are described by square tiles satisfying the above kind of constraints are called \define{Wang tilings} after Hao Wang, who introduced them in~\cite{Wang1961}.
\end{remark}

The idea behind this construction is that for any $x \in \Sigma^\NN$, if we provide a horizontal row of tiles of Figure~\ref{fig:Wang_turingo_no_machina} given by \[\texttt{seed}, \texttt{tile}_1(x_0),\texttt{tile}_2(x_1),\texttt{tile}_2(x_2),\texttt{tile}_2(x_3),\dots,\] then the local rules of the tileset force any configuration to encode the dynamics of the Turing machine on the infinite word $x$.

\begin{example}\label{ex:turingo_no_machina}
	Let us consider the following simple Turing machine. Let $Q = \{a,b\}$, $\Sigma = \{0,1,\sqcup\}$, $q_0 = a$ and let $\delta$ be given by
	\begin{align}
		\delta(a,\sqcup) & = (b,0,+1) & \delta(a,0) &= (b,1,+1) & \delta(a,1) &= (a,0,+1) \\
		\delta(b,\sqcup) & = (a,0,-1) & \delta(b,0) &= (b,1,0) & \delta(b,1) &= (a,0,+1).
	\end{align}
	A $\NN^2$-tiling associated to this Turing machine is shown in Figure~\ref{fig:Turingo_no_exampluru}. \qee
\end{example}

\begin{figure}[h!]
	\centering
	\input{img/Turingo_no_exampluru}
	\caption{A pattern of the Wang tiling associated to the machine of Example~\ref{ex:turingo_no_machina}. Each row, starting from the bottom, can be associated to an iteration of the Turing machine on an empty tape.}
	\label{fig:Turingo_no_exampluru}
\end{figure}
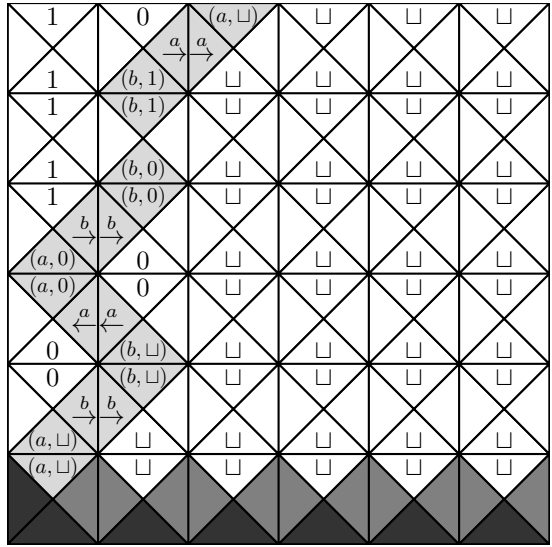

\begin{remark}
	If the tile $\texttt{seed}$ occurs at position $(0,0)$, then necessarily $\texttt{tile}_1(a)$ occurs at $(1,0)$ for some $a \in \Sigma$, and tiles $\texttt{tile}_2(b_n)$ for some $b_n \in \Sigma$ must occur at every position $(n,0)$ for $n \geq 2$. The tiles at the remaining positions of $\NN^2$ are uniquely determined by these tiles.
\end{remark}

Now consider an effectively closed set $Y \subseteq A^{\NN}$. By definition, there exists a Turing machine $\mathcal{M}_{Y}$ so that $\mathcal{M}_{Y}$ accepts a word $w \in A^*$ if and only if $[w]\cap Y = \varnothing$. From this point forward for every effectively closed set $Y$ we fix such a machine $\mathcal{M}_{Y}$.

Let us consider a new Turing machine $\mathcal{M}_{\texttt{search},Y}$ whose tape alphabet and behavior we describe as follows:

\begin{itemize}
	\item The tape alphabet of $\mathcal{M}_{\texttt{search},Y}$ consists of tuples in $A \times \Sigma'$, where $\Sigma' \supset A$ is the working alphabet of an universal Turing machine which contains a special symbol $\sqcup$. The symbols in the first component are not modified by the machine, but can be read.
	\item The dynamics of $\mathcal{M}_{\texttt{search},Y}$ when run on the sequence $(y_n,\sqcup)_{n \in \NN}$ are the following:
	\begin{enumerate}
		\item Define $k = 1$.
		\item For every word $w \in \bigcup_{j \leq k}A^{j}$, run the Turing machine $\mathcal{M}_{Y}$ on input $w$ for $k$ steps.
		\begin{itemize}
			\item If $\mathcal{M}_{Y}$ accepts $w$, check whether $y_0\dots y_{|w|-1} = w$. If it is the case, accept (go to state $q_{F}$).
		\end{itemize}
		\item Redefine $k$ as $k+1$ and go back to step (2).
	\end{enumerate}
\end{itemize}

As $Y$ is effectively closed, it follows that $\mathcal{M}_{\texttt{search},Y}$ run on the sequence $(y_n,\sqcup)_{n \in \NN}$ reaches its final state $q_F$ if and only if $(y_n)_{n \in \NN}\notin Y$.

Given $s \in \mathcal{W}_{\mathcal{M}}$, the set of \define{seeded tilings} of $\mathcal{T}_{\mathcal{M}}$ is the set of $\tau \in \mathcal{T}_{\mathcal{M}}$ such that $\tau(0,0)=s$. 

\begin{definition}
	Let $Y$ be an effectively closed set. The \define{search tilespace} $\mathcal{T}_{\texttt{Search},Y}$ is the set which consists on all seeded tilings of $\mathcal{T}_{\mathcal{M}_{\texttt{Search},Y}}$ such that 
	\begin{enumerate}
		\item The seed is given by the tile $\texttt{seed}$.
		\item We forbid the tiles $\mathtt{tile}_1(a,b)$ and $\mathtt{tile}_2(a,b)$ where $b \neq \sqcup$. 
		\item We forbid any tile carrying the final state $q_F$.
	\end{enumerate}
\end{definition}

In simpler words, we keep only tilings associated to the machine $\mathcal{M}_{\texttt{Search},Y}$ whose bottom row is of the form
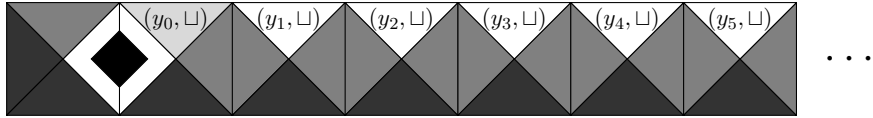
\[
\input{img/Turingo_no_examplurudos}\]

for some $(y_n)_{n \in \NN} \in A^{\NN}$. Furthermore, the third condition imposes that in any such tiling the machine never reaches the accepting state, that is, we will have that $(y_n)_{n \in \NN} \in Y$.

Next we will define an auxiliary tileset which will be useful in the part of proof of Theorem~\ref{thm:selfsimulation} which deals with the synchronization of the $\NN^2$-grids.

\begin{definition}
	Let $A$ be a finite alphabet. The \define{synchronization tilespace} $\mathcal{T}_{\texttt{Sync}}$ is the set of tilings $\sigma \colon \NN^2\to A$ with the property that
	\[ \sigma(u+(0,1)) = \sigma(u+(1,0)) \mbox{ for every } u \in \NN^2.  \]
\end{definition}

In other words, $\mathcal{T}_{\texttt{Sync}}$ consists on all tilings in which the antidiagonals are constant.

Finally, we use the two previous constructions to define the computation tileset $\mathcal{T}_{\texttt{Comp},Y}$ of an effectively closed set $Y$ as follows.

\begin{definition}\label{def.ComputationSub}
	The \define{computation tilespace} $\mathcal{T}_{\texttt{Comp},Y}$ associated to an effectively closed set $Y \subseteq A^{\NN}$ is the set of maps $(\tau,\sigma) \in \mathcal{T}_{\texttt{Search},Y} \times \mathcal{T}_{\texttt{Sync}}$ such that for every $u \in \NN^2$, $i \in \{1,2\}$ and $a \in A$, if $\tau(u+(1,0)) = \texttt{tile}_i(a,\sqcup)$ then $\sigma(u)=a$.
\end{definition}

We shall write down the two important properties of tilings in $\mathcal{T}_{\texttt{Comp},Y}$ in the following two propositions.

\begin{proposition}\label{prop.Comp}
	Let $(\tau,\sigma) \in \mathcal{T}_{\texttt{Comp},Y}$. We have that $(\sigma(0,n))_{n \in \NN} = (\sigma(n,0))_{n \in \NN} \in Y$.
\end{proposition}

\begin{proof}
	Let $(\tau,\sigma) \in \mathcal{T}_{\texttt{Comp},Y}$. From the definition of $\mathcal{T}_{\texttt{Sync}}$ it follows that $\sigma(0,n)=\sigma(n,0)$ for all $n\in\NN$, therefore it suffices to show that $(\sigma(n,0))_{n \in \NN} \in Y$
	
	By definition of $\mathcal{T}_{\texttt{Search},Y}$, we have $\tau(0,0) = \texttt{tile}_1(a_0,\sqcup)$ for some $a_0 \in A$. By the local rules of the tiling we have that for $n\geq 1$, $\tau(n,0)= \texttt{tile}_2(a_n)$ for some $a_n\in A$. Therefore the tiling $\tau$ corresponds to the space time diagram of the computation of $\mathcal{M}_{\texttt{search},Y}$ on the input $(a_n,\sqcup)_{n \in \NN}$. As we removed the symbol $q_F$, it follows that $(a_n)_{n \in \NN} \in Y$. Since $a_n=\sigma(n,0)$ for every $n \in \NN$ by Definition~\ref{def.ComputationSub}, we deduce that $(\sigma(n,0))_{n \in \NN} \in Y$. 
\end{proof}

\begin{proposition}\label{prop.Complete}
	For every $y \in Y$ there exists $(\tau,\sigma) \in \mathcal{T}_{\texttt{Comp},Y}$ so that $y = (\sigma(n,0))_{n \in \NN}$
\end{proposition}

\begin{proof}
	Fix $y \in Y$. Let $\sigma \in A^{\NN^2}$ be given by $\sigma(n-k,k)= y_n$ for all $n\in\NN$ and $k \leq n$. It is clear that $\sigma \in \mathcal{T}_{\texttt{Sync}}$.

	Since $y \in Y$, it follows that $\mathcal{M}_{\texttt{search},Y}$ on $(y_n,\sqcup)_{n \in \NN}$ does not reach the state $q_F$ and thus, if we let $\tau$ be the only configuration in $\mathcal{T}_{\texttt{Search},Y}$ which is compatible with $(y_n,\sqcup)_{n \in \NN}$, then the pair $(\tau,\sigma)$ satisfies the constraint of Definition~\ref{def.ComputationSub} with respect to $\sigma$ and therefore $(\tau,\sigma)\in \mathcal{T}_{\texttt{Comp},Y}$.
\end{proof}

In the next step of the proof, we shall overlay the computation tilings in the grids induced by the paradoxical subshifts in a product of two non-amenable groups. The role of the synchronization tilespace shall be to allow communication between these two components.

\subsection{Proof of Theorem~\ref{thm:selfsimulation}} For the remainder of the section, let $\Gamma_H$, $\Gamma_V$ be two finitely generated non-amenable groups. Fix generating sets $S_H, S_V$ and let $\Gamma = \Gamma_H \times \Gamma_V$ be their direct product.

Let $\Gamma \curvearrowright X \subseteq \{\symb{0},\symb{1}\}^\NN$ be an effectively closed action of $\Gamma$. Consider the following set of generators of $\Gamma$ \[S = \left(S_H \times \{1_{\Gamma_V}\} \right)\cup  \left(\{1_{\Gamma_H}\}\times S_V \right).\]  Let $A = \{\symb{0},\symb{1}\}^S$ and $Y_{\Gamma \curvearrowright X, S} \subseteq A^\NN$ be the set representation of $\Gamma \curvearrowright X$. We recall that $Y_{\Gamma \curvearrowright X, S}$ is an effectively closed subset of $A^\NN$ and that, for $s \in S$ and $y \in Y_{\Gamma \curvearrowright X, S}$, we denote by $\pi_s(y)$ the configuration $(y_n(s))_n \in X$. For simplicity of notation, and coherently with the last section, we shall denote $Y_{\Gamma \curvearrowright X, S}$ simply by $Y$.

By Proposition~\ref{prop:paradox_nonempty}, there exist symmetric finite sets $K_H \Subset \Gamma_H$ and $K_V \Subset \Gamma_V$ which contain the identity and so that the paradoxical subshifts defined on $\Gamma_H$ and $\Gamma_V$ with respect to $K_H$ and $K_V$ are nonempty. Let us denote the corresponding paradoxical subshifts respectively by $\paradox_H$ and $\paradox_V$. Consider their trivial extensions to $\Gamma$ as follows \begin{align*}
	\widetilde{\paradox}_H & = \{ \rho \in (A_{K_H})^{\Gamma} : \left(\rho(g,1_{\Gamma_V})\right)_{g \in \Gamma_H} \in \paradox_H, \\ & \quad \quad \mbox{ and } \rho(g,g') = \rho(g,1_{\Gamma_{V}}) \mbox{ for every } (g,g') \in \Gamma  \},\\
	\widetilde{\paradox}_V & = \{ \rho' \in (A_{K_V})^{\Gamma} : \left(\rho'(1_{\Gamma_H},g)\right)_{g \in \Gamma_V}  \in \paradox_V, \\ & \quad \quad  \mbox{ and } \rho'(g,g') = \rho'(1_{\Gamma_{H}},g') \mbox{ for every } (g,g') \in \Gamma \}.
\end{align*}

It is clear that $\widetilde{\paradox}_H$ and $\widetilde{\paradox}_V$ are both $\Gamma$-subshifts of finite type. Consider the $\Gamma$-subshift of finite type $\widetilde{\paradox} = \widetilde{\paradox}_H \times \widetilde{\paradox}_V$. It follows that for every $\widetilde{g} = (g,g')\in \Gamma$ there exists a continuous map $\gamma_{\widetilde{g}}\colon \NN^2 \times \widetilde{\paradox} \to \Gamma$ given by \[ \gamma_{\widetilde{g}}((n,n'),(\rho,\rho')) = \left(\gamma^H_{g}(n, \rho), \gamma^V_{g'}(n', \rho')  \right).   \]
Where $\gamma^H$ and $\gamma^V$ are the maps defined in Section~\ref{sec:paradox} which are determined by the restrictions of $\rho$ and $\rho'$ to $\Gamma_H \times \{1_{V}\}$ and to $\{1_{\Gamma_H}\}\times \Gamma_V$ respectively.

It follows from~Lemma~\ref{lem:injectivepaths} that for every ${\rho} \in \widetilde{\paradox}$ the map $(u,\widetilde{g})\mapsto \gamma_{\widetilde{g}}(u,{\rho})$ is injective, where $u \in \NN^2$ and $\widetilde{g} \in \Gamma$.

Recall that the computation tilespace $\mathcal{T}_{\texttt{Comp},Y}$ is a subset of all tuples $(\tau,\sigma) \in \mathcal{T}_{\texttt{Search},Y} \times \mathcal{T}_{\texttt{Sync}}$. Denote by $B$ the alphabet of all tiles which occur in some tiling in $\mathcal{T}_{\texttt{Search},Y}$. Also recall that given $\symb{t}\in \{\symb{G},\symb{B} \}$, $\overline{\symb{t}}$ denotes the opposite color. For $\rho \in \widetilde{\paradox}$, write $\rho = ({\rho}_H,{\rho}_V)\in \widetilde{\paradox}_H \times \widetilde{\paradox}_V$. Denote also by $\symb{t}_H(g)$,$\symb{t}_V(g)$ the colors of ${\rho}_H(g)$ and ${\rho}_V(g)$ respectively.

\begin{definition}
	The \define{final} $\Gamma$-subshift $Z_{\texttt{Final}}$ is the set of configurations $z=(\rho,\tau,\sigma) \in \widetilde{\paradox}\times B^{\Gamma} \times A^{\Gamma}$ which satisfy the following constraints:
	\begin{enumerate}
		\item \textbf{Embedding rule:} For every $g \in \Gamma$, $z$ respects the local rules of $\mathcal{T}_{\texttt{Comp},Y}$. More precisely, for every $g \in \Gamma$,
		\begin{enumerate}
			\item $\tau( gL_{\overline{\symb{t}_H(g)}}(\rho_H(g)) ,gL_{\overline{\symb{t}_V(g)}}(\rho_V(g))   ) = \texttt{seed}$.
			\item The right side of the tile $\tau(g)$ coincides with the left side of the tile $\tau(gL_{{\symb{t}_H(g)}}(\rho_H(g)))$.
			\item The upper side of the tile $\tau(g)$ coincides with the bottom side of the tile $\tau(gL_{{\symb{t}_V(g)}}(\rho_V(g)))$.
			\item $\sigma(gL_{{\symb{t}_H(g)}}(\rho_H(g))) = \sigma(gL_{{\symb{t}_V(g)}}(\rho_V(g)))$.
			\item If $\tau(gL_{{\symb{t}_H(g)}}(\rho_H(g))) = \texttt{tile}_i(a,\sqcup)$ for some $a \in A$, then $\sigma(g) = a$.
		\end{enumerate} 
		\item \textbf{Coherence rule:} For every $g \in \Gamma$,
		\begin{enumerate}
			\item  For every $s \in S_H\times  \{1_{\Gamma_V}\}$,  \[\sigma(gL_{\overline{\symb{t}_H}}(\rho_H(g)))(s)=\sigma(gs^{-1}L_{\overline{\symb{t}_H(gs^{-1})}}(\rho_H(gs^{-1})))(1_{\Gamma}).\]
			
			\item For every $s' \in  \{1_{\Gamma_H}\} \times S_V$, \[\sigma(gL_{\overline{\symb{t}_V}}(\rho_V(g)))(s')=\sigma(gs'^{-1}L_{\overline{\symb{t}_V(gs'^{-1})}}(\rho_V(gs'^{-1})))(1_{\Gamma}).\]
		\end{enumerate}
	\end{enumerate}
\end{definition}

The following proposition translates the technical description of $Z_{\texttt{Final}}$ into the properties we shall use to prove our result.

\begin{proposition}\label{claim:sft}
	$Z_{\texttt{Final}}$ is a $\Gamma$-subshift of finite type which satisfies:
	\begin{enumerate}
		\item For every $g \in \Gamma$, 	\[ \left(\tau(\gamma_g(u,\rho)),\sigma(\gamma_g(u,\rho))\right)_{u \in \NN^2}\in \mathcal{T}_{\texttt{Comp},Y}.   \]
		\item For every $g \in \Gamma$, $s\in S$ and $u \in \NN^2$ we have \[ \sigma(\gamma_g(u,\rho))(s) = \sigma(\gamma_{gs^{-1}}(u,\rho))(1_{\Gamma}).  \]
	\end{enumerate}
\end{proposition}

In other words, a configuration $z=(\rho,\tau,\sigma)$ in $Z_{\texttt{Final}}$ satisfies that (1) for every $g \in \Gamma$, the restriction of $(\rho,\tau)$ to the $\NN^2$-grid $\gamma_g(\NN^2,\rho)$ is a valid seeded tiling of $\mathcal{T}_{\texttt{Comp},Y}$, and (2), when moving by a generator $s$ of $\Gamma$, the component $\pi_{s}(y)$ of the element $y \in Y$ read from the grid of $g$ coincides with the component $\pi_{1_{\Gamma}}(y')$ of the element $y'\in Y$ read from the grid at $gs^{-1}$.

\begin{proof}
	From the definition of $Z_{\texttt{Final}}$ it is clear that it is a $\Gamma$-subshift of finite type. Indeed, $\widetilde{\paradox}$ is a $\Gamma$-subshift of finite type and all of the other rules are clearly local. 
	
	Let $z = (\rho,\tau,\sigma) \in Z_{\texttt{Final}}$ and fix $g \in \Gamma$. In order to prove the first property, notice that the element $\gamma_g((0,0),\rho)$ is given by the pair \[ \gamma_g((0,0),\rho) = ( gL_{\overline{\symb{t}_H(g)}}({\rho}_H(g)) ,gL_{\overline{\symb{t}_V(g)}}({\rho}_V(g))   ).    \]
	
	Therefore rule (1a) is stating that $\tau(\gamma_g((0,0),\rho))=\texttt{seed}$. Similarly, the positions $gL_{{\symb{t}_H(g)}}(\rho_H(g))$ and $gL_{{\symb{t}_V(g)}}(\rho_V(g))$ are respectively the ones directly right and up from position $g$, therefore rules (1b) and (1c) implement the local rules of a Wang tileset, namely, that colors must match horizontally and vertically. As we let $B$ be the set of all tiles which occur in $\mathcal{T}_{\texttt{Search},Y}$, these three rules already imply that $\left(\tau(\gamma_g(u,\rho))\right)_{u \in \NN^2}\in \mathcal{T}_{\texttt{Search},Y}. $ 
	
	The same argument as above shows that rule (1d) implements the fact that antidiagonals are constant in $\mathcal{T}_{\texttt{Sync}}$ and (1e) gives the rule that defines $\mathcal{T}_{\texttt{Comp},Y}$. From here it follows that rules (1a)-(1e) imply that \[\left(\tau(\gamma_g(u,\rho)),\sigma(\gamma_g(u,\rho))\right)_{u \in \NN^2}\in \mathcal{T}_{\texttt{Comp},Y}.\]
	
	Let us now show the second property. Let $s \in S$. Given $(n,n') \in \NN^2$, we need to show that \[ \sigma(\gamma_{g}((n,n'),\rho))(s) = \sigma(\gamma_{gs^{-1}}((n,n'),\rho))(1_{\Gamma}).  \]
	
	By definition of $S$, either $s \in S_H \times \{1_{\Gamma_V}\}$ or $s \in \{1_{\Gamma_H}\} \times S_V$. Suppose $s \in S_H \times \{1_{\Gamma_V}\}$. By definition of $\mathcal{T}_{\texttt{Sync}}$ it suffices to show that for every $m \in \NN$, \[ \sigma(\gamma_{g}((0,m),\rho))(s) = \sigma(\gamma_{gs^{-1}}((0,m),\rho))(1_{\Gamma}).  \]
	
	Write $g = (g_1,g_2)$ and note that, \begin{align*}
		\gamma_{g}((0,m),\rho) & = (\gamma^H_{g_1}(0,\rho_H),\gamma^V_{g_2}(m,\rho_V))\\
		& = (g_1L_{\overline{\symb{t}_H(g)}}({\rho}_H(g)),\gamma^V_{g_2}(m,\rho_V))\\
		& = (g_1,\gamma^V_{g_2}(m,\rho_V))L_{\overline{\symb{t}_H(g)}}({\rho}_H(g)).
	\end{align*} and 
	\begin{align*}
		\gamma_{gs^{-1}}((0,m),\rho) & = (\gamma^H_{g_1s^{-1}}(0,\rho_H),\gamma^V_{g_2}(m,\rho_V))\\
		& = (g_1s^{-1}L_{\overline{\symb{t}_H(gs^{-1})}}({\rho}_H(gs^{-1})),\gamma^V_{g_2}(m,\rho_V))\\
		& = (g_1,\gamma^V_{g_2}(m,\rho_V))s^{-1}L_{\overline{\symb{t}_H(gs^{-1})}}({\rho}_H(gs^{-1})).
	\end{align*}
	
	From rule (2a), we deduce that \begin{align*}
		[\sigma(\gamma_{g}((0,m),\rho))](s) & = [\sigma((g_1,\gamma^V_{g_2}(m,\rho_V))L_{\overline{\symb{t}_H(g)}}({\rho}_H(g)))   ](s)\\
		& = [\sigma( (g_1,\gamma^V_{g_2}(m,\rho_V))s^{-1} L_{\overline{\symb{t}_H(gs^{-1})}}({\rho}_H(gs^{-1})))   ](1_\Gamma)\\
		& = [\sigma( \gamma_{gs^{-1}}((0,m),\rho))   ](1_\Gamma).\\
	\end{align*} 
	Which is exactly what we wanted. An analogous argument using property (2b) shows that if $s' \in \{1_{\Gamma_H}\} \times S_V$, then for every $m \in \NN$ \[ \sigma(\gamma_{g}((m,0),\rho))(s') = \sigma(\gamma_{gs'^{-1}}((m,0),\rho))(1_\Gamma).  \]
	Again, by the definition of $\mathcal{T}_{\texttt{Sync}}$ this is enough to obtain what we want.\end{proof}

Finally, consider the map
\[\begin{array}{crcl}
	\phi\colon &Z_\texttt{Final}&\longrightarrow&\{\symb{0},\symb{1}\}^\NN\\
	&(\rho,\tau,\sigma)&\longmapsto & \left(\sigma(\gamma_{1_\Gamma}((n,0),\rho))(1_{\Gamma})\right)_{n\in\NN}
\end{array}\]

We will show that $\phi$ is a topological factor map from $\Gamma \curvearrowright Z_\texttt{Final}$ to $\Gamma \curvearrowright X$. It is clear from the definition that $\phi$ is a continuous map.
\begin{claim}\label{claim:subset}
	$\phi(Z_\texttt{Final}) \subseteq X$.
\end{claim}
\begin{proof}
	Let $(\rho,\tau,\sigma)\in Z_\texttt{Final}$. By the embedding rule, we have that \[(\tau(\gamma_{1_{\Gamma}}(u,\rho)),\sigma(\gamma_{1_{\Gamma}}(u,\rho)))_{u \in \NN^2}\in \mathcal{T}_{\texttt{comp},Y}.\] By Proposition~\ref{prop.Comp}, we deduce that $\left(\sigma(\gamma_{1_\Gamma}((n,0),\rho))\right)_{n\in\NN}\in Y$. Thus $\phi(\rho,\tau,\sigma)\in X$.
\end{proof}

\begin{claim}\label{claim:equiv}
	$\phi$ is $\Gamma$-equivariant.
\end{claim}

\begin{proof}
	Let $s \in S$ and $(\rho,\tau,\sigma)\in Z_\texttt{Final}$. We have
	\begin{align*}
		s\phi(\rho,\tau,\sigma) & = s\left(\sigma(\gamma_{1_\Gamma}((n,0),\rho)(1_{\Gamma})\right)_{n\in\NN}\\
		& = \left(\sigma(\gamma_{1_\Gamma}((n,0),\rho)(s)\right)_{n\in\NN}\\
		& = \left(\sigma(\gamma_{s^{-1}}((n,0),\rho)(1_{\Gamma})\right)_{n\in\NN}\\
		& = \left(\sigma(s^{-1}\gamma_{1_\Gamma}((n,0),s\rho)(1_{\Gamma})\right)_{n\in\NN}\\
		& = \left(s\sigma(\gamma_{1_\Gamma}((n,0),s\rho)(1_{\Gamma})\right)_{n\in\NN}\\
		& = \phi(s\rho,s\tau,s\sigma).
	\end{align*}
	
	Since $S$ is a generating set for $\Gamma$, it follows that $g\phi(\rho,\tau,\sigma) = \phi(g\rho,g\tau,g\sigma)$ for any $g \in \Gamma$ and thus $\phi$ is $\Gamma$-equivariant.\end{proof}

\begin{claim}\label{claim:surj}
	$\phi$ is surjective.
\end{claim}

\begin{proof}
	Fix $\rho \in \widetilde{\paradox}$. Let $\mathcal{R}\subseteq \Gamma$ be the set of \define{reachable} $r \in \Gamma$ for which there is $g \in \Gamma$ and $u \in \NN^2$ such that $r = \gamma_g(u,\rho)$. Since the map $(g,u) \mapsto \gamma_g(u,\rho)$ is injective (by Lemma~\ref{lem:injectivepaths}), we may denote unambiguously these values by $u(r)$ and $g(r)$ for $r \in \mathcal{R}$.
	
	Consider an arbitrary $x = (x_n)_{n \in \NN}\in X$ and let $y = (y_n)_{n \in \NN} \in Y$ such that $y_n(1_{\Gamma})=x_n$ for every $n \in \NN$. Consider the natural action $\Gamma \curvearrowright Y$ such that $(gy)_n(1_{\Gamma}) = (g(y_n(1_{\Gamma}))_{n \in \NN})_n$ for every $y \in Y$ and $g \in \Gamma$. By Proposition~\ref{prop.Complete}, for every $g \in \Gamma$ there exists $(\tau^g,\sigma^g) \in \mathcal{T}_{\texttt{Comp},Y}$ so that $\sigma^g(n,0) = (g^{-1}y)_n$ for every $n \in \NN$. 
	
	Let us fix $a_0 \in A$ and $w_0 \in B$ given by \[\begin{tikzpicture}[scale =1.8]
		\node at (-1,0) {$w_0 = $};
		\begin{scope}[shift = {(0,0)}]
			\clip (-0.5,-0.5) rectangle (+0.5,+0.5);
			\draw[black, fill=white] (-0.5,0.5)--(0,0)--(-0.5,-0.5)--cycle;
			\draw[black, fill=white] (-0.5,-0.5)--(0,0)--(+0.5,-0.5)--cycle;
			\draw[black, fill=white] (0.5,-0.5)--(0,0)--(+0.5,+0.5)--cycle;
			\draw[black, fill=white] (0.5,0.5)--(0,0)--(-0.5,+0.5)--cycle;
			\node at (0,0.35) {\scalebox{0.8}{\textbf{($a_0,\sqcup$)}}};
			\node at (0,-0.35) {\scalebox{0.8}{\textbf{($a_0,\sqcup$)}}};
		\end{scope}	
	\end{tikzpicture}.  \]
	
	Let us define $\tau \in B^{\Gamma}$ and $\sigma \in A^{\Gamma}$ as follows.
	\[  \tau(h) = \begin{cases}
		\tau^{g(h)}(u(h)) & \mbox{ if } h \in \mathcal{R}.\\
		w_0 & \mbox{ if } h \notin \mathcal{R}.
	\end{cases}     \]
	\[  \sigma(h) = \begin{cases}
		\sigma^{g(h)}(u(h)) & \mbox{ if } h \in \mathcal{R}.\\
		a_0 & \mbox{ if } h \notin \mathcal{R}.
	\end{cases}     \] 
	
	By construction, it follows that \[ \phi(\rho,\tau,\sigma) = \left(\sigma(\gamma_{1_\Gamma}((n,0),\rho)(1_{\Gamma})\right)_{n\in\NN} = (\sigma^{1_{\Gamma}}(n,0)(1_{\Gamma}))_{n \in \NN} = (y_n (1_{\Gamma}))_{n \in \NN} = (x_n)_{n \in \NN}.  \]
	It suffices to check that $(\rho,\tau,\sigma) \in Z_\texttt{Final}$. It is straightforward to check that the embedding rule holds. Let us check the coherence rule, let $g \in G, s \in S$ and $(n,n') \in \NN^2$. On the one hand we have 
	\[ \sigma(\gamma_g((n,n'),\rho)) = \sigma^g(n,n') = \sigma^g(n+n',0) = (g^{-1}y)_{n+n'}.   \]
	On the other hand we have \[ \sigma(\gamma_{gs^{-1}}((n,n'),\rho)) = \sigma^{gs^{-1}}(n,n') = \sigma^{gs^{-1}}(n+n',0) = (sg^{-1}y)_{n+n'}.  \]
	We obtain that \[ \sigma(\gamma_g((n,n'),\rho))(s) = (g^{-1}y)_{n+n'}(s) = (sg^{-1}y)_{n+n'}(1_{\Gamma}) = \sigma(\gamma_{gs^{-1}}((n,n'),\rho))(1_{\Gamma}). \]
	Which is the coherence rule. It follows that $(\rho,\tau,\sigma) \in Z_\texttt{Final}$ and thus $\phi$ is surjective.\end{proof}

Putting together Claims~\ref{claim:subset},~\ref{claim:equiv} and~\ref{claim:surj} we have that $\phi \colon Z_{\texttt{Final}} \to X$ is a topological factor map. By Claim~\ref{claim:sft} $Z_{\texttt{Final}}$ is a $\Gamma$-SFT. This completes the proof of Theorem~\ref{thm:selfsimulation}.

\section{Extensions of self-simulable groups}\label{sec:stability_properties}

The goal of this section is to study conditions under which a group $\Gamma$, which contains in some form a self-simulable group $\Delta$, is self-simulable. Our main result is that if $\Gamma$ is recursively presented and the ``adjacent'' cosets of a subgroup $\Delta$ under a suitable generating set $T$ of $\Gamma$ have ``enough space to communicate'', then $\Gamma$ is also self-simulable. Intuitively, in order to communicate an entire configuration between elements $g$ and $gt$ of $\Gamma$ we need that the cosets $g\Delta$ and $gt\Delta$ are ``close'' in infinitely many places. Our endeavor will be to find algebraic conditions which ensure that the former property holds. Let us mention that normality of $\Delta$ is enough, and we shall indeed give a less restrictive property which we call ``mediation''. See Definition~\ref{def:mediation}.

\begin{example}\label{ex:trivialite}
	Let us notice that the sole property of $\Gamma$ admitting a self-simulable subgroup is not enough for $\Gamma$ to be self-simulable. For instance, if $\Gamma = (F_2\times F_2)\ast F_2$, where $\ast$ denotes the free product, we have that $\Gamma$ admits the self-simulable subgroup $F_2\times F_2$, however, by Proposition~\ref{prop:infends} it is not self-simulable because it has infinitely many ends.\qee
\end{example}

As in the previous section, we shall encode elements of $A^{\NN}$ on infinite paths, and use copies of the paradoxical subshift to communicate these elements along different paths. Therefore, it will be natural to assume that $\Delta$ admits a covering by paths such that for the path $p\colon \NN \to \Delta$ starting at $h \in \Delta$, $n \mapsto p(n)t$ gives a path in $ht\Delta$. We shall call those coverings $t$-media.

In Section~\ref{sec:paradox}, we avoided the discussion of some technicalities associated to the paradoxical subshift $\paradox$ as they were not necessary in the proof of Theorem~\ref{thm:selfsimulation}. In this section, we shall require the existence of configurations in the paradoxical subshift with special properties. To this end, we shall make the correspondence between elements of $\paradox$ and path covers slightly more explicit.

Recall that given a color $\symb{t}\in \{\symb{G},\symb{B}\}$, we denote by $\overline{\symb{t}}$ the other color.

\begin{definition}
	Let $\Delta$ be a group. A \define{path cover} with moves $K\Subset \Delta$ is a collection of maps $(p_h)_{h \in \Delta}$ called \define{paths}, where $p_h \colon \NN \to \Delta$ is injective, satisfies $p_h(0) = h$ and $p_h(n)^{-1} p_h(n+1) \in K$ for all $h \in \Delta, n \in \NN$, and every $h \in \Delta$ is in the range of at most two paths.
	
	We say a path cover is \define{full} if every $h \in \Delta$ is in the range of two paths, and \define{colored} if it comes with a partition $\Delta$ into $\Delta_{\symb{G}}, \Delta_{\symb{B}}$ such that whenever $h \in \Delta_{\symb{t}}$ for $\symb{t}\in \{\symb{G},\symb{B}\}$, then for every $n \in \NN$ we have that $p_h(n+1) \in \Delta_{\overline{\symb{t}}}$.
\end{definition}

By definition, $p_h$ always has $h$ in its range, so the condition that $h$ is in the range of at most (resp.\ exactly) two paths says precisely that $h$ appears as an intermediate node in at most (resp.\ exactly) one path. Notice that if $\rho \in \paradox$ is a configuration in the Paradoxical subshift associated to some $K \Subset \Delta$, and $(\gamma_h)_{h \in \Delta}$ are the maps defined in Section~\ref{sec:paradox}, then the collection of maps $(p_h)_{h \in \Delta}$ given by \[ p_h(0)=h \mbox{ and } p_h(n)= \gamma_h(n-1,\rho) \mbox{ for every } n \geq 1,  \]
is a colored path cover with moves in $K$ which we call the \define{path cover induced by $\rho$}. It might not necessarily be a full path cover.

\begin{remark}
	If there exists a $2$-to-$1$ map $\varphi \colon \Delta \to \Delta$ such that for every $g \in \Delta$, $g^{-1}\varphi(g)\in K \Subset \Delta$ with $K$ symmetric, then there is a $2$-to-$1$ map $\widetilde{\varphi} \colon \Delta \to \Delta$ such that for every $g \in \Delta$, $g^{-1}\widetilde{\varphi}(g)\in K^2 \setminus \{1_{\Delta}\}$. Indeed, let $L\subseteq \Delta$ be the set of $g \in \Delta$ for which $\varphi(g)=g$. As the map is $2$-to-$1$, for every $g \in L$ we can choose elements $g',g''$ such that $\{g,g',g''\}$ are all distinct, $\varphi(g')=g$ and $\varphi(g'')=g'$. We can then define $\widetilde{\varphi}$ as follows \[\widetilde{\varphi}(h) = \begin{cases}
		g' & \mbox{ if } g \in L, \\
		g & \mbox{ if } h=g'' \mbox{ for some } g \in L, \\
		\varphi(h) & \mbox{ otherwise. }
	\end{cases}   \]
	Clearly $g^{-1}\widetilde{\varphi}(g)\in K^2 \setminus \{1_{\Delta}\}$. In other words, by choosing a suitable symmetric set we may always assume that the $2$-to-$1$ map has no fixed points and thus we may remove the identity from the set of moves.
\end{remark}

Given a configuration $\rho \in \paradox$, we may extract a $2$-to-$1$ map $\varphi_{\rho} \colon \Delta \to \Delta$ and a color function $\symb{C}_{\rho}\colon \Gamma \to \{ \symb{G},\symb{B} \}$ by letting $\varphi_{\rho}(h) = hR_{\symb{t}}(\rho(h))$ and $\symb{C}_{\rho}(h) = \symb{t}(\rho(h))$ for every $h \in \Delta$.

\begin{definition}
	Let $\paradox$ be a paradoxical subshift in $\Delta$ and let $\rho\in \paradox$. Let $\varphi_{\rho}$ and $\symb{C}_{\rho}$ as above. We say $\rho$ is \define{well-founded} if for every $h \in \Delta$ there is $n \geq 1$ such that $\symb{C}_{\rho}(\varphi_{\rho}^n(h)) \neq \symb{C}_{\rho}(h)$.
\end{definition}

In simpler words, a configuration $\rho \in \paradox$ is well-founded if for every element of $\Delta$, following the function $\varphi_{\rho}$ eventually leads to a distinct color. In particular, this means that there are no loops, monochromatic cycles or monochromatic infinite paths. Clearly, if $\rho$ is well-founded, then the path cover induced by $\rho$ is colored and full.

\begin{lemma}
	\label{lem:WellFounded}
	Let $\paradox$ be the paradoxical subshift in $\Delta$ defined by some symmetric set $K \not\ni 1_{\Delta}$. If $\paradox \neq \emptyset$, then it contains a well-founded configuration.
\end{lemma}

\begin{proof}

	Let $\rho \in \paradox$. We say that $h \in \Delta$ is in a monochromatic cycle of $\rho$ if there is a smallest $n \geq 1$ such that $\varphi_{\rho}^n(h)=h$ and for every $0 \leq i \leq n-1$ we have $\symb{C}_{\rho}(\varphi_{\rho}^i(h)) = \symb{C}_{\rho}(h)$. Notice that we always have $n \geq 2$, as $1_{\Delta} \not\in K$. We first show that we can find a configuration with no monochromatic cycles. Indeed, suppose there is $h \in \Delta$ in a monochromatic cycle. Let $h' \neq h$ such that $\varphi_{\rho}(h')=\varphi_{\rho}(h)$ and swap the colors of $h'$ and $h$. That is, modify the color function $\symb{C}_{\rho}$ to a function $\widetilde{\symb{C}}$ such that $\widetilde{\symb{C}}(h')=\symb{C}_{\rho}(h)$ and $\widetilde{\symb{C}}(h)=\symb{C}_{\rho}(h')$ and is equal to $\symb{C}_{\rho}$ everywhere else. As $n \geq 2$ it follows that $h$ is no longer in a monochromatic cycle, and no new monochromatic cycles were constructed. 
	
	Given $\rho\in \paradox$ and $h \in \Delta$, let $\rho_{h}$ be the configuration such that $\rho_h = \rho$ if $h$ is not in a monochromatic cycle, and such that $\rho_h$ is the configuration with the colors of the incoming edges of $\varphi_{\rho}(h)$ swapped as described above if $h$ is in a monochromatic cycle. Clearly $\rho_h \in \paradox$. Letting $(h_k)_{k \in \NN}$ be an enumeration of $\Delta$, $\rho^{(0)}=\rho$ and $\rho^{(k)} = (\rho^{(k-1)})_{h_{k-1}}$ for $k \geq 1$, it follows that none of the elements $h_0,\dots,h_{k-1}$ are in monochromatic cycles in $\rho^{(k)}$. Any accumulation point of the sequence $(\rho^{(k)})_{k \in \NN}$ gives a configuration in $\paradox$ where no $h \in \Delta$ is in a monochromatic cycle.
	
	Notice that the set $\paradox^{\otimes}$ of all $\rho \in \paradox$ where no $h \in \Delta$ is in a monochromatic cycle is closed and nonempty. Therefore, in order to obtain a well-founded configuration, it suffices to show that, for every $h \in \Delta$, the set of $\rho \in \paradox^{\otimes}$ for which there is $n \geq 1$ such that $\symb{C}_{\rho}(\varphi_{\rho}^n(h)) \neq \symb{C}_{\rho}(h)$ is a dense open set. Indeed, as $\Delta$ is countable, by the Baire category theorem we would have that the intersection of these dense open sets would be a dense $G_\delta$ set, therefore nonempty.
	
	Fix $h \in \Delta$ and denote by $U_h$ the set of all $\rho \in \paradox^{\otimes}$ such that there is $n \geq 1$ such that $\symb{C}_{\rho}(\varphi_{\rho}^n(h)) \neq \symb{C}_{\rho}(h)$. It is clear that $U_h$ is an open set in $\paradox^{\otimes}$ (it is a union of cylinder sets). Let us show that $U_h$ is dense. Let $\rho \in \paradox^{\otimes}$ and $F \Subset \Delta$. If $\rho \in U_h$ we are done, otherwise, we have that $\symb{C}_{\rho}(\varphi_{\rho}^n(h)) = \symb{C}_{\rho}(h)$ for every $n \geq 1$ and as there are no monochromatic cycles in $\rho$, we conclude that the elements $\varphi_{\rho}^n(h)$ are all distinct. Therefore there is $\bar{n} \in \NN$ such that for every $s \in F$, $\varphi_{\rho}(s) \neq \varphi_{\rho}^{\bar{n}}(h)$.
	
	Let $u \in \Delta$ such that $\varphi(u) = \varphi^{\bar{n}}(h)$ and $u \neq \varphi^{\bar{n}-1}(h)$. By the previous argument we have $u \notin F$. Let $\rho^*$ be the configuration where the colors of $\varphi^{\bar{n}-1}(h)$ and $u$ are swapped. As the forward path from $\varphi^{\bar{n}}(h)$ is monochromatic and infinite, it follows that no monochromatic cycles occur in $\rho^*$, thus $\rho^* \in \paradox^{\otimes}$. Furthermore, $C_{\rho^*}(\varphi_{\rho^*}^{\bar{n}-1}(h)) \neq C_{\rho^*}(h)$, therefore $\rho^* \in U_h$. Finally, as no modifications were introduced in $F$, we have that $\rho^* \in [\rho|_F]$. As $\rho$ and $F$ were arbitrary, this shows that $U_h$ is dense in $\paradox^{\otimes}$.\end{proof}

\begin{lemma}
	\label{lem:Connection}
	Let $\Delta$ be a group and $K \Subset \Delta$ a symmetric set with $K \not\ni 1$.
	\begin{itemize}
		\item If the paradoxical subshift $\paradox$ is constructed with moves $K$, then $\paradox$ is nonempty if and only if there is a full and colored path cover of $\Delta$ with moves $K$.
		\item If there is a path cover of $\Delta$ with moves $K$, then there is a full and colored path cover of $\Delta$ with moves $K^2 \setminus \{1_{\Delta}\}$.
	\end{itemize}
\end{lemma}

\begin{proof}
	We begin with the first item. Suppose $\paradox$ is nonempty. By Lemma~\ref{lem:WellFounded}, we can find a well-founded configuration $\rho \in \paradox$. As discussed before, it is clear the path cover of $\Delta$ induced by $\rho$ is full, colored and uses moves in $K$. Conversely, if $\Delta$ has a full and colored path cover with moves $K$, then it gives a (well-founded) configuration of $\paradox$ in an obvious way: for every $h \in \Delta$ of color $\symb{t}$ we have $L_{\symb{t}}(h)= h^{-1}p_h(1)$, and for $n \geq 2$ we have  $L_{\overline{\symb{t}}}(p_h(n-1)) = p_h(n-1)^{-1}p_h(n)$. Notice that this also determines the value of $R_{\symb{t}}$ uniquely.
	
	For the second item, from a path cover with moves $K$ we can construct a full and colored path cover with the moves $K^2 \setminus \{1_{\Delta}\}$ as follows: First, to ensure the path cover is full, we make sure that every node $h'$ is an intermediate node of one of the paths $p_h$ for some $h \in \Delta$. To do this, whenever $h'$ is not already an intermediate node of any path, pick $h\in \Delta$ such that $h' \in hK$ and change the path $p_h$ into a path $\widetilde{p}_h$ such that $\widetilde{p}_h(1)=h'$ and $\widetilde{p}_h(n)=p_h(n-1)$ for $n \geq 2$. Notice that as long as $h^{-1}p(1) \in K$ and the rest of the moves of $p_h$ are in $K^2 \setminus \{1_{\Delta}\}$, then $h^{-1}\widetilde{p}_h(1) = h^{-1}h' \in K$ and the rest of the moves of $\widetilde{p}_h$ are in $K^2 \setminus \{1_{\Delta}\}$. Notice a fixed $h\in \Delta$ can be chosen for at most $|K|$ elements of $\Delta$, therefore the path $p_h$ gets modified finitely many times. Therefore if we iterate this process over all $h' \in \Delta$ which are not an intermediate node of any path, we end up with a full path cover.
	
	From a full path cover with moves in $K^2 \setminus \{1\}$ it is easy to obtain a coloring of $\Delta$ which is consistent with the path cover through a compactness argument. \end{proof}

Now we are ready to give the algebraic definition of the class of groups on which our extension scheme applies.

\begin{definition}\label{def:mediation}
	A subgroup $\Delta \leqslant \Gamma$ is called \define{mediated} if there exists a finite generating set $T$ of $\Gamma$ such that for every $t \in T$ the subgroup $\Delta \cap t\Delta t^{-1}$ is nonamenable.
\end{definition}

This algebraic definition is justified in the following definition and proposition.

\begin{definition}
	Let $\Delta \leqslant \Gamma$ and $t \in \Gamma$. A \define{$t$-medium} (on $\Delta$) is a path cover $(p_h)_{h \in \Delta}$ with moves $K \Subset \Delta$ such that $t^{-1}Kt \subseteq \Delta$.
\end{definition}

\begin{proposition}
	A subgroup $\Delta \leqslant \Gamma$ is mediated if and only if there is a finite generating set $T$ for $\Gamma$ such that for each $t \in T$, there exists a $t$-medium on $\Delta$.
\end{proposition}

\begin{proof}
	Let $t\in \Gamma$. We will show that there exists a $t$-medium on $\Delta$ if and only if $\Delta \cap t\Delta t^{-1}$ is nonamenable. The result follows directly from this claim.
	
	Suppose that $\Sigma = \Delta \cap t\Delta t^{-1}$ is nonamenable, then it admits a path cover with moves on some set $K \Subset \Sigma$, which in turn induces a path cover of $\Delta$ with the same moves. Furthermore, as $\Sigma = \Delta \cap t\Delta t^{-1}$ we have $K\Subset \Delta$ and $t^{-1}Kt \subseteq t^{-1}\Sigma t \subseteq \Delta$, and thus we have a $t$-medium on $\Delta$.
	
	Conversely, suppose there exists a $t$-medium on $\Gamma$ with moves $K\Subset \Delta$. Let $H = \langle K \rangle$ be the subgroup of $\Delta$ generated by $K$. As $t^{-1}Kt \Subset \Delta$, we have that $K\Subset t\Delta t^{-1}$ and thus $H \leqslant \Delta \cap t\Delta t^{-1}$. As the moves are in $K$, the restriction of the path cover on $\Delta$ to $H$ is still path cover and thus $H$ is nonamenable, which implies that  $\Delta \cap t\Delta t^{-1}$ is nonamenable.
\end{proof}

\begin{remark}\label{rem:sigmas}
	The condition that $\Delta \cap t\Delta t^{-1}$ is non-amenable is of course equivalent to the condition that there exists a non-amenable subgroup $\Sigma \leqslant \Delta \cap t\Delta t^{-1}$. In most of our applications
	$\Delta \cap t\Delta t^{-1}$ will not be so easy to understand, but the
	assumptions will always provide a natural choice for a subgroup $\Sigma$ for which we
	can check that $\Sigma \leqslant \Delta$ and $t^{-1} \Sigma t \leqslant \Delta$.	
\end{remark}

\begin{remark}
	\label{rem:FillingAndColoring}
	Note that $\Delta \cap t \Delta t^{-1}$ is a subgroup of $\Gamma$, so applying Lemma~\ref{lem:Connection} yields that if $\Delta$ admits a $t$-medium with a symmetric set of moves $K\setminus \{1_{\Delta}\}$, then $\Delta$ admits a full and colored $t$-medium with moves $K^2\setminus \{1_{\Delta}\}$. 
\end{remark}

The main idea behind the definition of $t$-media is the following: suppose $(p_g)_{g \in \Gamma}$ is a path cover of $\Gamma$ with moves $K$, and $t \in \Gamma$, then we can define $q_g \colon \NN \to \Gamma$ by $q_g(n) = p_{gt^{-1}}(n) t$ for every $g \in \Gamma$, and call $(q_g)_{g \in \Gamma}$ the $t$-\define{conjugate path cover} of $(p_g)_{g \in \Gamma}$. This is indeed a path cover with moves in $t^{-1} K t$, since we have just shifted the path cover by a right translation.

Furthermore, if $(p_g)_{g \in \Gamma}$ is a path cover of $\Gamma$ with moves in $K \Subset \Delta$, then $(p_g)_{g \in \Gamma}$ naturally specifies a path cover of $\Delta$ on each coset of $\Delta$: picking representatives $\Gamma = \bigcup_{g \in R} gH$, for each $g \in R$ we obtain a path cover by $q_h(n) = g^{-1} p_{gh}(n)$. Note that the choice of representatives only translates the path covers, and in particular does not affect mediation.

Now we are ready to prove Theorem~\ref{thm:mediaintroduction}. That is, that if $\Gamma$ is a finitely generated and recursively presented group which contains a mediated self-simulable subgroup $\Delta$, then $\Gamma$ is self-simulable.

\begin{proof}[Proof of Theorem~\ref{thm:mediaintroduction}]
	Let $\Gamma \curvearrowright X$ be an effectively closed action and let $\Delta \leqslant \Gamma$ be the mediated self-simulable subgroup. Without loss of generality, suppose that $X \subseteq \{\symb{0},\symb{1}\}^{\NN}$. Let $T \Subset \Gamma$ be a (not necessarily symmetric) finite generating set, and let $\overline{T} = \{\bar t : t \in T\}$ be a disjoint set (one can think of it as a set of formal inverses). As $\Delta$ is mediated, we can associate to each $t \in T$ a symmetric set $K_t \Subset \Delta \cap t \Delta t^{-1}$ such that $1_{\Gamma}\notin K_t$ and there is a $t$-medium for $\Delta$ with moves $K_t$. Let $K_{\bar t} = t^{-1}(K_t) t$.
	
	For each $t \in T \cup \overline{T}$, denote by $\paradox_t$ the paradoxical subshift on $\Gamma$ constructed with moves $M_t = K_t^2 \setminus \{1_{\Gamma}\}$ (using the alphabet $(M_t)^3 \times \{\symb{G}, \symb{B}\}$). Let us define the $\Gamma$-subshift,
	\[ \mathbf{W} = \prod_{t \in T \cup \overline{T}} (\paradox_t \times \{\symb{0},\symb{1}\}^\Gamma). \]
	For every $t \in T \cup \overline{T}$, write the projection to the $t$-th layer as $\pi_t \colon \mathbf{W} \to \paradox_t \times \{\symb{0},\symb{1}\}^{\Gamma}$. For $w \in \mathbf{W}$, if $\pi_t(w) = (\rho,z)\in \paradox_t \times \{\symb{0},\symb{1}\}^{\Gamma}$, let $\gamma_{g,t}(n,w)$ denote the function $\gamma_g(n,\rho)$ and $\beta_{g,t}(n,\rho) = z(\gamma_{g,t}(n,w))$.
	
	Next, we shall define a $\Gamma$-subshift $\mathbf{Y}$ by imposing additional constraints on configurations $w \in \mathbf{W}$.
	
	\begin{itemize}
		\item \textbf{Constraint 1: }For each $g \in \Gamma$, there exists $x \in X$, such that for every $t \in T$ we have $(\beta_{g,t}(n,w))_{n \in \NN} = x$. We call $x$ the \define{simulated configuration at $g$}.
		\item \textbf{Constraint 2: }For each $g \in \Gamma$, if the simulated configuration at $g$ is $x$, then for every $t \in T$ we have that $(\beta_{g,\bar t}(n,w))_{n \in \NN} =  tx$.
		\item \textbf{Constraint 3: }For all $t \in T$ and $g \in \Gamma$, if $\pi_t(w)(g) = ((k_1,k_2,k_3,\symb{c}), b)$, then $\pi_{\overline{t}}(w)(gt) = ((t^{-1}k_1 t,t^{-1} k_2 t,t^{-1} k_3 t,\symb{c}),b)$.
	\end{itemize}
	We argue that $\mathbf{Y}$ is a sofic $\Gamma$-subshift. To see this, notice that the first two constraints can be defined using a recursively enumerable list of forbidden pattern codings whose support is contained in $\Delta$ (because $M_t \subseteq \Delta$ for all $t \in T \cup \overline{T}$). In other words, the configurations in $\mathbf{W}$ which satisfy the first two constraints consist on copies on each coset of $\Delta$ of configurations of an effectively closed subshift (by patterns) on $\Delta$. As $\Gamma$ is recursively presented, so is $\Delta$, and this effectively closed subshift is topologically conjugate to an effectively closed action of $\Delta$ by Proposition~\ref{prop:conjugacy_rec_presented_effsubshift}. As $\Delta$ is self-simulable, it follows that there is a subshift of finite type on $\Delta$ which factors onto this subshift and thus can be extended to a $\Gamma$-subshift of finite type which factors onto the set of configurations of $\mathbf{W}$ which satisfy the first two constraints. The last constraint is a local rule, therefore we conclude that $\mathbf{Y}$ is a sofic $\Gamma$-subshift. 
	
	Therefore, it suffices to show that $\mathbf{Y}$ factors onto $\Gamma \curvearrowright X$. Fix $t \in T$ and let $\phi \colon \mathbf{Y} \to X$ be the map defined by \[\phi(w)_n = \beta_{1_{\Gamma},t}(n,w) \mbox{ for every } w \in \mathbf{Y}.\] The map $\phi$ is obviously continuous, and the codomain is $X$ by the first constraint. We need to show that $\phi$ is $\Gamma$-equivariant and surjective. 
	
	First we show that $\phi$ is $\Gamma$-equivariant. Since $T$ is a generating set, so is $T^{-1}$ and thus it suffices to show that $\phi(t^{-1}w)=t^{-1}\phi(w)$ for every $t \in T$ and $w \in \mathbf{Y}$. 
	
	Fix $w \in \mathbf{Y}$ and let $x =\phi(w)\in X$. Let us denote for $n \in \NN$, $p(n+1)= \gamma_{1_{\Gamma},t}(n,w)$ and $p(0)= 1_{\Gamma}$. Suppose that the color at $\pi_t(w)(1_{\Gamma})$ is blue (the other case is symmetric), and denote by $\dont$ any symbol. Then we have \[ \pi_t(w)(p(0))=((p(1),\dont,\dont,\symb{B}),\dont), \] and \[
	\pi_t(w)(p(n))=((p(n)^{-1}p(n+1),\dont,p(n)^{-1}p(n-1),\symb{G}),x_{n-1}) \mbox{ for every } n \geq 1. \]

	By the third condition we have \[  \pi_{\bar t}(w)(p(0)t) = (((t^{-1}p(1)t, \dont, \dont, \symb{B}), \dont),  \]
	and
	\[ \pi_{\bar t}(w)(p(n)t) = ((t^{-1}(p(n)^{-1}p(n+1))t,\dont,t^{-1}(p(n)^{-1}p(n-1))t, \symb{G}), x_{n-1}). \]
	A direct calculation shows then that $\gamma_{t,\bar t}(n,w) = p(n) t$ for every $n \in \NN$, and thus $(\beta_{t,\bar{t}}(n,w))_{n \in \NN} = x$. By the second constraint, we obtain that the configuration simulated at $t$ is $t^{-1}x$. Therefore, $\phi(t^{-1} w) = t^{-1} x$ as required.
	
	Finally, we show surjectivity. Let $x \in X$ be arbitrary. For $t \in T$, pick a $t$-medium for $\Delta$ with moves $K_t$. By Remark~\ref{rem:FillingAndColoring}, there is a full and colored $t$-medium with moves $M_t = K_t^2 \setminus \{1\}$, and by Lemma~\ref{lem:Connection} the subshift $\paradox_t$ is nonempty. Construct a full and colored $t$-medium on $\Gamma$ with moves $M_t$ by using this path cover on every $\Delta$-coset. Take this as the $\paradox_t$-component of $w$ for each $t \in T$. For $\bar t \in \overline{T}$, use the third constraint to obtain the contents of the $\paradox_{\bar t}$-component. This is valid because the moves on that layer are in $t^{-1} M_t t =  M_{\bar t}$. Postulating that the simulated configuration at $g$ is $g^{-1}x$ specifies all the $\{\symb{0},\symb{1}\}$ symbols on all layers because the configuration is well-founded, From here it is clear that none of the constraints defining $\mathbf{Y}$ are broken.
\end{proof}

Since by Proposition~\ref{prop:SelfSimuImpNonAmenable} a recursively presented self-simulable group is non-amenable, we can always choose $\Sigma = \Delta$ when $\Delta$ is normal, giving the following handy version of the previous result.

\begin{corollary}\label{cor:normalSS}
	Suppose $\Gamma$ is finitely generated, recursively presented and has a normal self-simulable subgroup $\Delta$. Then $\Gamma$ is self-simulable.
\end{corollary}

In other words, whenever we have a short exact sequence of the form \[ 1 \to N \to \Gamma \to H \to 1, \]
where $N$ is self-simulable and $\Gamma$ is finitely generated and recursively presented, then $\Gamma$ is also self-simulable. It also follows that if $N$ is self-simulable then $N \times H$ is, for any finitely generated and recursively presented group $H$. 

\begin{example}
	Take $\Gamma = F_4$ the free group on $4$ generators. This group is not self-simulable due to it having infinitely many ends (Proposition~\ref{prop:infends}). However, it has as a quotient the self-simulable group $F_2 \times F_2$, therefore we have a short exact sequence of the form \[1 \to N \to F_4 \to F_2 \times F_2 \to 1.    \]
	Thus we conclude that admitting a self-simulable quotient does not imply that the group is self-simulable.\qee
\end{example}

We do not have a counterexample for split extensions, and under some conditions on the quotient and its action, extensions on the right are indeed self-simulable. 

\begin{corollary}
	Suppose $\Delta$ is a recursively presented and self-simulable group, and contains a direct product of two non-amenable groups. If $\Gamma$ is finitely generated and recursively presented and $\phi\colon \Delta \to \Aut(\Gamma)$ has amenable image, then $\Gamma \rtimes_\phi \Delta$ is self-simulable.
\end{corollary}

\begin{proof}
	Observe that the fact that $\Gamma \rtimes_\phi \Delta$ is finitely generated and recursively presented follows from the assumptions. We show that the natural copy $\Delta \leqslant \Gamma \rtimes_\phi \Delta$ is mediated. Let $\Delta_1 \times \Delta_2 \leqslant \Delta$ be a direct product of non-amenable groups, and consider the action of $\Delta_i$ on $\Gamma$. We have an exact sequence
	\[ 1 \to K_i \to \Delta_i \to \phi(\Delta_i) \to 1, \]
	where $K_i$ is the kernel of $\phi$ restricted to $\Delta_i$. By the amenability of $\phi(\Delta_i)\subseteq \Aut(\Gamma)$ and non-amenability of $\Delta_i$ we have that $K_i$ is non-amenable. Then $\Sigma = K_1 \times K_2 \leqslant \Delta$ acts trivially by conjugation on $\Gamma$, i.e.\ $g^{(k_1, k_2)} = g$, in particular $g^{-1} \Sigma g = \Sigma$ for all $g \in \Gamma$ and of course we have $g^{-1} \Sigma g \leqslant \Delta$ for any $g \in \Delta$, thus this non-amenable subgroup satisfies the mediated condition.
\end{proof}

In particular, if $\Aut(\Gamma)$ is amenable, then $\Gamma \rtimes \Delta$ is self-simulable whenever $\Delta$ is self-simulable and contains a direct product of two non-amenable groups.

\begin{remark}
	In the proof of Theorem~\ref{thm:mediaintroduction} the condition that $\Delta$ is self-simulable can be weakened. What we essentially use is the weaker property that for any effectively closed $\Delta$-subshift $X$, the subshift $X^{\Gamma/\Delta} = \{x \in A^\Gamma : \mbox{ for every } g \in \Gamma, gx|_{\Delta} \in X\}$, where an independent copy of $X$ appears on each $\Delta$-coset, is $\Gamma$-sofic. Note that if $\Gamma$ is self-simulable and recursively presented, subshifts of the form $X^{\Gamma/\Delta}$ are always $\Gamma$-sofic for any finitely-generated subgroup $\Delta$ and $\Delta$-effectively closed subshift $X$, so a subgroup $\Delta$ can certainly have this property without being self-simulable (every finitely-generated group is a subgroup of a self-simulable group). We do not have any examples where this additional generality is needed.
\end{remark}

\section{Rigid classes of self-simulable groups}\label{sec:QIrigidity}

The goal of this section is to prove Theorems~\ref{thm:commensurable} and~\ref{thm:quasiisometry}. Recall that two groups $\Gamma_1$ and $\Gamma_2$ are called \define{commensurable} if there exist finite index subgroups $H_1 \leqslant \Gamma_1$ and $H_2 \leqslant \Gamma_2$ which are isomorphic.

\begin{lemma}\label{lem:normalfiniteindex_isSS}
	Let $\Gamma$ be a finitely generated self-simulable group and suppose $N \leqslant \Gamma$ is a finite index subgroup. Then $N$ is self-simulable.
\end{lemma}

\begin{proof}
	Let $N\curvearrowright X$ be an effectively closed action. As $N$ has finite index in $\Gamma$, we can find a representative of the left cosets $L = \{\ell_0 = 1_{\Gamma},\ell_1,\dots,\ell_k\}$ of $N$ in $\Gamma$. Every $g \in \Gamma$ can be written uniquely as a product $g=\ell h$ with $\ell \in L$ and $h \in N$.
	
	Consider the space $X[L]$ of formal sums $\ell_0 x_0 + \ell_1 x_1 + \dots + \ell_k x_k$ with $x_i \in X$ for $0 \leq i \leq k$. Notice that if $X \subseteq A^{\NN}$, then we can formally identify $X[L]$ with a subset of $(A^L)^{\NN}$, we shall use this space to obtain the classical action of $\Gamma$ induced from a subgroup of finite index. Indeed, we define $\Gamma\curvearrowright X[L]$ through \[  g(\ell_0 x_0 + \ell_1 x_1 + \dots + \ell_k x_k)  = \ell'_0(n_0 x_0) + \ell'_1(n_1 x_1) + \dots + \ell'_k(n_k x_k). \]
	Where for every $0 \leq i \leq k$, $\ell'_i \in L$ and $n_i\in N$ is the unique pair of elements such that $\ell'_i n_i = g \ell_i$. Notice that the map $(\ell_0,\dots,\ell_k) \overset{g}{\mapsto} (\ell'_0,\dots,\ell'_k)$ is a permutation, and thus the formal sum in the right is well defined.
	It is clear that $\Gamma\curvearrowright X[L]$ is indeed an action. Let us argue that it is effectively closed. Notice that if $S$ is a set of generators for $N$, then $S \cup L$ is a finite generating set for $\Gamma$. For each $s \in S \cup L$, we can find pairs $(\ell'_{i,s},w_{i,s})_{0 \leq i \leq k}$ where $\ell'_{i,s} \in L$ and $w_{i,s} \in S^*$ are such that $s\ell_i = \ell'_{i,s}\underline{w_{i,s}}$ and then use the Turing machine which computes $N \curvearrowright X$ to compute $\underline{w_{i,s}}x_i$. As there are finitely many generators, the pairs $(\ell'_{i,s},w_{i,s})_{0 \leq i \leq k}$ can be hard-coded in a Turing machine and thus we conclude that $\Gamma \curvearrowright X[L]$ is effectively closed.
	
	As $\Gamma$ is self-simulable, there exists a $\Gamma$-subshift of finite type $Z$ which factors onto $\Gamma\curvearrowright X[L]$. Observe that the $N$-subaction of a $\Gamma$-subshift of finite type is topologically conjugated to an $N$-subshift of finite type whenever $N$ has finite index in $\Gamma$ (this is usually called the \define{higher-block shift}. The proof of this last claim can be found in~\cite[Proposition 7]{CarrollPenland}). It follows that there is a $N$-subshift of finite type which factors onto the $N$-subaction of $\Gamma \curvearrowright X[L]$. Composing this factor map with the projection of $X[L]$ to the $\ell_0$-coordinate we obtain that $N \curvearrowright X$ is a factor of this $N$-subshift of finite type, and thus $N$ is self-simulable.
\end{proof}

\begin{lemma}\label{lem:SS_passesup}
	Let $\Gamma$ be a finitely generated group and suppose $N \trianglelefteq \Gamma$ is a finite index normal subgroup which is self-simulable. Then $\Gamma$ is self-simulable.
\end{lemma}

Notice that if $\Gamma$ is recursively presented, this follows immediately from Corollary~\ref{cor:normalSS} without the finite index assumption.

\begin{proof}
	Let $\Gamma \curvearrowright X$ be an effectively closed action. As $N$ is a finitely generated subgroup of $\Gamma$, it follows that the $N$-subaction $N\curvearrowright X$ of $\Gamma \curvearrowright X$ is also effectively closed. As $N$ is self-simulable, there exists an $N$-subshift of finite type $Y$ which factors onto $N \curvearrowright X$ through a map $\phi \colon Y \to X$. 
	
	Let $L\ni 1_{\Gamma}$ be a finite set of representatives of left cosets of $N$ such that $\Gamma = LN$. Let $B$ be the alphabet of $Y$ and consider the $\Gamma$-subshift $Z$ consisting on the configurations $z \in (B \cup \{\varnothing \})^{\Gamma}$ which satisfy the following two constraints:
	\begin{enumerate}
		\item If $z(g) \in B$, then $(z(gn))_{n \in N}\in Y$.
		\item For every $g \in \Gamma$, there is a unique $\ell^{-1} \in L^{-1}$ such that $z( g\ell^{-1}) \neq \varnothing$.
	\end{enumerate}
	As $Y$ is a subshift of finite type and $L^{-1}$ is finite, it follows that $Z$ is a subshift of finite type. Moreover, for every $y \in Y$ we can construct $z \in Z$ with $z|_N = y$ by letting \[z(g) = \begin{cases}
		y(g) & \mbox{ if } g \in N \\
		\varnothing & \mbox{ if } g \notin N.
	\end{cases}   \]
	
	By definition, there is a well-defined map $\psi\colon Z \to L$ such that $z|_{\psi(z)N} \in Y$. We define $\widetilde{\phi}\colon Z \to X$ by \[ \widetilde{\phi}(z) = \psi(z) \phi((\psi(z)^{-1} z)|_{N}  ) \mbox{ for every }z\in Z.  \]
	Where the inner product is the shift map, and the outer product is the action $\Gamma \curvearrowright X$.
	We claim $\widetilde{\phi}$ is a factor map. It is clearly continuous and surjective. To see that $\widetilde{\phi}$ is $\Gamma$-equivariant, suppose first that $n \in N$. Then we have that $\psi(nz) = \psi(z)$ and therefore we obtain \begin{align*}
		\widetilde{\phi}(nz) & = \psi(nz)\phi((\psi(nz)^{-1}nz)|_N)\\
		& = \psi(z)\phi( (\psi(z)^{-1} n \psi(z) \psi(z)^{-1}z)|_N ) \\ & = \psi(z)(\psi(z)^{-1} n \psi(z))\phi((\psi(z)^{-1}z)|_N)\\ & = n\widetilde{\phi}(z).
	\end{align*}
	
	Now let $\ell \in L$ and notice that there is $n \in N$ such that $\psi(\ell z) =  \ell\psi(z)n$. We can therefore write \begin{align*}
		\widetilde{\phi}(\ell z) & = \psi( \ell z)\phi((\psi( \ell z)^{-1}\ell z)|_N)\\
		& = \ell\psi(z)n \phi((   n^{-1}\psi(z)^{-1} \ell^{-1} \ell z)|_N ) \\ & = \ell\psi(z) \phi((\psi(z)^{-1} z)|_N ) \\ & = \ell \widetilde{\phi}(z).
	\end{align*}
	As $L$ and $N$ generate $\Gamma$, we conclude that $\widetilde{\phi}$ is $\Gamma$-equivariant and thus a factor map. Hence $\Gamma$ is self-simulable.
\end{proof}

An immediate consequence of Lemma~\ref{lem:SS_passesup} and Lemma~\ref{lem:normalfiniteindex_isSS} is that the property of being self-simulable is an invariant of commensurability. That is, we have proven Theorem~\ref{thm:commensurable}.

Let us recall that a function $f\colon X \to Y$ between metric spaces $(X,d_X)$ and $(Y,d_Y)$ is a \define{quasi-isometry} if there exist real constants $A \geq 1$ and $B,C\geq 0$ such that \begin{enumerate}
	\item $f$ is a \textbf{quasi-isometric embedding: }for every $x,y \in X$ \[ \frac{1}{A}d_X(x,y)-B \leq d_Y(f(x),f(y)) \leq Ad_X(x,y)+B.  \]
	\item $f(X)$ is \textbf{relatively dense:} for every $z \in Y$ there is $x \in X$ such that $d(z,f(x)) \leq C$.
\end{enumerate}

Let $\Gamma$ be a finitely generated group generated by $S\Subset \Gamma$. For $g \in \Gamma$, let $|g|_S$ be the minimum length of a word $w \in S^*$ such that $g = \underline{w}$. We have that $d_S(g,h) = |g^{-1}h|_S$ is a metric on $\Gamma_1$ called the \define{word metric}. Notice that if two finitely generated groups are quasi-isometric with respect to the word metrics induced by a specific choice of generating sets, then they are quasi-isometric with respect every choice of generating sets. In what follows, when referring to quasi-isometries of finitely generated groups we implicitly consider the group as a metric space with respect to some word metric.

The rest of the section is dedicated to the proof of Theorem~\ref{thm:quasiisometry}, namely, that for finitely presented groups the property of being self-simulable is stable under quasi-isometries. By Proposition~\ref{prop:SelfSimuImpNonAmenable}, we already know that Theorem~\ref{thm:quasiisometry} holds trivially for amenable groups and thus in our proof we may assume that both groups involved are non-amenable. This will greatly simplify our proof due to a result of Whyte~\cite[Theorem 4.1]{whyte_amenability_1999} which implies that every quasi-isometry between non-amenable finitely generated groups is at bounded distance from a bilipschitz map.

For the remainder, fix two non-amenable finitely presented groups $\Gamma_1$ and $\Gamma_2$, an effectively closed action $\Gamma_1 \curvearrowright X$ and a bilipschitz map $f \colon \Gamma_2 \to \Gamma_1$ such that $f(1_{\Gamma_2}) = 1_{\Gamma_1}$. We will assume that $\Gamma_2$ is self-simulable and prove the existence of a subshift of finite type which factors onto $\Gamma_1 \curvearrowright X$.

\begin{lemma}\label{lem:generatinggoodQI}
	There exist finite symmetric generating sets $S \Subset \Gamma_1$ and $T,U \Subset \Gamma_2$ such that $1_{\Gamma_1}\in S$, $1_{\Gamma_2}\in T \cap U$ and \begin{enumerate}
		\item $\Gamma_1$ can be presented by $S$ with relations of length at most $3$.
		\item $\Gamma_2$ can be presented by $T$ with relations of length at most $3$.
		\item $f(h T) \subseteq f(h)S$ for every $h \in \Gamma_2$.
		\item $f(h)S \subseteq f(hU)$ for every $h \in \Gamma_2$.
	\end{enumerate}
\end{lemma}

\begin{proof}
	Let us fix symmetric generating sets $S_1$ and $S_2$ of $\Gamma_1$ and 
	$\Gamma_2$ respectively. As $f$ is a bilipschitz map, it is a bijection and there is $D \geq 1$ such that \[ \frac{1}{D}d_{S_2}(h_1,h_2) \leq d_{S_1}(f(h_1),f(h_2)) \leq Dd_{S_2}(h_1,h_2),  \] 
	
	As $\Gamma_2$ is finitely presented, there is a finite set of words $W_2 \subseteq S_2^*$ which presents $\Gamma_2$. Let $K_2 \geq 1$ be the largest length of a word in $W_2$ and take $T$ as the symmetric closure of the set of all words on $S_2^*$ of length at most $K_2$. Take $R_2$ as the set of all words on $T$ such that
	\begin{itemize}
		\item if $w \in W_2$, then $w \in R_2$.
		\item for every $s \in S_2$ and $t_1,t_2 \in T$ such that ${t_1s} = {t_2}$ then $t_1st_2^{-1}\in R_2$.
	\end{itemize}
	It is clear that $R_2$ presents $\Gamma_2$ with the generating set $T$ and thus we have (2).
	
	Now let $h \in \Gamma_2$ and $t \in T$. notice that $d_{S_2}(h,ht) = |t|_{S_2}$ and thus $d_{S_1}(f(h),f(ht))\leq D|t|_{S_2}$. This means that we have $f(ht) \in f(h)S_1^{\lfloor DK_2\rfloor}$ for every $t \in T$ and thus $f(hT)\subseteq f(h)S_1^{\lfloor DK_2\rfloor}$.
	
	As $\Gamma_1$ is finitely presented, there is a finite set of words $W_1 \subseteq S_1^*$ which presents $\Gamma_1$. Let $K_1 \geq 1$ be the largest length of a word in $W_1 \cup S_1^{\lfloor DK_2\rfloor}$. Letting $S$ be the symmetric closure of the set of all words on $S_1$ of length at most $K_1$ we conclude, in the same way as with $T$ and $\Gamma_2$, that $\Gamma_1$ is presented by words of length at most $3$ in $S$. By the choice of $K_1$, it also follows that for every $h \in \Gamma_2$, $f(hT) \subseteq f(h)S_1^{\lfloor DK_2\rfloor} \subseteq f(h)S$ and so (1) and (3) hold.
	
	Finally, let $s\in S$ and $h \in \Gamma_2$ such that $f(h)s = f(hu)$ for some $u \in \Gamma_2$. It follows that $d_{S_1}(f(h),f(hu) ) = |s|_{S_1} \leq K_1$. Then we have that \[ d_{S_2}(h,hu)\leq DK_1  \]
	Therefore $|u|_{S_2} \leq DK_1$ Letting $U$ be the symmetric closure of the union of $T$ with all the words of length at most $DK_1$ we obtain (4).\end{proof}

Let $S,T,U$ be sets satisfying the conditions of Lemma~\ref{lem:generatinggoodQI} and let $Y \subseteq (S^T)^{\Gamma_2}$ be the set of configurations $y$ which satisfy the following cocycle-like conditions: 

\begin{enumerate}
	\item For every $h \in \Gamma_2$ and $t \in T$ we have \[ y(h)(1_{\Gamma_2}) = 1_{\Gamma_1} \mbox{ and } y(h)(t^{-1}) = (y(ht^{-1})(t))^{-1}. \]
	\item For every $h \in \Gamma_2$ and $t_1,t_2,t_3 \in T$ such that ${t_3t_2t_1} = 1_{\Gamma_2}$ we have 
	\begin{align*}
		& \mbox{ if } \left[ y(h t_1^{-1}t_2^{-1})(t_3^{-1}) = s_3 \mbox{ and } y(ht_1^{-1})(t_2^{-1}) = s_2 \mbox{ and } y(h)(t_1^{-1})=s_1 \right],\\ & \mbox{ then } {s_1s_2s_3} = 1_{\Gamma_1}.
	\end{align*}
\end{enumerate}

From the definition above, it follows that $Y$ equipped with the left shift action $\Gamma_2 \curvearrowright Y$ is a $\Gamma_2$-subshift of finite type. It is nonempty as a valid configuration may be defined by letting $y(h)(t) = f(h)^{-1}f(ht)$ for every $h \in \Gamma_2$ and $t \in T$.

Consider the map $\gamma \colon T \times Y \times X \to Y \times X$ given by \[ \gamma(t,(y,x)) = (ty, (y(1_{\Gamma_2})(t^{-1}))^{-1}x) \mbox{ for every } t \in T \mbox{ and } ({y},x) \in {Y}\times X.  \]
The term $ty$ in the expression above is given by the left shift action, and the second term is given by the action $\Gamma_1 \curvearrowright X$.

\begin{lemma}The map $\gamma$ induces a left action $\Gamma_2 \overset{\gamma}{\curvearrowright} Y \times X$ which is topologically conjugate to an effectively closed left action of $\Gamma_2$.
\end{lemma}

\begin{proof}
	Let us first argue that $\gamma$ induces an action of $\Gamma_2$. As $T$ is a set of generators of $\Gamma_2$, it suffices to show that every relation in a presentation by $T$ (as an abstract monoid) yields the trivial map. As $\Gamma_2$ can be presented by relators of length at most $3$, the conditions on configurations of $Y$ ensure that this condition holds, therefore $\gamma$ induces a left action of $\Gamma_2$ on $Y \times X$.
	
	The left shift action $\Gamma_2 \curvearrowright Y$ is a subshift of finite type and thus an effectively closed subshift (by patterns). As $\Gamma_2$ is finitely presented, by Proposition~\ref{prop:conjugacy_rec_presented_effsubshift}, $\Gamma_2 \curvearrowright Y$ is topologically conjugated to an effectively closed action $\Gamma_2 \curvearrowright \widetilde{Y} \subseteq (S^T)^{\NN}$ through a map $\phi \colon \widetilde{Y}\to Y$. Let us note that, by the construction given in the proof of Proposition~\ref{prop:conjugacy_rec_presented_effsubshift}, the map $\phi$ can be chosen (by letting $\rho(0)$ be the empty word in $T^*$) such that $\widetilde{y}_0 = \phi(\widetilde{y})(1_{\Gamma_2})$ for any $\widetilde{y}\in \widetilde{Y}$. It follows that the action $\Gamma_2 \overset{\gamma}{\curvearrowright} Y \times X$ is topologically conjugated to the action $\Gamma_2 \overset{\gamma}{\curvearrowright} \widetilde{Y} \times X$ given by \[ t(\widetilde{y},x) =  (t\widetilde{y}, \widetilde{y}_0(t^{-1})^{-1}x) \mbox{ for every } t \in \Gamma_2 \mbox{ and } (\widetilde{y},x) \in \widetilde{Y}\times X. \]
	
	Finally, this action is effectively closed because both $\Gamma_2 \curvearrowright \widetilde{Y}$ and $\Gamma_1 \curvearrowright X$ are effectively closed and $\widetilde{y}_0(t^{-1})^{-1}$ can be effectively computed from $\widetilde{y}$ and $t$. \end{proof}

Let us now suppose that $\Gamma_2$ is self-simulable. Then there exists a finite alphabet $A$, a $\Gamma_2$-subshift of finite type $Z \subseteq A^{\Gamma_2}$ and a topological factor map $\phi_2 \colon Z \to Y \times X$ such that $\Gamma_2 \overset{\gamma}{\curvearrowright} Y \times X$ is the factor of the left shift action $\Gamma_2 \curvearrowright Z$ under $\phi_2$.

Without loss of generality (by recoding $Z$) we may assume that $Z$ is described by a nearest neighbor set of forbidden patterns. In particular, there is $\mathcal{L} \subseteq A^T$ such that $z \in Z$ if and only if $(hz)|_{T} \in \mathcal{L}$ for every $h \in \Gamma_2$. 

Furthermore, as $\phi_2$ is continuous, if $\phi_2(z)=(y,x)$ then the symbol $y(1_{\Gamma_2})$ depends only on finitely many coordinates of $Z$. We may further recode $Z$ so that this information is contained in the coordinate $z(1_{\Gamma_2})$, that is, such that there is a function $\Psi\colon A \to S^T$ with the property that if $\phi_2(z)=(y,x)$ then $y(h) = \Psi(z(h))$ for every $h \in \Gamma_2$. Note that $\Psi$ extends to a $1$-block factor map $\psi \colon Z \to Y$.

In what follows, we shall describe a $\Gamma_1$-subshift of finite type $W$ which encodes configurations of $Z$ on $\Gamma_1$ using configurations that encode the space of bilipschitz maps which resemble $f$. Our description is essentially based on a construction of Cohen~\cite{Cohen2014} which describes a process to encode the structure of a finitely presented group into another quasi-isometric group through a subshift of finite type.

\subsection{Construction of $W$}

Consider the alphabet $A$ which defines $Z$ and recall there is a map $\Psi \colon A \to S^T$. The subshift of finite type $W$ is the set of configurations $w \in A^{\Gamma_1}$ which satisfy the following constraints.

\textbf{Constraint 1: } For every $g \in \Gamma_1$ and $t,t_1,t_2,t_3 \in T$ such that ${t_3t_2t_1}= 1_{\Gamma_2}$, we have

\begin{enumerate}
	\item  $\Psi(w(g))(1_{\Gamma_2}) = 1_{\Gamma_1}$.
	\item $\Psi(w(g))(t)\Psi(g\Psi(w(g))(t))(t^{-1}) = 1_{\Gamma_1}$.
	\item $\Psi(w(g)(t_3))\Psi(g\Psi(w(g)(t_3)))(t_2)\Psi(g\Psi(g\Psi(w(g)(t_3)))(t_2))(t_1) = 1_{\Gamma}$.
\end{enumerate}

While from this definition it is clear that this condition can be implemented using finitely many forbidden patterns, its meaning is not immediately clear and it is quite painful to parse. Given $w \in A^{\Gamma_1}$, we may define a map $*_w \colon \Gamma_1 \times T \to \Gamma_1$ given by $ g *_w t = g \Psi(w(g))(t)$. In terms of this action the three conditions above translate in following more intuitive presentation.

\begin{enumerate}
	\item  $*_w 1_{\Gamma_2} = \operatorname{id}_{\Gamma_1}$.
	\item $*_w{t}\circ *_w{t^{-1}} = \operatorname{id}_{\Gamma_1}$.
	\item $*_w{t_3} \circ *_w{t_2}\circ *_w{t_1} = \operatorname{id}_{\Gamma_1}$.
\end{enumerate}

Therefore this first constraint encodes the fact that this map extends to a well-defined right action $\Gamma_1 \overset{*_w}{\curvearrowleft} \Gamma_2$ for every $w \in W$. 
For $h \in \Gamma_2$, we will write $*_w h$ to denote this action.

\textbf{Constraint 2: } For every $g \in \Gamma_1$ and every $s \in S$ there exists $u \in U$ such that \[  g*_w u = gs.   \]

This constraint encodes the fact that the actions $\Gamma_1 \overset{*_w}{\curvearrowleft} \Gamma_2$ are transitive. Indeed, for $g_0,g \in \Gamma_1$ there are $s_1\dots s_n \in S^*$ such that $g = g_0s_1\dots s_n$. Let $g_i = g_0s_1\dots s_i$. By constraint 4, there is $u_i \in U$ such that $g_{i-1}*_w u_i = g_{i-1}s_i = g_{i}$ Letting $h= {u_{0}\dots u_{n-1}}$ we obtain that $g_0*_w h = g_n = g$.

Given $g \in \Gamma_1$ we can extract from $w$ a configuration $\chi(w,g) \in A^{\Gamma_2}$ given by $\chi(w,g)(h) = w(g *_w h)$. This configuration satisfies the following two useful equations. For every $g,g' \in \Gamma_1$ and $h \in \Gamma_2$, we have \[h^{-1}\chi(w,g) = \chi(w,g*_w h) \mbox{ and } \chi(gw,g') = \chi(w, g^{-1}g').\]

Recall that there is a finite set of patterns $\mathcal{L} \subseteq A^T$ such that $z \in Z$ if and only if for every $h \in \Gamma_2, (gz)|_{T} \in \mathcal{L}$. 

\textbf{Constraint 3: } For every $g \in \Gamma_1$, we have that $\chi(w,g)|_{T} \in \mathcal{L}$.

This constraint encodes the fact that $\chi(w,g) \in Z$. Indeed, for every $h \in \Gamma_2$ we have $h^{-1}\chi(w,g)|_T = \chi(w,g*_w h)|_T \in \mathcal{L}$ and thus $\chi(w,g)\in Z$.

From the considerations above, it is clear that $W$ is a $\Gamma_1$-subshift of finite type.

\subsection{Construction of the factor map}

Let us recall we have a factor map $\phi_2 \colon Z \to Y \times X$. Let $\zeta$ denote the projection of $\phi_2$ onto the second coordinate, that is, if $\phi_2(z)=(y,x)$ then $\zeta(z)=x$. 

Recall we have constructed a $\Gamma_1$-subshift of finite type $W$ whose configurations satisfy constraints (1)-(3). We define $\phi_1 \colon W \to X$ by \[ \phi_1(w) = \zeta(\chi(w,1_{\Gamma_1})). \]

We claim that $\phi_1$ is a topological factor map. This will be proven in the following three claims.

\begin{claim}\label{claim:1}
	The map $\phi_1$ is continuous.
\end{claim}

\begin{proof}
	If $h \in \Gamma_2$, then $\chi(w,1_{\Gamma_1})(h) = w(1_{\Gamma_1} *_w h)$ and thus it depends upon the values of $w$ on the set $S^{|h|_T}$. We conclude that $\phi_1$ is continuous. Furthermore, $\phi_2$ is continuous, we have that $\zeta$ is continuous and thus $\phi_1$ is continuous.
\end{proof}

\begin{claim}\label{claim:2}
	The map $\phi_1$ is $\Gamma_1$-equivariant.
\end{claim}

\begin{proof}
	As $\Gamma_1$ is generated by $S$, it suffices to show that for every $s \in S$ and $w \in W$, we have $s\phi_1(w) = \phi_1(sw)$.
	
	By condition 2, that there is $u \in U$ such that $1_{\Gamma_1} *_w u = s^{-1}$. From here, we obtain that \[ \chi(sw,1_{\Gamma_1}) = \chi(w, s^{-1}) = \chi(w, 1_{\Gamma_1}*_w u) = u^{-1} \chi( w, 1_{\Gamma}).    \]
	Using that $\phi_2$ is $\Gamma_2$-equivariant, we get 
	\[ \phi_2(\chi(sw,1_{\Gamma})) = u^{-1}\phi_2(\chi(w, 1_{\Gamma_1})). \]
	
	Let us introduce a bit of notation. Let $t_1\dots t_n \in T^*$ such that $u = \underline{t_1\dots t_n}$. Let $\ell_0 = 1_{\Gamma_1}$, $\ell_n = s^{-1}$ and for $1 \leq j < n$ define $\ell_j$ as the element such that $\ell_0*_w {t_1\dots t_j} = \ell_j$.
	
	Let us write $(y,x) = \phi_2(\chi(w, 1_{\Gamma_1}))$. By the previous equation we have that $\phi_2(\chi(sw,1_{\Gamma_1})) = u^{-1}(y,x)$. By definition of the action $\Gamma_2 \curvearrowright Y \times X$, we have that \[  u^{-1}(y,x) = (u^{-1}y, {\left(y(t_1t_2\dots t_{n-1})(t_n)\right)^{-1}\dots \left(y(t_1)(t_2)\right)^{-1}\left( y(1_{\Gamma_1})(t_1)\right)^{-1}}x ).  \]
	
	Recall that $\Psi\colon A \to S^T$ extends to a $1$-block factor map $\psi \colon Z \to Y$ and thus $y = \psi( \chi(w, 1_{\Gamma_1}) )$. Therefore we obtain that for every $h \in \Gamma_2$ we have \[y(h) =h^{-1}y(1_{\Gamma_2}) = \psi(h^{-1}\chi(w,1_{\Gamma_1}))(1_{\Gamma_2}) =  \Psi(\chi( w, 1_{\Gamma_1}*_w h )(1_{\Gamma_2})). \]
	
	In particular, for every $1 \leq j \leq n$, then 
	\begin{align*}
		y(t_1t_2\dots t_{j-1})(t_j) & = \Psi(\chi( w,(1_{\Gamma_1}*_w t_1t_2\dots t_{j-1})(1_{\Gamma_2}))(t_j) \\ & = \Psi(\chi( w, \ell_{j-1} )(1_{\Gamma_2}))(t_j)\\
		&= \ell_{j-1}^{-1}\ell_{j}.
	\end{align*}
	
	From this equation, we obtain that 
	\begin{align*}
		u^{-1}(y,x) & = (u^{-1}y, ( \ell_{n-1}^{-1}\ell_n)^{-1}(\ell_{n-2}^{-1}\ell_{n-1})^{-1}\dots(\ell_{0}^{-1}\ell_1 )^{-1}x)\\ & = (u^{-1}y, \ell_n^{-1} \ell_0 x)\\ &  = (u^{-1}y, sx).
	\end{align*}
	
	Therefore \[ \phi_1(sw) = \zeta( u^{-1} \chi(w, 1_{\Gamma_1})) = s\zeta(\chi(w, 1_{\Gamma_1})) = s \phi_1(w).   \]
	
	From where we deduce that $\phi_1$ is indeed $\Gamma_1$-equivariant.
\end{proof}

Notice that up to this point, we have not used $f$ in any meaningful way other that to determine the constant $D$. As far as the definition goes, the actions induced by elements of $w$ might not be free. The next claim proving surjectivity of $\phi_1$ relies strongly on the existence of this particular bilipschitz map $f$. 

\begin{claim}\label{claim:3}
	The map $\phi_1$ is surjective.
\end{claim}

\begin{proof}
	Let $x \in X$ be arbitrary. Let us first define $y \in S^T$ by \[ y(h)(t) = f(h)^{-1}f(ht). \]
	By Lemma~\ref{lem:generatinggoodQI} part (3) we have that $y \in S^T$. It is clear by definition that $y(h)(1_{\Gamma_2})= 1_{\Gamma_1}$. Furthermore, for every $t$ we have \[(y(ht^{-1})(t))^{-1} = (f(ht^{-1})^{-1}f(ht^{-1}t))^{-1} = f(h)^{-1}f(ht^{-1}) = y(h)(t^{-1}).\]
	Finally, for $t_1,t_2,t_3 \in T$ such that ${t_3t_2t_1}= 1_{\Gamma_2}$ we have 
	\begin{align*}
		y(h)(t_1^{-1}) & = \underbracket{f(h)^{-1}f(ht_1^{-1})}_{s_1},\\ y(ht_1^{-1})(t_2^{-1}) & = \underbracket{f(ht_1^{-1})^{-1}f(ht_1^{-1}t_2^{-1})}_{s_2}\\  y(ht_1^{-1}t_2^{-1})(t_3^{-1}) & = \underbracket{f(ht_1^{-1}t_2^{-1})^{-1}f(ht_1^{-1}t_2^{-1}t_3^{-1})}_{s_3}.
	\end{align*}
	As ${t_1^{-1}t_2^{-1}t_3^{-1}}= 1_{\Gamma_2}$ we verify that \[{s_1s_2s_3} = f(h)^{-1}f(ht_1^{-1}t_2^{-1}t_3^{-1})= f(h)^{-1}f(h) = 1_{\Gamma_1}.
	\]
	Therefore we conclude that $y \in Y$, and thus $(y,x) \in Y \times X$. Let $z \in Z$ such that $\phi_2(z) = (y,x)$, it suffices to construct $w \in W$ such that $\chi(w,1_{\Gamma_1}) = z$, as then we would have $\phi_1(w) = \zeta(\chi(w,1_{\Gamma_1}))= \zeta(z) = x$.
	
	Let $w \in A^{\Gamma_1}$ be given by $w(g) = z(f^{-1}(g))$ for every $g \in \Gamma_1$. Notice that for every $t \in T$ and $g \in \Gamma_1$, we have 
	\begin{align*}
		g *_w t & = g\Psi(w(g))(t) \\ & =  g\Psi(z(f^{-1}(g)))(t)\\ & = g(f(f^{-1}(g))^{-1}f(f^{-1}(g)t)\\ & = f(f^{-1}(g)t).
	\end{align*}
	As $T$ is a generating set, we deduce that for any $h \in \Gamma_2$ then $g *_w h = f(f^{-1}(g)h)$.
	
	From this we deduce that constraint $1$ holds on $w$. Indeed, let $g \in \Gamma_1$  be arbitrary. For the first condition we have \[g *_{w} 1_{\Gamma_1} = f(f^{-1}(g)1_{\Gamma_1}) = g.\]
	
	For the second condition, let $t \in T$. We have \begin{align*}
		g *_{w} t *_{w} t^{-1} & = f(f^{-1}(g)t) *_w t^{-1}\\
		& = f(f^{-1}(g)tt^{-1}) = g.
	\end{align*}
	For the third condition, let $t_1,t_2,t_3 \in T$ such that $t_3t_2t_1 = 1_{\Gamma_2}$  we have \begin{align*}
		g *_{w} t_3 *_{w} t_2 *_{w} t_1 & = f(f^{-1}(g)t_3) *_{w} t_2 *_{w} t_1 \\
		& = f(f^{-1}(g)t_3t_2) *_{w} t_1 \\
		& = f(f^{-1}(g)t_3t_2t_1) = g.
	\end{align*}
	
	By Lemma~\ref{lem:generatinggoodQI} part (4), we get that for every $h \in \Gamma_2$ we have $f(h)S\subseteq f(hU)$ therefore for every $s \in S$ and $g \in \Gamma_1$ we obtain that $gs$ can be written as $f(f^{-1}(g)u) = g *_w u$ for some $u \in U$ and thus constraint 2 holds. 
	
	Finally, recall that we have asked that $f(1_{\Gamma_2}) = 1_{\Gamma_1}$. By construction we have that for every $h \in \Gamma_2$, 
	\[\chi(w,1_{\Gamma_1})(h) = w(1_{\Gamma} * h) = w(f(f^{-1}(1_{\Gamma_1})h)) = w(f(h)) = z(h).\] We conclude that $z = \chi(w,1_{\Gamma_1})$, and in particular that constraint $3$ holds and thus $w \in W$. This shows that $\phi_1$ is surjective.\end{proof}

We have shown that $W$ is a $\Gamma_1$-subshift of finite type and in Claims~\ref{claim:1},~\ref{claim:2} and~\ref{claim:3} that $\phi_1 \colon W \to X$ is a topological factor map. As $\Gamma_1 \curvearrowright X$ was an arbitrary effectively closed action, we obtain that $\Gamma_1$ is self-simulable. This completes the proof of Theorem~\ref{thm:quasiisometry}.

\section{Examples and applications}\label{sec:applications}

In this section we provide several handy versions of the theorems in the last sections which enables us to show that several well known non-amenable groups in the literature are self-simulable. 

\subsection{Branch groups}

An immediate consequence of Theorem~\ref{thm:commensurable} is that whenever a group is commensurable to a product of two finitely generated non-amenable groups, then it must be self-simulable. An interesting class of groups for which a slighter weaker property holds is the class of branch groups. An extensive survey about these groups can be found in~\cite{BartholdiGrigorchuk2003Branchgroups}.

We will not provide the definition of these groups here and just refer the reader to~\cite[Definition 1.1]{BartholdiGrigorchuk2003Branchgroups}. The sole property of these groups we will use, and which can be deduced directly from the definition, is that any branch group $\Gamma$ admits a finite index subgroup $H$ which is isomorphic to the direct product of at least two isomorphic copies of some group $L$.

\begin{corollary}\label{cor:branches}
	Let $\Gamma$ be a finitely generated branch group. If $\Gamma$ is non-amenable then $\Gamma$ is self-simulable. The converse holds whenever $\Gamma$ is recursively presented.
\end{corollary}

\begin{proof}
	Suppose $\Gamma$ is non-amenable, then the finite index subgroup $H$ is non-amenable (and finitely generated). From here it follows that $L$ must be non-amenable and finitely generated, and thus $H$ is isomorphic to the direct product of at least two finitely generated non-amenable groups. By Theorem~\ref{thm:selfsimulation} it follows that $H$ is self-simulable. As $\Gamma$ is commensurable with $H$ we obtain that $\Gamma$ is self-simulable by Theorem~\ref{thm:commensurable}.
	
	The converse is a direct consequence of Proposition~\ref{prop:SelfSimuImpNonAmenable}.
\end{proof}

\subsection{Burger-Mozes groups}

A construction of Burger and Mozes~\cite{BurgerMozes2020} provides a class of finitely presented simple groups which act geometrically on the direct product of two infinite simplicial trees with constant bounded degree $\geq 3$. By the {\v S}varc-Milnor lemma, these groups are quasi-isometric to the direct product $F_2 \times F_2$, which is self-simulable by Theorem~\ref{thm:selfsimulation}. Putting this fact together with Theorem~\ref{thm:quasiisometry} yields the following corollary.

\begin{corollary}\label{cor:Burgermozes}
	The groups of Burger and Mozes are self-simulable. In particular, there exist finitely presented non-amenable simple groups which are self-simulable.
\end{corollary}

A reader interested in the groups by Burger and Mozes, will find plenty of information in the following chapter of Bartholdi~\cite{bartholdi_2018}.

\subsection{Thompson's groups}
\label{sec:V}

Next, we provide a handy tool for showing that a group is self-simulable, by simply producing a sufficiently rich faithful action on a zero-dimensional space.

\begin{lemma}
	\label{lem:acts_on_0d}
	Let $\Gamma$ be a finitely generated and recursively presented group that acts faithfully on a zero-dimensional space $X$ by homeomorphisms, such that for each nonempty open set $U \subseteq X$, the subgroup of elements of $\Gamma$ which act trivially on $X\setminus U$ is non-amenable. Then $\Gamma$ is self-simulable.
\end{lemma}

\begin{proof}
	For an open set $U\subseteq X$, write $\Gamma_U$ for the subgroup of elements which act trivially on $X\setminus U$. As $X$ is zero-dimensional, we can find two nonempty disjoint clopen sets $A,B$ such that $X = A \cup B$. As $\Gamma\curvearrowright X$ is faithful, it follows that $\Gamma_A\cap\Gamma_B = \{1_{\Gamma}\}$ and thus they commute and form a direct product inside $\Gamma$, As both $\Gamma_A$ and $\Gamma_B$ are non-amenable, they contain finitely generated non-amenable groups $\Delta_A$ and $\Delta_B$, and we thus have that the finitely generated group $\Delta_A \times \Delta_B$ embeds into $\Gamma$. Let $T$ be a finite generating set for $\Gamma$ and for each $t\in T$ pick some arbitrary point $x_t \in A$. We have $t x_t \in C_t$ where $C_t \in \{A, B\}$. As $\Gamma$ acts by homeomorphisms, there's a clopen set $U_t$ which contains $x_t$ such that $U_t \subseteq A$ and $t U_t \subseteq C_t$.
	
	By our hypothesis, $\Gamma_{t  U_t}$ is non-amenable and thus we can pick a finitely generated non-amenable group $\Sigma_t$ inside $\Gamma_{t  U_t}$. Thus we have that $\Sigma_t \leqslant \Gamma_{t  U_t} \leqslant \Gamma_{C_t}$ and $t^{-1}(\Sigma_t)t  \leqslant \Gamma_{U_t} \leqslant \Gamma_A$. Let $\widetilde{\Delta}_A$ be the group generated by $\Delta_A$, the groups $\Sigma_t$ for $t \in T$ and the groups $t^{-1}(\Sigma_t)t$ for those $t$ such that $C_t =A$. Let $\widetilde{\Delta}_B$ be the group generated by $\Delta_B$ and the groups $t^{-1}(\Sigma_t)t$ for those $t$ such that $C_t =B$.
	
	By definition, the groups $\widetilde{\Delta}_A$ and $\widetilde{\Delta}_B$ are finitely generated. As they contain respectively $\Delta_A$ and $\Delta_B$, they are non-amenable and hence by Theorem~\ref{thm:selfsimulation}, the group $\widetilde{\Delta}_A \times \widetilde{\Delta}_B$ is self-simulable. Furthermore, we have that $\widetilde{\Delta}_A \times \widetilde{\Delta}_B \leqslant \Gamma_A \times \Gamma_B \leqslant \Gamma$. As $\widetilde{\Delta}_A \times \widetilde{\Delta}_B$ contains both $\Sigma_t$ and $t^{-1}(\Sigma_t)t$ for every $t \in T$, we conclude that it is mediated and thus $\Gamma$ is self-simulable by Theorem~\ref{thm:mediaintroduction}.
\end{proof}

The Thompson's groups are three finitely presented groups, commonly denoted by $F$, $T$ and $V$, which were introduced by R. Thompson as potential counterexamples for the now disproved Von Neumann-Day conjecture. The group $V$ was, to the best of our knowledge, the first example of an infinite and finitely presented simple group. The canonical reference for these groups are the notes~\cite{CanFloParr_thompson}. 

The group $V$ can be described as a subgroup of homeomorphisms of the Cantor space $C=\{\symb{0},\symb{1}\}^{\NN}$. An element $\phi$ of $V$ can be described by fixing two finite partitions $C = [w_1]\cup [w_2] \cup \dots \cup [w_n]$ and $C= [u_1]\cup [u_w] \cup \dots \cup [u_n]$, where $w_i,u_i \in \{\symb{0},\symb{1}\}^*$ for every $i \in \{1,\dots,n\}$. Then $\phi$ is the homeomorphism on $C$ which replaces the prefixes $w_i$ by $u_i$, that is, $\phi(w_i z) = u_iz$ for every $z \in \{\symb{0},\symb{1}\}^{\NN}$. This is usually described in terms of ``tree pairs'' made of carets and a finite permutation, see for instance~\cite{Corwin2013EmbeddingAN}. It is well-known that $V$ is non-amenable, a short proof being that the natural action of $V$ on $\{\symb{0},\symb{1}\}^{\NN}$ does not admit any Borel $V$-invariant probability measure, as every cylinder would necessarily have the same measure (and thus if $V$ were amenable it would contradict the Krylov-Bogolyubov theorem, see for instance~\cite[Theorem 4.4]{KerrLiBook2016}).

Brin-Thomson's group $nV$ was constructed in~\cite{Br04} as a higher dimensional analogue of Thompson's $V$ group which consists of homeomorphisms on $C^n$ where $C = \{\symb{0},\symb{1}\}^\NN$ is the Cantor set. We will not provide a precise definition of $nV$, but only use the following two simple facts: (1) $nV$ is a finitely generated and recursively presented group (in fact, it is finitely presented and has a decidable word problem), (2) $nV$ is non-amenable, and (3) if we fix any cylinder set $D = [w_1] \times [w_2] \times \cdots [w_n]$ for non-trivial $w_i \in \{\symb{0},\symb{1}\}^*$ then the subgroup of $nV$ which acts trivially on $C^n \setminus D$ is isomorphic to $nV$.

\begin{corollary}\label{cor:ThompsonV}
	Thompson's $V$, and more generally the higher Brin-Thompson groups $nV$ for $n \geq 1$, are self-simulable.
\end{corollary}

\begin{proof}
	Consider the natural faithful action of $nV \curvearrowright C^n$ by homeomorphisms, where $C= \{\symb{0},\symb{1}\}^\NN$ is the Cantor set. By property (3) above, for every nonempty open set $U$ the restriction of $nV$ to elements which act trivially on $C^n \setminus U$ contains a copy of $nV$, thus is non-amenable by property (2). Since $nV$ is recursively presented by property (1), Lemma~\ref{lem:acts_on_0d} applies.
\end{proof}

The natural action $V \curvearrowright \{\symb{0},\symb{1}\}^{\NN}$ is expansive and effectively closed. Therefore the previous result has the consequence that the natural action of $V$ on $\{\symb{0},\symb{1}\}^\NN$ is topologically conjugate to a sofic subshift.

The groups $F$ and $T$ are more commonly described as subgroups of homeomorphisms of $[0,1]$ and $\RR/\ZZ$ respectively. More precisely, $F$ (resp.\ $T$) is the group of homeomorphisms of $[0,1]$ (resp.\ $\RR/\ZZ$) which consists of piecewise linear maps whose slopes are powers of $2$ and the endpoints of each linear part are dyadic. Of course, these two groups can equivalently be described as homeomorphisms of $C=\{\symb{0},\symb{1}\}^{\NN}$. In the case of $F$, it is the subgroup of $V$ where the homeomorphisms are defined by partitions $C = [w_1]\cup [w_2] \cup \dots \cup [w_n]$ and $C= [u_1]\cup [u_w] \cup \dots \cup [u_n]$ where the sequences $(w_i)_{1 \leq i \leq n}$ and $(u_i)_{1 \leq i \leq n}$ are both given in lexicographical order. $T$ is defined similarly but the sequence $(u_i)_{1 \leq i \leq n}$ is allowed to be in lexicographical order up to a cyclic permutation. 

It is well known that $T$ is non-amenable (the same proof of non-existence of a Borel $V$-invariant measure works for $T$), but the status of the amenability of $F$ is a famous open problem. In what follows, we shall establish a link between the amenability of $F$ and self-simulability of $F$ and $T$.

\begin{corollary}
	\label{cor:FT}
	Thompson's $F$ is self-simulable if and only if $F$ is non-amenable. If $F$ is non-amenable then Thompson's $T$ is self-simulable.
\end{corollary}

\begin{proof}
	Suppose that $F$ is amenable, as $F$ is finitely presented it follows by Proposition~\ref{prop:SelfSimuImpNonAmenable} that there exists an effectively closed action of $F$ which is not the factor of any subshift. Thus $F$ is not self-simulable. 
	
	Conversely, suppose $F$ is non-amenable, let $\Gamma$ denote either $F$ or $T$ and consider the natural faithful action of $\Gamma$ on $\{\symb{0},\symb{1}\}^{\NN}$ by homeomorphisms. For every non-trivial word $w \in \{\symb{0},\symb{1}\}^{*}$ the subgroup of $\Gamma$ which fixes $\{\symb{0},\symb{1}\}^{\NN}\setminus [w]$ is isomorphic to $F$ (this subgroup consists of the elements of $\Gamma$ which act as the identity on any element of $[0,1]$ (resp.\ $\RR/\ZZ$) whose binary expansion does not begin with $w$). By Lemma~\ref{lem:acts_on_0d}, we conclude that $\Gamma$ is self-simulable.
\end{proof}

\subsection{General linear groups}

We will now show that the general linear groups $\operatorname{GL}_n(\ZZ)$ are self-simulable in high enough dimension. Let us first introduce some helpful notation.

Let $\Gamma$ be a finitely generated group which acts faithfully on a Cartesian product $A_1 \times A_2 \times \cdots \times A_n$. For $I \subseteq \{1,2,\dots,n\}$, denote by $\Gamma_I$ the subgroup of all elements $g \in \Gamma$ such that for some bijection $h \colon \prod_{i \in I} A_i \to \prod_{i \in I} A_i$, we have $g(a_1, a_2,\dots, a_n) = (b_1, b_2, \dots, b_n)$ where $b_j = a_j$ if $j \notin I$, and $b_j = h((a_i)_{i \in I})_j$ otherwise.

\begin{lemma}
	\label{lem:acts_on_product}
	Let $\Gamma$ be a finitely generated and recursively presented group which acts faithfully on a Cartesian product $A_1 \times A_2 \times \cdots \times A_n$. Suppose $n \geq 5$ and $\Gamma$ is generated by the subgroups $\Gamma_I$ for $I \subseteq \{1,2,...,n\}$ with $|I| = 2$, which are all non-amenable and finitely-generated. Then $\Gamma$ is self-simulable.
\end{lemma}

\begin{proof}
	Split $N = \{1,2,\dots,n\}$ into two disjoint sets $J$, $J'$ both with cardinality at least two. As the action is faithful, we have $\Gamma_J \cap \Gamma_{J'} = \{1_{\Gamma}\}$ and that $\Gamma_J$ commutes with $\Gamma_{J'}$. Therefore the subgroup $\Delta \leqslant \Gamma$ generated by both $\Gamma_J$ and $\Gamma_{J'}$ is isomorphic to $\Gamma_J \times \Gamma_{J'}$, which is a product of two finitely generated non-amenable groups and thus self-simulable by Theorem~\ref{thm:selfsimulation}. In what follows we shall show that $\Delta$ is mediated.
	
	By hypothesis, we can take the union of a generating set for each $\Gamma_I$ with $|I| = 2$ to get a generating set for $\Gamma$. Fix $I \subseteq N$. Then either (1) $I \subseteq J$, (2) $I \subseteq J'$ or (3) $I$ intersects both $J$ and $J'$. In cases (1) and (2) we clearly have that $\Gamma_I \leqslant \Gamma_J$ and $\Gamma_I \leqslant \Gamma_{J'}$ respectively. In case (3), as $n \geq 5$, we can find $I' = \{j, j'\}$ with $j \neq j'$ such that $I' \cap I = \varnothing$ and $I' \subseteq J$ or $I' \subseteq J'$. We conclude that $\Gamma_{I'} \leqslant \Gamma_J \times \Gamma_{J'}$ commutes with $\Gamma_{I}$ and thus can be taken as the group $\Sigma$ in Remark~\ref{rem:sigmas}.
\end{proof}

\begin{corollary}
	\label{cor:GL}
	The general linear group $\operatorname{GL}_n(\ZZ)$ and the special linear group $\operatorname{SL}_n(\ZZ)$ are self-simulable for $n \geq 5$.
\end{corollary}

\begin{proof}
	Let $\Gamma=\operatorname{SL}_n(\ZZ)$ for $n \geq 5$ and take $A_i = \ZZ$ for every $1 \leq i \leq 5$. Then clearly for every $I \subseteq \{1,\dots,n\}$ with $|I|=2$ we have that $\Gamma_I$ is isomorphic to $\operatorname{SL}_2(\ZZ)$, which is finitely generated and non-amenable. It is well known that for any Euclidean domain $R$, $\mathrm{SL}_n(R)$ is generated by the transvection matrices which add a multiple of a row (resp./ column) to another row (resp./ column), and therefore $\mathrm{SL}_n(\ZZ)$ is indeed generated by the $\Gamma_I$ for $|I|=2$. We conclude by Lemma~\ref{lem:acts_on_product} that $\operatorname{SL}_n(\ZZ)$ is self-simulable. As $\operatorname{SL}_n(\ZZ)$ is a normal subgroup of $\operatorname{GL}_n(\ZZ)$, we get that $\operatorname{GL}_n(\ZZ)$ is self-simulable by Corollary~\ref{cor:normalSS}.
\end{proof}

In Lemma~\ref{lem:acts_on_product} there is nothing special about using subgroup generators with $|I| = 2$, and indeed, we will now present a generalized method which extends this lemma and will allows us to prove that several new groups are also self-simulable.

\subsection{Groups generated by subgroups}

In this section, we shall present an abstract generalization of Lemma~\ref{lem:acts_on_product} that will enable us to cover (outer) automorphism groups of free groups (with enough generators), braid groups (with enough braids), and some classes of right-angled Artin groups. We also give an alternative proof for the Brin-Thompson's group $2V$ using this method.

Let $\mathcal{S}$ a set with two relations defined on it: an \define{orthogonality relation} $\bot \subseteq \mathcal{S}^2$ which is symmetric, and a \define{containment relation} ${\leq} \subseteq \mathcal{S}^2$ which is the order of a join-semilattice operation $\vee$. Also require that whenever $a \leq b$ and $b \bot c$ then $a \bot c$. 

\begin{definition}
	Let $\mathcal{S}$ be as above. We say $A \subseteq\mathcal{S}$ is \define{orthogonal} if for every pair of distinct $a,b \in A$ we have $a \bot b$. 
	
	For $A, B, \mathcal{N} \subseteq\mathcal{S}$, we say that the set $B$ is \define{weakly $\mathcal{N}$-suborthogonal} to $A$ if for every $b \in B$ there is $c \in \mathcal{N}$ and $a \in A$ such that $b \bot c$ and $c \leq a$.
\end{definition}

In the definition above, we should think of the set $\mathcal{S}$ as generalizing the power set of the interval $\{1,\dots,n\}$ in Lemma~\ref{lem:acts_on_product}. More precisely, let us say that a finitely-generated group $\Gamma$ is \define{compatible} with $(\mathcal{S},\bot,\leq)$ if for each $s \in \mathcal{S}$ we have a subgroup $\Gamma_s \leqslant \Gamma$, and the collection $(\Gamma_s)_{s \in \mathcal{S}}$ respects the operations in the sense that whenever $s \bot t$, then both $\Gamma_s\cap \Gamma_t = \{1_{\Gamma}\}$ and $[\Gamma_s, \Gamma_t] = 1$, and whenever $s \leq t$ then $\Gamma_s \leqslant \Gamma_t$. Say $B \subseteq\mathcal{S}$ is \define{generating} if $\Gamma$ is generated by the subgroups $(\Gamma_b)_{b \in B}$.

\begin{lemma}
	\label{lem:suborthogonal}
	Let $(\mathcal{S},\bot,\leq)$ be as above and $\Gamma$ be a finitely generated, recursively presented group which is compatible with $(\mathcal{S},\bot,\leq)$. Let $(\Gamma_s)_{s \in \mathcal{S}}$ be as above, and let $\mathcal{N},A,B \subseteq \mathcal{S}$ such that
	\begin{enumerate}
		\item $\Gamma_s$ is finitely generated and non-amenable for every $s \in \mathcal{N}$,
		\item $A \subseteq\mathcal{N}$ is finite, orthogonal and $|A|\geq 2$,
		\item $B$ is generating and weakly $\mathcal{N}$-suborthogonal to $A$.
	\end{enumerate} Then $\Gamma$ is self-simulable.
\end{lemma}

\begin{proof}
	The group $\Delta = \langle \bigcup_{a \in A} \Gamma_a \rangle$ generated by the subgroups $\Gamma_a$ is isomorphic to the direct product $\prod_{a \in A} \Gamma_a$. As $A$ is finite and $|A|\geq 2$, then $\Delta$ is self-simulable by Theorem~\ref{thm:selfsimulation}. We shall show that it is mediated. Observe that because $\Gamma$ is finitely-generated and $B$ is generating, we can find a finite generating set $T$ for $\Gamma$ with $T \subseteq\bigcup_{b \in B} \Gamma_b$. If $t \in T \cap \Gamma_b$ with $b \in B$, by weak $\mathcal{N}$-suborthogonality there is $\Sigma = \Gamma_c$ with $c \in \mathcal{N}$, such that $c \bot b$ ($[\Gamma_c,\Gamma_b]=1$) and $c \leq a$ ($\Gamma_c \leqslant \Gamma_a$) for some $a \in A$. Now, as $c \in \mathcal{N}$ implies $\Sigma$ is non-amenable, the condition $c \bot b$ implies that $t^{-1}\Sigma t = \Sigma$ (even $t^{-1}gt = g$ for all $g \in \Sigma$), and $c \leq a$ implies $\Sigma \leqslant \Gamma_a \leqslant \Delta$. This proves that $\Delta$ is mediated and thus $\Gamma$ is self-simulable.
\end{proof}

Notice that we can always take $|A| = 2$ by taking the join of all but one element of $A$, and this approach will be taken in most examples.

\begin{corollary}\label{cor:AutFn}
	For all $n \geq 5$, the groups $\Aut(F_n)$ and $\Out(F_n)$ are self-simulable.
\end{corollary}

\begin{proof}
	Consider a finite set $N = \{1,2,\dots,n\}$ that freely generates a free group $F_N$, and let $\Gamma$ be its automorphism group or outer automorphism group (these cases are almost identical). Let $\mathcal{S}$ be the power set $\mathcal{P}(N)$ of $N$ and define $\bot$ as disjointness and $\leq$ as containment. To $I \subseteq N$ associate the group $\Gamma_I$ which, in the case of $\Aut(F_N)$, is the natural copy of $\Aut(F_I)$ that fixes generators outside of $I$ (in the case of $\Out(F_N)$ it is the quotient of $\Aut(F_I)$ inside $\Gamma$). From this definition it is clear that whenever $I \bot I'$ then $\Gamma_I \cap \Gamma_{I'} = \{1_{F_N}\}$ and $[\Gamma_I,\Gamma_{I'}]=1$.
	
	Let $\mathcal{N} = \{I \subseteq N :|I| \geq 2\}$ and fix $I \in \mathcal{N}$. The group $\Gamma_I$ is non-amenable in both cases. Observe that $\Out(F_2)$ is isomorphic to $\operatorname{GL}_2(\ZZ)$ which is non-amenable, and thus $\Aut(F_2)$ is non-amenable as it admits a non-amenable quotient. Furthermore, the groups $\Gamma_I$ are finitely generated because for every $k \in \NN$, the group $\Aut(F_k)$ is generated by Nielsen transformations~\cite{Nielsen1924}, and thus a fortiori the quotient $\Out(F_k)$ is generated by them as well. For the same reason, the set $B = \{I \subseteq N : |I| = 2\}$ is generating (notice that the permutations of generators in the Nielsen transformations are generated by transpositions). Let $A = \{\{1,2\}, \{3,\dots,n\}\}$. It is clear that $A$ is orthogonal, therefore to satisfy the conditions of Lemma~\ref{lem:suborthogonal} it suffices to check that $B$ is weakly $\mathcal{N}$-suborthogonal to $A$. The argument for this is the same as in the proof of Lemma~\ref{lem:acts_on_product}, namely, either $b \in B$ is contained in $\{1,2\}$ or $\{3,\dots,n\}$ in which case we take $c = a$ as the set which does not intersect $b$, or $b$ intersects both of those sets in which case we take $a = \{3,\dots,n\}$ and $c = \{3,\dots,n\} \setminus b$, and note that $|c| \geq 2$ and thus $c \in \mathcal{N}$.
\end{proof}

Braid groups are a class of finitely presented groups whose elements can be represented visually as a set of intertwining strands. We will not give a proper definition of these groups here, readers interested in them can find a survey on braid groups in~\cite{Magnus1974} or~\cite{kassel2008braid}. We will only use the following facts about braid groups. The braid groups are all nested in the sense that for $m \leq n$ the group $B_m$ occurs as the subgroup of $B_n$ where exactly $m$ fixed contiguous strands can be intertwined, (2) The braid group $B_n$ is generated by the braids which pass the $i+1$-th braid over the $i$-th braid for $i = 1,\dots,n-1$, and (3) $B_n$ is non-amenable for $n \geq 3$. These facts are all elementary and can found in the previous two references.

\begin{corollary}\label{cor:braidgroups}
	Braid groups $B_n$ are self-simulable for $n \geq 7$.
\end{corollary}

\begin{proof}
	Consider a finite set $N = \{1,2,\dots,n\}$, let $\mathcal{S}$ be the set of discrete sub-intervals of $N$, and to each $I \in \mathcal{S}$ associate the braid group $B_I$ on these strands. We have the natural inclusion $I \leq J \implies B_I \leq B_J$. Take $\mathcal{N} = \{I \in \mathcal{S} : |I| \geq 3\}$. By (1) and (3) it follows that every $B_I$ is a non-amenable and finitely generated group. The set $B = \{I : |I| = 2\}$ is generating by (2), and the partition $A = \{\{1,2,3\}, \{4,\dots,n\}\}$ satisfies the assumptions of Lemma~\ref{lem:suborthogonal}, the proof being analogous as those of Corollaries~\ref{lem:acts_on_product} and~\ref{cor:AutFn}.
\end{proof}

\begin{remark}
	With this technique is also possible to obtain again that Brin-Thompson $2V$ is self-simulable. Indeed, 
	consider the natural faithful action of Brin-Thompson $2V$, which we denote by $\Gamma$, on $C \times C$ where $C = \{\symb{0},\symb{1}\}^\NN$ is the Cantor set. Let $\mathcal{S}$ be the set of all clopen sets in $C \times C$. Let $\bot$ be disjointness and $\leq$ be containment. For $I \in \mathcal{S}$ let $\Gamma_I$ be the subgroup of elements which act trivially on $(C\times C)\setminus I$. It follows that whenever $I \leq J$ then $\Gamma_I \leqslant \Gamma_J$ and, as the action is faithful, that whenever $I \bot J$, then $\Gamma_I$ and $\Gamma_J$ commute and have trivial intersection.
	
	Take $\mathcal{N} = \{[w] \times C : w \in \{\symb{0},\symb{1}\}^*\}$, and observe that by (2) $\Gamma_I \cong \Gamma$ for all $I \in \mathcal{N}$, so these groups are indeed finitely-generated and non-amenable. Proposition~3.2 of \cite{Br04} shows that for $\varepsilon = 1/4$ the set $B = \{I \times C : \mu(I) < \varepsilon\}$ is generating, where $\mu$ is the uniform Bernoulli measure on $C$. Taking $A = \{[0] \times C, [1] \times C\}$, we have that $A \subseteq\mathcal{N}$ is a finite orthogonal set, and $B$ is weakly $\mathcal{N}$-suborthogonal to $A$ because for any $I \times C \in B$, $([0] \times C) \setminus (I \times C)$ is a positive measure clopen set, thus nonempty, thus for some $w \in \{\symb{0},\symb{1}\}^*$ we have $([w] \times C) \subseteq([0] \times C)$ with $([w] \times C) \bot B$, concluding the proof of weak suborthogonality.
\end{remark}

\subsection{Right-angled Artin groups} Let $G = (V,E)$ be a finite simple graph (undirected edges and no loops). The \define{right-angled Artin group} (RAAG) defined by $G$ is group $\Gamma[G]$ generated by the vertices $V$ which satisfies that $u,v \in V$ commute whenever $\{u,v\}\in E$.\[ \Gamma[G] = \langle  V \ \mid\ [u,v]=1 \mbox{ for } \{u,v\}\in E  \rangle. \]

RAAGs provide a formalism for mapping finite simple graphs into finitely generated groups, and thus some properties of graphs can be translated into properties of the groups they generate. A recent survey on RAAGs can be found in~\cite{Charney2007}. In what follows we shall provide conditions on a graph which imply that their corresponding RAAG is self-simulable.

Let us fix some notation beforehand. Let $G=(V,E)$ be a graph. We denote by $\overline{G}$ the \define{complement} of $G$, that is, the graph with vertices $V$ and where $\{u,v\}$ is an edge of $\overline{G}$ if and only if $\{u,v\}$ is not an edge of $G$. We say that an edge $\{u,v\}\in E$ is \define{adjacent} to a node $w$ if either $\{u,w\}\in E$ or $\{v,w\}\in E$. Let us also denote, for $A,B \subseteq V$ by $E(A,B)$ the set of edges $\{u,v\}$ for which $u \in A$ and $v \in B$.

Below, we make our statements in terms of the complement of a graph, as they are easier to describe.

\begin{corollary}\label{cor:raag1}
	Suppose $G = (V,E)$ is a finite connected graph which has two edges with the property that no $v \in V$ is adjacent to both of them. Then $\Gamma[\overline{G}]$ is self-simulable.
\end{corollary}

\begin{proof}
	We shall employ Lemma~\ref{lem:suborthogonal}. Take $\mathcal{S}=\mathcal{P}(V)$ the set of subsets of vertices, let $a \bot b$ mean that $E(a,b)=\varnothing$ and interpret $\leq$ as inclusion. For $a\subseteq V$, let $\Gamma_a$ be the subgroup of $\Gamma[\overline{G}]$ generated by the vertices of $a$. 
	
	Notice that $(\Gamma_a)_{a \subseteq V}$ is compatible with $(\mathcal{S},\bot,\leq)$, as $a\bot b$ means that $[\Gamma_a,\Gamma_b]=1$ and that they have trivial intersection, and clearly $a\leq b$ implies that $\Gamma_a \leqslant \Gamma_b$.

	Take as $B$ the set of singleton subsets of $V$. Then $B$ is generating by definition. Take $\mathcal{N} = E$ and notice that they all give induced copies of the free group on two generators, therefore they are finitely generated and non-amenable. Finally, by assumption there are two edges $e=\{u_1,u_2\}$ and $e' = \{v_1,v_2\}$ such that no node is adjacent to both $e$ and $e'$. Take $A= \{e,e'\}$. Notice that this assumption implies that $A$ is orthogonal.
	
	Finally, we just need to check that $B$ is weakly $\mathcal{N}$-suborthogonal to $A$. if $b = \{v\}\in B$, then there is $e''\in \{e,e'\}$ which is not adjacent to $v$. Choosing $c = a = e''$ gives that $b\bot c \leq a$.
\end{proof}

\begin{example}
	Let $G$ be a finite connected simple graph with diameter at least five. Then by taking a shortest path of length $5$ between two vertices, it follows that the extremal edges satisfy the condition of Corollary~\ref{cor:raag1}. It follows that $\Gamma[\overline{G}]$ is self-simulable.\qee
\end{example}

\begin{example}
	Let $G$ be a cycle of length $n \geq 7$. Then $\Gamma[\overline{G}]$ is self-simulable. Indeed, if $n \geq 8$ it follows that $G$ satisfies the condition of Corollary~\ref{cor:raag1}. If $n = 7$, Let $(\mathcal{S},\bot,\leq)$ and $B$ be the set of singleton subsets of $V=\{1,\dots,7\}$ as in the proof of Corollary~\ref{cor:raag1}, take $\mathcal{N}$ as the set of vertices which induce connected subgraphs with at least two nodes (which are again, finitely generated and non-amenable), and pick $A  \{\{1,2\},\{4,5,6\}\}$. Again, it is clear that $A$ is orthogonal. 
	
	Now let $b \in B$, if $b = \{3\}$ pick $c = \{5,6\}$ and $a = \{4,5,6\}$. If $b = \{7\}$ pick $c = \{4,5\}$ and $a = \{4,5,6\}$. If $b \subseteq \{1,2\}$ pick $a = c = \{4,5,6\}$ and finally, if $b \subseteq \{4,5,6\}$ pick $a = c = \{1,2\}$. The result then follows from Lemma~\ref{lem:suborthogonal}.\qee
\end{example}

\section{Perspectives and questions}\label{sec:questions}

In this final section we present several topics which we either did not explore, or were unable to answer in this study. We hope some of these questions might motivate further studies about the class of self-simulable groups.

\subsection{Properties of the extension}

Our main theorem provides a subshift of finite type extension for any effectively closed action, but says very little about the properties of said extension. More precisely, one might wonder what kind of Borel invariant probability measures the extension can admit, or whether the extension can be made transitive or minimal provided the effectively closed action is transitive or minimal.

\begin{question}\label{Q:minimal}
	Let $\Gamma$ be a self-simulable group, and suppose the effectively closed action $\Gamma\curvearrowright X$ is minimal (resp.\ transitive). Is there a $\Gamma$-subshift of finite type extension which is minimal (resp.\ transitive)?
\end{question}

\begin{question}\label{Q:measure}
	Let $\Gamma$ be a self-simulable group and $\Gamma\curvearrowright X$ be an effectively closed action which admits a $\Gamma$-invariant Borel probability measure $\nu$. Is there a $\Gamma$-subshift of finite type extension $Z$ which admits a $\Gamma$-invariant Borel probability measure $\mu$? Can we have that $\Gamma \curvearrowright (X,\nu)$ is a measure-theoretic factor of $\Gamma \curvearrowright (Z,\mu)$ such that $\nu$ is the push-forward of $\mu$?
\end{question}

Regarding Question~\ref{Q:minimal}, a first step in a tentative adaptation of Theorem~\ref{thm:selfsimulation} would be to produce a minimal analogue of the paradoxical subshift. This can at least be done in the case when $\Gamma = F_2$ (or more generally, any finitely generated free group on at least two elements). Indeed, consider the alphabet $A$ given by the four tiles shown below. 
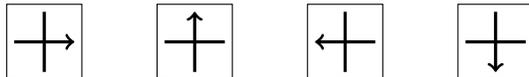
\[ \begin{tikzpicture}
	\begin{scope}[shift = {(0,0)}]
		\draw  (0,0) rectangle (1,1);
		\draw [very thick] (0.5,0.1) to (0.5,0.9);
		\draw [very thick, ->] (0.1,0.5) to (0.9,0.5);
	\end{scope}
	\begin{scope}[shift = {(2,0)}]
		\draw  (0,0) rectangle (1,1);
		\draw [very thick, ->] (0.5,0.1) to (0.5,0.9);
		\draw [very thick] (0.1,0.5) to (0.9,0.5);
	\end{scope}
	\begin{scope}[shift = {(4,0)}]
		\draw  (0,0) rectangle (1,1);
		\draw [very thick] (0.5,0.1) to (0.5,0.9);
		\draw [very thick, <-] (0.1,0.5) to (0.9,0.5);
	\end{scope}
	\begin{scope}[shift = {(6,0)}]
		\draw  (0,0) rectangle (1,1);
		\draw [very thick, <-] (0.5,0.1) to (0.5,0.9);
		\draw [very thick] (0.1,0.5) to (0.9,0.5);
	\end{scope}
\end{tikzpicture}   \]

We can define a subshift of finite type $\mathbf{M}$ on $F_2$ by imposing that along both generators of $F_2$ an arrowhead must be matched with an arrow tail. It is not hard to produce a collection of paradoxical path covers from any configuration in $\mathbf{M}$, and it can be shown that the shift action on $\mathbf{M}$ is minimal. However, adapting the rest of the proof seems hard. One important issue is that bi-infinite paths on any paradoxical subshift are unavoidable (by compactness) and therefore there is no evident way to restrict the contents of those paths in the computation layer. Furthermore, a second issue arises from the fact that the return times to a cylinder on the minimal system $F_2 \curvearrowright X$ we are simulating might coordinate with the return times of a pattern in $F_2 \curvearrowright \mathbf{M}$, thus introducing patterns in the subshift of finite type cover of $F_2 \curvearrowright X$ which do not occur with bounded gaps.

Concerning Question~\ref{Q:measure}, our methods are far from providing a useful approach, due to the fact that they strongly rely on constructions based on paradoxical decompositions, which thus admit no invariant measures.

\subsection{The class of self-simulable groups}

The rest of our questions are about the class of self-simulable groups. We are still quite far from having an idea of how to algebraically characterize this class of groups. 

The first candidate for a self-simulable group, which was studied by the first two named authors in collaboration with N. Aubrun, were the \define{surface groups}, that is, fundamental groups of closed orientable surfaces. A surface group is either virtually abelian or non-amenable, and all the non-amenable ones are commensurable. The non-amenable ones are strongly related to tilings of the hyperbolic plane, and a restricted variant of self-simulation has been proved in this setting~\cite{AubSab_2016_rowconstrained}. The methods of the present paper are completely different, and we are not aware of any method to apply them to surface groups.

\begin{question}
	Let $\Gamma$ be the fundamental group of a closed orientable surface of genus at least two. Is $\Gamma$ self-simulable? More generally, is every one-ended word-hyperbolic group self-simulable?
\end{question}

Another problem we tried to tackle and were unable to resolve is the following. By Corollary~\ref{cor:normalSS}, we have that if $\Gamma$ and $\Delta$ are finitely generated and recursively presented groups, and $\Gamma$ is self-simulable, then $\Gamma \times \Delta$ is self-simulable. We do not know whether the following restricted version of the converse holds. 

\begin{question}\label{Q:horriblyannyoing}
	Suppose $\Gamma \times \Delta$ is self-simulable and $\Delta$ is amenable, is $\Gamma$ self-simulable?
\end{question}

In fact, we do not know the answer to Question~\ref{Q:horriblyannyoing} even in the case when $\Delta=\ZZ$.

Another question is related to how much can Theorem~\ref{thm:mediaintroduction} be extended. We know that solely containing a self-simulable subgroup is not enough to be self-simulable (see Example~\ref{ex:trivialite}). However, our examples all rely on the few obstructions we constructed in Section~\ref{sec:counterexamples}. Therefore we are unable to answer the following question.

\begin{question}
	Let $\Gamma$ be a one-ended group which contains a self-simulable group. Is $\Gamma$ self-simulable?
\end{question}

Most of the theorems we proved in Sections~\ref{sec:stability_properties} and~\ref{sec:QIrigidity} used the hypothesis that $\Gamma$ was recursively presented. More importantly, the proof of quasi-isometric rigidity uses in a fundamental way that the groups involved are finitely presented. We do not know if this condition can be relaxed. 

\begin{question}
	Is the property of being self-simulable quasi-isometrically rigid for finitely generated recursively presented groups?
\end{question}

We showed that $\operatorname{GL}_n(\ZZ)$ is self-simulable when $n \geq 5$, and it is clearly not self-simulable when $n = 1$ (because it is finite) or when $n = 2$ (because it is virtually $\ZZ_2 * \ZZ_3$ and thus multi-ended).

\begin{question}
	Are the groups $\operatorname{GL}_3(\ZZ)$ and $\operatorname{GL}_4(\ZZ)$ self-simulable?
\end{question}

Similar questions can be posed for $\Aut(F_n)$ and $\Out(F_n)$ for $n\leq 4$ and $B_n$ for $n \leq 6$. In the same spirit, we also wonder what is exactly the class of RAAGs which is self-simulable.

\begin{question}
	For which finite simple graphs $G= (V,E)$ is the right-angled Artin group $\Gamma[G]$ self-simulable?
\end{question}

We showed that Thompson's $V$ is self-simulable. In light of Corollary~\ref{cor:FT}, the following question is interesting. 

\begin{question}
	Are Thompson's $F$ and $T$ self-simulable?
\end{question}

Notice that if one is able to show that either $F$ or $T$ are not self-simulable, it would imply that $F$ is amenable.

\Addresses

\bibliographystyle{abbrv}
\bibliography{ref}

\end{document}

%% file: img/villexample.tex
	\begin{tikzpicture}
	\def \c{0.4}
	\def \b{0.3}
	\def \a{0.2}
	\begin{scope}[scale = 0.75, shift={(-5.5,0)},rotate=0]

	\draw [->] (0,0) to (-5,0);
	\draw [->] (0,0) to (5,0);
	\draw [->] (0,0) to (-3,-4);
	\draw [->] (0,0) to (3,4);
	
	\node at (1.5,0.3) {\scalebox{0.7}{$a_1$}};
	\node at (0.7,1.6) {\scalebox{0.7}{$a_2$}};
	
	\draw [->] (-3,0) to (-4,-4/3);
	\draw [->] (-3,0) to (-2, 4/3);
	
	\draw [->] (3,0) to (2,-4/3);
	\draw [->] (3,0) to (4, 4/3);
	
	\draw [->] (-2,-8/3) to (-11/3,-8/3);
	\draw [->] (-2,-8/3) to (-1/3,-8/3);
	
	\draw [->] (2,8/3) to (11/3,8/3);
	\draw [->] (2,8/3) to (1/3,8/3);
	
	\draw [->] (0,0) to (0,5);
	\draw [->] (0,0) to (0,-5);

	\draw[fill = white] (0,0) circle (\c);
	\draw[fill = white] (0,1.5) circle (\c);
	\node at (0,1.5) {\scalebox{1.5}{$\symb{\star}$}};
	\draw[fill = white] (0,3) circle (\c);
	\node at (0,0) {\scalebox{1.5}{$\symb{\star}$}};
	\draw[fill = white] (0,-1.5) circle (\c);
	\node at (0,-1.5) {\scalebox{1.5}{$\symb{\star}$}};
	\draw[fill = white] (0,-3) circle (\c);
	\node at (0,-3) {\scalebox{1.5}{$\symb{\star}$}};
	\node at (0,3) {\scalebox{1.5}{$\symb{\star}$}};

	\draw[fill = white] (-3,0) circle (\b);
	\node at (-3,0) {\scalebox{1}{$\symb{1}$}};
	\draw[fill = white] (3,0) circle (\b);
	\node at (3,0){\scalebox{1}{$\symb{1}$}};
	\draw[fill = white] (2,8/3) circle (\b);
	\node at (2,8/3) {\scalebox{1}{$\symb{0}$}};
	\draw[fill = white] (-2,-8/3) circle (\b);
	\node at (-2,-8/3) {\scalebox{1}{$\symb{0}$}};

	\draw [->] (-4,0) to (-3.5,4/6);
	\draw [->] (-4,0) to (-4.5,-4/6);
	\draw[fill = white] (-4,0) circle (\a);
	\node at (-4,0) {\scalebox{0.6}{$\symb{0}$}};
	\draw [->] (-3.5,-2/3) to (-4.3,-2/3);
	\draw [->] (-3.5,-2/3) to (-2.7,-2/3);
	\draw[fill = white] (-3.5,-2/3) circle (\a);
	\node at (-3.5,-2/3) {\scalebox{0.6}{$\symb{0}$}};
	\draw [->] (-2.5,2/3) to (-3.3,2/3);
	\draw [->] (-2.5,2/3) to (-1.7,2/3);
	\draw[fill = white] (-2.5,2/3) circle (\a);
	\node at (-2.5,2/3) {\scalebox{0.6}{$\symb{1}$}};

	\draw [->] (4,0) to (3.5,-4/6);
	\draw [->] (4,0) to (4.5,4/6);
	\draw[fill = white] (4,0) circle (\a);
	\node at (4,0) {\scalebox{0.6}{$\symb{0}$}};
	\draw [->] (3.5,2/3) to (4.3,2/3);
	\draw [->] (3.5,2/3) to (2.7,2/3);
	\draw[fill = white] (3.5,2/3) circle (\a);
	\node at (3.5,2/3) {\scalebox{0.6}{$\symb{1}$}};
	\draw [->] (2.5,-2/3) to (3.3,-2/3);
	\draw [->] (2.5,-2/3) to (1.7,-2/3);
	\draw[fill = white] (2.5,-2/3) circle (\a);
	\node at (2.5,-2/3) {\scalebox{0.6}{$\symb{1}$}};

	\draw [->] (-2.5,-10/3) to (-3.3,-10/3);
	\draw [->] (-2.5,-10/3) to (-1.7,-10/3);
	\draw[fill =white] (-2.5,-10/3) circle (\a);
	\node at (-2.5,-10/3) {\scalebox{0.6}{$\symb{1}$}};
	\draw [->] (-3,-8/3) to (-3.5,-10/3);
	\draw [->] (-3,-8/3) to (-2.5,-6/3);
	\draw[fill = white] (-3,-8/3) circle (\a);
	\node at (-3,-8/3) {\scalebox{0.6}{$\symb{1}$}};
	\draw [->] (-1,-8/3) to (-1.5,-10/3);
	\draw [->] (-1,-8/3) to (-0.5,-6/3);
	\draw[fill = white] (-1,-8/3) circle (\a);
	\node at (-1,-8/3) {\scalebox{0.6}{$\symb{1}$}};

	\draw [->] (2.5,10/3) to (3.3,10/3);
	\draw [->] (2.5,10/3) to (1.7,10/3);
	\draw[fill = white] (2.5,10/3) circle (\a);
	\node at (2.5,10/3) {\scalebox{0.6}{$\symb{1}$}};
	\draw [->] (3,8/3) to (3.5,10/3);
	\draw [->] (3,8/3) to (2.5,6/3);
	\draw[fill = white] (3,8/3) circle (\a);
	\node at (3,8/3) {\scalebox{0.6}{$\symb{0}$}};
	\draw [->] (1,8/3) to (1.5,10/3);
	\draw [->] (1,8/3) to (0.5,6/3);
	\draw[fill = white] (1,8/3) circle (\a);
	\node at (1,8/3) {\scalebox{0.6}{$\symb{1}$}};
	\end{scope}

	\begin{scope}[scale = 0.75, shift={(5,0)},rotate=0]
	\draw [<->] (-4,0) to (4,0);
	\draw [<->] (-4,1.5) to (4,1.5);
	\draw [<->] (-4,3) to (4,3);
	\draw [<->] (-4,-1.5) to (4,-1.5);
	\draw [<->] (-4,-3) to (4,-3);
	\draw [<->] (0,-4) to (0,4);
	\draw [<->] (-1.5,-4) to (-1.5,4);
	\draw [<->] (1.5,-4) to (1.5,4);
	\draw [<->] (3,-4) to (3,4);
	\draw [<->] (-3,-4) to (-3,4);
	\node at (0.75,0.3) {\scalebox{0.7}{$a_1$}};

	\foreach \x in { -3, -1.5, 0, 1.5, 3 }{
		\foreach \y in { -3, -1.5, 0, 1.5, 3 }{
			\draw[fill = white] (\x,\y) circle (\c);	
		}
	}

	\node at (-3,3) {\scalebox{1.5}{$\symb{1}$}};
	\node at (-3,1.5) {\scalebox{1.5}{$\symb{0}$}};
	\node at (-3,0) {\scalebox{1.5}{$\symb{0}$}};
	\node at (-3,-1.5) {\scalebox{1.5}{$\symb{1}$}};
	\node at (-3,-3) {\scalebox{1.5}{$\symb{1}$}};
	
	\node at (-1.5,3) {\scalebox{1.5}{$\symb{0}$}};
	\node at (-1.5,1.5) {\scalebox{1.5}{$\symb{0}$}};
	\node at (-1.5,0) {\scalebox{1.5}{$\symb{1}$}};
	\node at (-1.5,-1.5) {\scalebox{1.5}{$\symb{1}$}};
	\node at (-1.5,-3) {\scalebox{1.5}{$\symb{0}$}};
	
	\node at (0,3) {\scalebox{1.5}{$\symb{\star}$}};
	\node at (0,1.5) {\scalebox{1.5}{$\symb{\star}$}};
	\node at (0,0) {\scalebox{1.5}{$\symb{\star}$}};
	\node at (0,-1.5) {\scalebox{1.5}{$\symb{\star}$}};
	\node at (0,-3) {\scalebox{1.5}{$\symb{\star}$}};
	
	\node at (1.5,3) {\scalebox{1.5}{$\symb{0}$}};
	\node at (1.5,1.5) {\scalebox{1.5}{$\symb{0}$}};
	\node at (1.5,0) {\scalebox{1.5}{$\symb{1}$}};
	\node at (1.5,-1.5) {\scalebox{1.5}{$\symb{1}$}};
	\node at (1.5,-3) {\scalebox{1.5}{$\symb{0}$}};
	
	\node at (3,3) {\scalebox{1.5}{$\symb{1}$}};
	\node at (3,1.5) {\scalebox{1.5}{$\symb{0}$}};
	\node at (3,0) {\scalebox{1.5}{$\symb{0}$}};
	\node at (3,-1.5) {\scalebox{1.5}{$\symb{1}$}};
	\node at (3,-3) {\scalebox{1.5}{$\symb{1}$}};
	
	\end{scope}
	
	\end{tikzpicture}
	

%% file: img/Wang_turingo_no_machina.tex
\begin{tikzpicture}[scale =1.8]

	\begin{scope}[shift = {(0,0)} ]
		\node at (0,-0.65) {\textbf{$\texttt{seed}$}};
		\clip (-0.5,-0.5) rectangle (+0.5,+0.5);
		\draw[black, fill = black!80] (-0.5,0.5)--(0,0)--(-0.5,-0.5)--cycle;
		\draw[black, fill=black!80] (-0.5,-0.5)--(0,0)--(+0.5,-0.5)--cycle;
		\draw[black, fill=white] (0.5,-0.5)--(0,0)--(+0.5,+0.5)--cycle;
		\draw[ fill = black] (0.25,0) -- (0.5,0.25) -- (0.5,-0.25) -- cycle;
		\draw[black, fill=black!50] (0.5,0.5)--(0,0)--(-0.5,+0.5)--cycle;
	\end{scope}
	
	\begin{scope}[shift = {(1.5,0)} ]
	\node at (0,-0.65) {\textbf{$\texttt{tile}_1(a)$}};
	\clip (-0.5,-0.5) rectangle (+0.5,+0.5);
	\draw[black, fill=white] (-0.5,0.5)--(0,0)--(-0.5,-0.5)--cycle;
	\draw[ fill = black] (-0.25,0) -- (-0.5,0.25) -- (-0.5,-0.25) -- cycle;
	\draw[black, fill=black!80] (-0.5,-0.5)--(0,0)--(+0.5,-0.5)--cycle;
	\draw[black, fill=black!50] (0.5,-0.5)--(0,0)--(+0.5,+0.5)--cycle;
	\draw[black, fill=black!15] (0.5,0.5)--(0,0)--(-0.5,+0.5)--cycle;
	\node at (0,0.35) {\textbf{$(q_0,a)$}};
	\end{scope}

	\begin{scope}[shift = {(3,0)} ]
	\node at (0,-0.65) {\textbf{$\texttt{tile}_2(a)$}};
	\clip (-0.5,-0.5) rectangle (+0.5,+0.5);
	\draw[black, fill=black!50] (-0.5,0.5)--(0,0)--(-0.5,-0.5)--cycle;
	\draw[black, fill=black!80] (-0.5,-0.5)--(0,0)--(+0.5,-0.5)--cycle;
	\draw[black, fill=black!50] (0.5,-0.5)--(0,0)--(+0.5,+0.5)--cycle;
	\draw[black, fill=white] (0.5,0.5)--(0,0)--(-0.5,+0.5)--cycle;
	\node at (0,0.35) {\textbf{$a$}};
	\end{scope}

	\begin{scope}[shift = {(4.5,0)} ]
		\clip (-0.5,-0.5) rectangle (+0.5,+0.5);
		\draw[black, fill = black!80] (-0.5,0.5)--(0,0)--(-0.5,-0.5)--cycle;
		\draw[black, fill=black!50] (-0.5,-0.5)--(0,0)--(+0.5,-0.5)--cycle;
		\draw[black] (0.5,-0.5)--(0,0)--(+0.5,+0.5)--cycle;
		\draw[black, fill=black!50] (0.5,0.5)--(0,0)--(-0.5,+0.5)--cycle;
	\end{scope}

	\begin{scope}[shift ={(0,-1.5)}]
	
		\begin{scope}[shift = {(-1.5,0)} ]
		\clip (-0.5,-0.5) rectangle (+0.5,+0.5);
		\draw[black, fill=white] (-0.5,0.5)--(0,0)--(-0.5,-0.5)--cycle;
		\draw[black, fill=white] (-0.5,-0.5)--(0,0)--(+0.5,-0.5)--cycle;
		\draw[black, fill=white] (0.5,-0.5)--(0,0)--(+0.5,+0.5)--cycle;
		\draw[black, fill=white] (0.5,0.5)--(0,0)--(-0.5,+0.5)--cycle;
		\node at (0,0.35) {\textbf{$a$}};
		\node at (0,-0.35) {\textbf{$a$}};
		\end{scope}
	
	\begin{scope}[shift = {(0,0)} ]
	\clip (-0.5,-0.5) rectangle (+0.5,+0.5);
	\draw[black, fill=white] (-0.5,0.5)--(0,0)--(-0.5,-0.5)--cycle;
	\draw[black, fill=black!15] (-0.5,-0.5)--(0,0)--(+0.5,-0.5)--cycle;
	\draw[black, fill=white] (0.5,-0.5)--(0,0)--(+0.5,+0.5)--cycle;
	\draw[black, fill=black!15] (0.5,0.5)--(0,0)--(-0.5,+0.5)--cycle;
	\node at (0,0.35) {\textbf{$(s',b')$}};
	\node at (0,-0.35) {\textbf{$(s,b)$}};
	\end{scope}
	
	\begin{scope}[shift = {(1.5,0)} ]
	\clip (-0.5,-0.5) rectangle (+0.5,+0.5);
	\draw[black, fill=black!15] (-0.5,0.5)--(0,0)--(-0.5,-0.5)--cycle;
	\draw[black, fill=black!15] (-0.5,-0.5)--(0,0)--(+0.5,-0.5)--cycle;
	\draw[black, fill=white] (0.5,-0.5)--(0,0)--(+0.5,+0.5)--cycle;
	\draw[black, fill=white] (0.5,0.5)--(0,0)--(-0.5,+0.5)--cycle;
	\node at (0,0.35) {\textbf{$c'$}};
	\node at (-0.35,0) {\textbf{$\stackrel{\ell'}{\leftarrow}$}};
	\node at (0,-0.35) {\textbf{$(\ell,c)$}};
	\end{scope}
	
	\begin{scope}[shift = {(3,0)} ]
	\clip (-0.5,-0.5) rectangle (+0.5,+0.5);
	\draw[black, fill=white] (-0.5,0.5)--(0,0)--(-0.5,-0.5)--cycle;
	\draw[black, fill=black!15] (-0.5,-0.5)--(0,0)--(+0.5,-0.5)--cycle;
	\draw[black, fill=black!15] (0.5,-0.5)--(0,0)--(+0.5,+0.5)--cycle;
	\draw[black, fill=white] (0.5,0.5)--(0,0)--(-0.5,+0.5)--cycle;
	\node at (0,0.35) {\textbf{$d'$}};
	\node at (0.35,0) {\textbf{$\stackrel{r'}{\rightarrow}$}};
	\node at (0,-0.35) {\textbf{$(r,d)$}};
	\end{scope}

		\begin{scope}[shift = {(4.5,0)} ]
		\clip (-0.5,-0.5) rectangle (+0.5,+0.5);
		\draw[black, fill=black!15] (-0.5,0.5)--(0,0)--(-0.5,-0.5)--cycle;
		\draw[black, fill=white] (-0.5,-0.5)--(0,0)--(+0.5,-0.5)--cycle;
		\draw[black, fill=white] (0.5,-0.5)--(0,0)--(+0.5,+0.5)--cycle;
		\draw[black, fill=black!15] (0.5,0.5)--(0,0)--(-0.5,+0.5)--cycle;
		\node at (0,0.35) {\textbf{$(q,a)$}};
		\node at (-0.35,0) {\textbf{$\stackrel{q}{\rightarrow}$}};
		\node at (0,-0.35) {\textbf{$a$}};
		\end{scope}
		
		\begin{scope}[shift = {(6,0)} ]
		\clip (-0.5,-0.5) rectangle (+0.5,+0.5);
		\draw[black, fill=white] (-0.5,0.5)--(0,0)--(-0.5,-0.5)--cycle;
		\draw[black, fill=white] (-0.5,-0.5)--(0,0)--(+0.5,-0.5)--cycle;
		\draw[black, fill=black!15] (0.5,-0.5)--(0,0)--(+0.5,+0.5)--cycle;
		\draw[black, fill=black!15] (0.5,0.5)--(0,0)--(-0.5,+0.5)--cycle;
		\node at (0,0.35) {\textbf{$(q,a)$}};
		\node at (0.35,0) {\textbf{$\stackrel{q}{\leftarrow}$}};
		\node at (0,-0.35) {\textbf{$a$}};
		\end{scope}

	\end{scope}
\end{tikzpicture}

%% file: img/Turingo_no_exampluru.tex
\begin{tikzpicture}[scale =1.2]

\newcommand{\seeding}[2]{
	\begin{scope}[shift = {(#1,#2)}]
		\clip (-0.5,-0.5) rectangle (+0.5,+0.5);
		\draw[black, fill = black!80] (-0.5,0.5)--(0,0)--(-0.5,-0.5)--cycle;
		\draw[black, fill=black!80] (-0.5,-0.5)--(0,0)--(+0.5,-0.5)--cycle;
		\draw[black, fill=white] (0.5,-0.5)--(0,0)--(+0.5,+0.5)--cycle;
		\draw[ fill = black] (0.25,0) -- (0.5,0.25) -- (0.5,-0.25) -- cycle;
		\draw[black, fill=black!50] (0.5,0.5)--(0,0)--(-0.5,+0.5)--cycle;
	\end{scope}	
}	

\newcommand{\vertical}[2]{
	\begin{scope}[shift = {(#1,#2)}]
		\clip (-0.5,-0.5) rectangle (+0.5,+0.5);
		\draw[black, fill = black!80] (-0.5,0.5)--(0,0)--(-0.5,-0.5)--cycle;
		\draw[black, fill=black!50] (-0.5,-0.5)--(0,0)--(+0.5,-0.5)--cycle;
		\draw[black] (0.5,-0.5)--(0,0)--(+0.5,+0.5)--cycle;
		\draw[black, fill=black!50] (0.5,0.5)--(0,0)--(-0.5,+0.5)--cycle;
	\end{scope}	
}	
	
\newcommand{\negrito}[2]{
	\begin{scope}[shift = {(#1,#2)}]
	\clip (-0.5,-0.5) rectangle (+0.5,+0.5);
	\draw[black, fill=black!80] (-0.5,0.5)--(0,0)--(-0.5,-0.5)--cycle;
	\draw[black, fill=black!80] (-0.5,-0.5)--(0,0)--(+0.5,-0.5)--cycle;
	\draw[black, fill=black!80] (0.5,-0.5)--(0,0)--(+0.5,+0.5)--cycle;
	\draw[black, fill=black!80] (0.5,0.5)--(0,0)--(-0.5,+0.5)--cycle;
	\end{scope}	
}

\newcommand{\negritoL}[2]{
	\begin{scope}[shift = {(#1,#2)}]
	\clip (-0.5,-0.5) rectangle (+0.5,+0.5);
	\draw[black, fill=black!80] (-0.5,0.5)--(0,0)--(-0.5,-0.5)--cycle;
	\draw[black, fill=black!80] (-0.5,-0.5)--(0,0)--(+0.5,-0.5)--cycle;
	\draw[black, fill=white] (0.5,-0.5)--(0,0)--(+0.5,+0.5)--cycle;
	\draw[black, fill=black!80] (0.5,0.5)--(0,0)--(-0.5,+0.5)--cycle;
	\end{scope}	
}

\newcommand{\base}[2]{
	\begin{scope}[shift = {(#1,#2)}]
	\clip (-0.5,-0.5) rectangle (+0.5,+0.5);
	\draw[black, fill=black!50] (-0.5,0.5)--(0,0)--(-0.5,-0.5)--cycle;
	\draw[black, fill=black!80] (-0.5,-0.5)--(0,0)--(+0.5,-0.5)--cycle;
	\draw[black, fill=black!50] (0.5,-0.5)--(0,0)--(+0.5,+0.5)--cycle;
	\draw[black, fill=white] (0.5,0.5)--(0,0)--(-0.5,+0.5)--cycle;
	\node at (0,0.35) {\textbf{$\sqcup$}};
	\end{scope}	
}

\newcommand{\basecero}[2]{
	\begin{scope}[shift = {(#1,#2)}]
	\clip (-0.5,-0.5) rectangle (+0.5,+0.5);
	\draw[black, fill=white] (-0.5,0.5)--(0,0)--(-0.5,-0.5)--cycle;
	\draw[ fill = black] (-0.25,0) -- (-0.5,0.25) -- (-0.5,-0.25) -- cycle;
	\draw[black, fill= black!80] (-0.5,-0.5)--(0,0)--(+0.5,-0.5)--cycle;
	\draw[black, fill=black!50] (0.5,-0.5)--(0,0)--(+0.5,+0.5)--cycle;
	\draw[black, fill=black!15] (0.5,0.5)--(0,0)--(-0.5,+0.5)--cycle;
	\node at (0,0.35) {\scalebox{0.8}{\textbf{$(a,\sqcup)$}}};
	\end{scope}	
}

\newcommand{\transmit}[3]{
	\begin{scope}[shift = {(#1,#2)}]
	\clip (-0.5,-0.5) rectangle (+0.5,+0.5);
	\draw[black, fill=white] (-0.5,0.5)--(0,0)--(-0.5,-0.5)--cycle;
	\draw[black, fill=white] (-0.5,-0.5)--(0,0)--(+0.5,-0.5)--cycle;
	\draw[black, fill=white] (0.5,-0.5)--(0,0)--(+0.5,+0.5)--cycle;
	\draw[black, fill=white] (0.5,0.5)--(0,0)--(-0.5,+0.5)--cycle;
	\node at (0,0.35) {\textbf{#3}};
	\node at (0,-0.35) {\textbf{#3}};
	\end{scope}	
}

\seeding{0}{1}
\basecero{1}{1}
\base{2}{1}
\base{3}{1}
\base{4}{1}
\base{5}{1}

\vertical{0}{2}
\begin{scope}[shift = {(1,2)} ]
\clip (-0.5,-0.5) rectangle (+0.5,+0.5);
\draw[black, fill=white] (-0.5,0.5)--(0,0)--(-0.5,-0.5)--cycle;
\draw[black, fill=black!15] (-0.5,-0.5)--(0,0)--(+0.5,-0.5)--cycle;
\draw[black, fill=black!15] (0.5,-0.5)--(0,0)--(+0.5,+0.5)--cycle;
\draw[black, fill=white] (0.5,0.5)--(0,0)--(-0.5,+0.5)--cycle;
\node at (0,0.35) {\textbf{$0$}};
\node at (0.35,0) {\textbf{$\stackrel{b}{\rightarrow}$}};
\node at (0,-0.35) {\scalebox{0.8}{\textbf{$(a,\sqcup)$}}};
\end{scope}
\begin{scope}[shift = {(2,2)} ]
\clip (-0.5,-0.5) rectangle (+0.5,+0.5);
\draw[black, fill=black!15] (-0.5,0.5)--(0,0)--(-0.5,-0.5)--cycle;
\draw[black, fill=white] (-0.5,-0.5)--(0,0)--(+0.5,-0.5)--cycle;
\draw[black, fill=white] (0.5,-0.5)--(0,0)--(+0.5,+0.5)--cycle;
\draw[black, fill=black!15] (0.5,0.5)--(0,0)--(-0.5,+0.5)--cycle;
\node at (0,0.35) {\scalebox{0.8}{\textbf{$(b,\sqcup)$}}};
\node at (-0.35,0) {\textbf{$\stackrel{b}{\rightarrow}$}};
\node at (0,-0.35) {\textbf{$\sqcup$}};
\end{scope}
\transmit{3}{2}{$\sqcup$}
\transmit{4}{2}{$\sqcup$}
\transmit{5}{2}{$\sqcup$}

\vertical{0}{3}
\begin{scope}[shift = {(1,3)} ]
\clip (-0.5,-0.5) rectangle (+0.5,+0.5);
\draw[black, fill=white] (-0.5,0.5)--(0,0)--(-0.5,-0.5)--cycle;
\draw[black, fill=white] (-0.5,-0.5)--(0,0)--(+0.5,-0.5)--cycle;
\draw[black, fill=black!15] (0.5,-0.5)--(0,0)--(+0.5,+0.5)--cycle;
\draw[black, fill=black!15] (0.5,0.5)--(0,0)--(-0.5,+0.5)--cycle;
\node at (0,0.35) {\scalebox{0.8}{\textbf{$(a,0)$}}};
\node at (0.35,0) {\textbf{$\stackrel{a}{\leftarrow}$}};
\node at (0,-0.35) {\textbf{$0$}};
\end{scope}
\begin{scope}[shift = {(2,3)} ]
\clip (-0.5,-0.5) rectangle (+0.5,+0.5);
\draw[black, fill=black!15] (-0.5,0.5)--(0,0)--(-0.5,-0.5)--cycle;
\draw[black, fill=black!15] (-0.5,-0.5)--(0,0)--(+0.5,-0.5)--cycle;
\draw[black, fill=white] (0.5,-0.5)--(0,0)--(+0.5,+0.5)--cycle;
\draw[black, fill=white] (0.5,0.5)--(0,0)--(-0.5,+0.5)--cycle;
\node at (0,0.35) {\textbf{$0$}};
\node at (-0.35,0) {\textbf{$\stackrel{a}{\leftarrow}$}};
\node at (0,-0.35) {\scalebox{0.8}{\textbf{$(b,\sqcup)$}}};
\end{scope}
\transmit{3}{3}{$\sqcup$}
\transmit{4}{3}{$\sqcup$}
\transmit{5}{3}{$\sqcup$}

\vertical{0}{4}
\begin{scope}[shift = {(1,4)} ]
\clip (-0.5,-0.5) rectangle (+0.5,+0.5);
\draw[black, fill=white] (-0.5,0.5)--(0,0)--(-0.5,-0.5)--cycle;
\draw[black, fill=black!15] (-0.5,-0.5)--(0,0)--(+0.5,-0.5)--cycle;
\draw[black, fill=black!15] (0.5,-0.5)--(0,0)--(+0.5,+0.5)--cycle;
\draw[black, fill=white] (0.5,0.5)--(0,0)--(-0.5,+0.5)--cycle;
\node at (0,0.35) {\textbf{$1$}};
\node at (0.35,0) {\textbf{$\stackrel{b}{\rightarrow}$}};
\node at (0,-0.35) {\scalebox{0.8}{\textbf{$(a,0)$}}};
\end{scope}
\begin{scope}[shift = {(2,4)} ]
\clip (-0.5,-0.5) rectangle (+0.5,+0.5);
\clip (-0.5,-0.5) rectangle (+0.5,+0.5);
\draw[black, fill=black!15] (-0.5,0.5)--(0,0)--(-0.5,-0.5)--cycle;
\draw[black, fill=white] (-0.5,-0.5)--(0,0)--(+0.5,-0.5)--cycle;
\draw[black, fill=white] (0.5,-0.5)--(0,0)--(+0.5,+0.5)--cycle;
\draw[black, fill=black!15] (0.5,0.5)--(0,0)--(-0.5,+0.5)--cycle;
\node at (0,0.35) {\scalebox{0.8}{\textbf{$(b,0)$}}};
\node at (-0.35,0) {\textbf{$\stackrel{b}{\rightarrow}$}};
\node at (0,-0.35) {\textbf{$0$}};
\end{scope}
\transmit{3}{4}{$\sqcup$}
\transmit{4}{4}{$\sqcup$}
\transmit{5}{4}{$\sqcup$}

\vertical{0}{5}
\transmit{1}{5}{$1$}
\begin{scope}[shift = {(2,5)} ]
\clip (-0.5,-0.5) rectangle (+0.5,+0.5);
\draw[black, fill=white] (-0.5,0.5)--(0,0)--(-0.5,-0.5)--cycle;
\draw[black, fill=black!15] (-0.5,-0.5)--(0,0)--(+0.5,-0.5)--cycle;
\draw[black, fill=white] (0.5,-0.5)--(0,0)--(+0.5,+0.5)--cycle;
\draw[black, fill=black!15] (0.5,0.5)--(0,0)--(-0.5,+0.5)--cycle;
\node at (0,0.35) {\scalebox{0.8}{\textbf{$(b,1)$}}};
\node at (0,-0.35) {\scalebox{0.8}{\textbf{$(b,0)$}}};
\end{scope}
\transmit{3}{5}{$\sqcup$}
\transmit{4}{5}{$\sqcup$}
\transmit{5}{5}{$\sqcup$}

\vertical{0}{6}
\transmit{1}{6}{$1$}
\begin{scope}[shift = {(2,6)} ]
\clip (-0.5,-0.5) rectangle (+0.5,+0.5);
\draw[black, fill=white] (-0.5,0.5)--(0,0)--(-0.5,-0.5)--cycle;
\draw[black, fill=black!15] (-0.5,-0.5)--(0,0)--(+0.5,-0.5)--cycle;
\draw[black, fill=black!15] (0.5,-0.5)--(0,0)--(+0.5,+0.5)--cycle;
\draw[black, fill=white] (0.5,0.5)--(0,0)--(-0.5,+0.5)--cycle;
\node at (0,0.35) {\textbf{$0$}};
\node at (0.35,0) {\textbf{$\stackrel{a}{\rightarrow}$}};
\node at (0,-0.35) {\scalebox{0.8}{\textbf{$(b,1)$}}};
\end{scope}
\begin{scope}[shift = {(3,6)} ]
\clip (-0.5,-0.5) rectangle (+0.5,+0.5);
\draw[black, fill=black!15] (-0.5,0.5)--(0,0)--(-0.5,-0.5)--cycle;
\draw[black, fill=white] (-0.5,-0.5)--(0,0)--(+0.5,-0.5)--cycle;
\draw[black, fill=white] (0.5,-0.5)--(0,0)--(+0.5,+0.5)--cycle;
\draw[black, fill=black!15] (0.5,0.5)--(0,0)--(-0.5,+0.5)--cycle;
\node at (0,0.35) {\scalebox{0.8}{\textbf{$(a,\sqcup)$}}};
\node at (-0.35,0) {\textbf{$\stackrel{a}{\rightarrow}$}};
\node at (0,-0.35) {\textbf{$\sqcup$}};
\end{scope}
\transmit{4}{6}{$\sqcup$}
\transmit{5}{6}{$\sqcup$}

\end{tikzpicture}

%% file: img/Turingo_no_examplurudos.tex
\begin{tikzpicture}[scale =1.5]
	\begin{scope}[shift = {(0,0)}]
		\clip (-0.5,-0.5) rectangle (+0.5,+0.5);
		\draw[black, fill = black!80] (-0.5,0.5)--(0,0)--(-0.5,-0.5)--cycle;
		\draw[black, fill=black!80] (-0.5,-0.5)--(0,0)--(+0.5,-0.5)--cycle;
		\draw[black, fill=white] (0.5,-0.5)--(0,0)--(+0.5,+0.5)--cycle;
		\draw[ fill = black] (0.25,0) -- (0.5,0.25) -- (0.5,-0.25) -- cycle;
		\draw[black, fill=black!50] (0.5,0.5)--(0,0)--(-0.5,+0.5)--cycle;
	\end{scope}
	\begin{scope}[shift = {(1,0)}]
		\clip (-0.5,-0.5) rectangle (+0.5,+0.5);
		\draw[black, fill=white] (-0.5,0.5)--(0,0)--(-0.5,-0.5)--cycle;
		\draw[ fill = black] (-0.25,0) -- (-0.5,0.25) -- (-0.5,-0.25) -- cycle;
		\draw[black, fill= black!80] (-0.5,-0.5)--(0,0)--(+0.5,-0.5)--cycle;
		\draw[black, fill=black!50] (0.5,-0.5)--(0,0)--(+0.5,+0.5)--cycle;
		\draw[black, fill=black!15] (0.5,0.5)--(0,0)--(-0.5,+0.5)--cycle;
		\node at (0,0.35) {\scalebox{0.8}{\textbf{$(y_0,\sqcup)$}}};
	\end{scope}	
	\begin{scope}[shift = {(2,0)}]
		\clip (-0.5,-0.5) rectangle (+0.5,+0.5);
		\draw[black, fill=black!50] (-0.5,0.5)--(0,0)--(-0.5,-0.5)--cycle;
		\draw[black, fill=black!80] (-0.5,-0.5)--(0,0)--(+0.5,-0.5)--cycle;
		\draw[black, fill=black!50] (0.5,-0.5)--(0,0)--(+0.5,+0.5)--cycle;
		\draw[black, fill=white] (0.5,0.5)--(0,0)--(-0.5,+0.5)--cycle;
		\node at (0,0.35) {\scalebox{0.8}{\textbf{$(y_1,\sqcup)$}}};
	\end{scope}	
	\begin{scope}[shift = {(3,0)}]
		\clip (-0.5,-0.5) rectangle (+0.5,+0.5);
		\draw[black, fill=black!50] (-0.5,0.5)--(0,0)--(-0.5,-0.5)--cycle;
		\draw[black, fill=black!80] (-0.5,-0.5)--(0,0)--(+0.5,-0.5)--cycle;
		\draw[black, fill=black!50] (0.5,-0.5)--(0,0)--(+0.5,+0.5)--cycle;
		\draw[black, fill=white] (0.5,0.5)--(0,0)--(-0.5,+0.5)--cycle;
		\node at (0,0.35) {\scalebox{0.8}{\textbf{$(y_2,\sqcup)$}}};
	\end{scope}	
	\begin{scope}[shift = {(4,0)}]
		\clip (-0.5,-0.5) rectangle (+0.5,+0.5);
		\draw[black, fill=black!50] (-0.5,0.5)--(0,0)--(-0.5,-0.5)--cycle;
		\draw[black, fill=black!80] (-0.5,-0.5)--(0,0)--(+0.5,-0.5)--cycle;
		\draw[black, fill=black!50] (0.5,-0.5)--(0,0)--(+0.5,+0.5)--cycle;
		\draw[black, fill=white] (0.5,0.5)--(0,0)--(-0.5,+0.5)--cycle;
		\node at (0,0.35) {\scalebox{0.8}{\textbf{$(y_3,\sqcup)$}}};
	\end{scope}	
	\begin{scope}[shift = {(5,0)}]
		\clip (-0.5,-0.5) rectangle (+0.5,+0.5);
		\draw[black, fill=black!50] (-0.5,0.5)--(0,0)--(-0.5,-0.5)--cycle;
		\draw[black, fill=black!80] (-0.5,-0.5)--(0,0)--(+0.5,-0.5)--cycle;
		\draw[black, fill=black!50] (0.5,-0.5)--(0,0)--(+0.5,+0.5)--cycle;
		\draw[black, fill=white] (0.5,0.5)--(0,0)--(-0.5,+0.5)--cycle;
		\node at (0,0.35) {\scalebox{0.8}{\textbf{$(y_4,\sqcup)$}}};
	\end{scope}	
	\begin{scope}[shift = {(6,0)}]
		\clip (-0.5,-0.5) rectangle (+0.5,+0.5);
		\draw[black, fill=black!50] (-0.5,0.5)--(0,0)--(-0.5,-0.5)--cycle;
		\draw[black, fill=black!80] (-0.5,-0.5)--(0,0)--(+0.5,-0.5)--cycle;
		\draw[black, fill=black!50] (0.5,-0.5)--(0,0)--(+0.5,+0.5)--cycle;
		\draw[black, fill=white] (0.5,0.5)--(0,0)--(-0.5,+0.5)--cycle;
		\node at (0,0.35) {\scalebox{0.8}{\textbf{$(y_5,\sqcup)$}}};
	\end{scope}	
	\begin{scope}[shift = {(7,0)}]
		\node at (0,0) {\scalebox{1.5}{$\cdots$}};
	\end{scope}
\end{tikzpicture}